\documentclass{article}
\usepackage{bm}
\usepackage{tikz,mathpazo}
\usepackage{pgf}
\usepackage[top=2.5cm, bottom=2.5cm, left=3cm, right=3cm]{geometry}   
\usepackage{indentfirst}
\usepackage{graphicx}
\usepackage{graphics}
\usepackage[toc,page,title,titletoc,header]{appendix}
\usepackage{bbm}
\usepackage{listings}  
\usepackage{amsmath}
\usepackage{setspace} 
\usepackage{caption}
\usepackage{multirow} 
\usepackage{lipsum,multicol}
\usepackage{pdfpages}
\usepackage{float}
\usepackage{amsthm}
\usepackage{url}
\usepackage{colortbl}
\usepackage{subfigure}
\usepackage{epsfig}
\usepackage{epstopdf}
\usepackage{fancyhdr}
\usepackage{abstract}
\usetikzlibrary{shapes.geometric, arrows}
\usepackage{amssymb}
\usepackage{latexsym}
\usepackage{verbatim}
\usepackage[numbers]{natbib} 
\usepackage{booktabs}
\usepackage{subcaption}
\usepackage{enumitem}
\usepackage{siunitx}

\usepackage{algorithm}  
\usepackage{algpseudocode}

\usepackage{hyperref}
\hypersetup{
    colorlinks=true,
    linkcolor=blue,
    filecolor=magenta,      
    urlcolor=cyan,
    citecolor = cyan,
}

\allowdisplaybreaks[4]

\newcommand{\inner}[1]{\left\langle #1 \right\rangle}
\newcommand{\norm}[1]{\left\Vert #1\right\Vert}
\newcommand{\bb}[1]{\mathbb{#1}}

\newcommand{\conv}[0]{\mathrm{conv}\,}

\newcommand{\X}{{ \ca{X} }}

\newcommand{\ca}[1]{\mathcal{#1}}

\newcommand{\Diag}[0]{\mathrm{Diag}}

\newcommand{\tp}{^\top}

\newcommand{\xk}{{x_{k} }}

\newcommand{\D}{D}

\newcommand{\Rnd}{\mathbb{R}^{n\times d}}
\newcommand{\Rn}{\mathbb{R}^n}

\newcommand{\Rm}{\mathbb{R}^m}

\newcommand{\Rmd}{\mathbb{R}^{m\times d}}

\newtheorem{theo}{Theorem}[section]
\newtheorem{lem}[theo]{Lemma}
\newtheorem{prop}[theo]{Proposition}
\newtheorem{examples}[theo]{Example}

\newtheorem{defin}[theo]{Definition}
\newtheorem{rmk}[theo]{Remark}
\newtheorem{assumpt}[theo]{Assumption}

\usepackage[figuresright]{rotating}
\usepackage{pdflscape}

\newcolumntype{C}[1]{>{\centering\arraybackslash}m{#1}}

\usepackage{longtable}

\numberwithin{equation}{section}
\newcommand{\manualeqtag}[2]{%
  \refstepcounter{equation}%
  \tag{#1}\label{#2}%
}

\title{Decentralized Stochastic Subgradient-type Methods with Communication Compression for Nonsmooth Nonconvex Optimization}
\author{Siyuan Zhang, ~ Nachuan Xiao,  ~ Xin Liu}

\begin{document}
\maketitle
\begin{abstract}
In this paper, we consider the nonsmooth nonconvex decentralized optimization problem, where inter-agent communication is compressed. We propose a general framework that unifies various decentralized stochastic subgradient-type methods with unbiased compression and contractive compression with error compensation. By relating the consensus-error iterates and the averaged iterates to the trajectories of continuous-time differential inclusions, we establish global convergence for all methods encompassed by our framework when the objective functions are nonsmooth and lack Clarke regularity. Based on our framework, we further develop several compression-based methods, including decentralized stochastic subgradient methods utilizing sign-based regularization and gradient-tracking momentum. Preliminary numerical experiments empirically support our theoretical results and highlight the communication-accuracy trade-off of the newly developed methods. 
\end{abstract}

\section{Introduction}\label{sec:intro}

In this paper, we consider the following decentralized optimization problem (DOP) over an undirected connected network $\mathtt{G} = (\mathtt{V}, \mathtt{E})$, 
\begin{equation*}
\manualeqtag{DOP}{Prob_DOP}
\begin{aligned}
\min_{{\bm x}_1, {\bm x}_2, \ldots, {\bm x}_d \in \Rn} & \quad   \sum_{i=1}^d \mathbb{E}_{\zeta_{i} \sim \mathcal{P}_{i}} F_{i}({\bm x}_i, \zeta_{i}), \\
\text{s.t.} & \quad  {\bm x}_{i}={\bm x}_{j}, \quad \forall(i, j) \in \mathtt{E} .
\end{aligned}
\end{equation*}
Here, the node set $\mathtt{V}= \{1, \ldots, d\}$ stands for the collection of agents, and the edge set $\mathtt{E}$ represents the communication links between agents. Each agent $i$ is associated with a local cost function $F_{i}(\cdot, \zeta_{i})$ and a local data distribution $\mathcal{P}_{i}$. The expectation-form cost function is denoted by $f_i({\bm x}) := \mathbb{E}_{\zeta_{i} \sim \mathcal{P}_{i}} F_{i}({\bm x}, \zeta_{i}),$
which is assumed  to be locally Lipschitz continuous, nonconvex and possibly nonsmooth. 


The optimization problem \eqref{Prob_DOP} has wide applications in wireless edge computing \cite{deng2020edge, cicconetti2020decentralized, meng2021learning}, multi-agent target seeking \cite{suzuki1999distributed, pu2016noise}, privacy-preserving systems \cite{kasyap2021privacy, bonawitz2021federated} and distributed learning \cite{zhang2021multi, gao2021consensus}. By eliminating the need for a central server, decentralized optimization effectively mitigates inherent issues in traditional centralized optimization, such as server failures, privacy leakage, and communication bottlenecks. However, high-dimensional variables and limited bandwidth resources necessitate the development of more communication-efficient methods. One way to alleviate the
communication overhead is communication compression, which transmits compressed messages between the agents using compression operators. Mainstream compression operators include quantization operators \cite{alistarh2017qsgd, wen2017terngrad, zhang2017zipml, horvoth2022natural} and sparsification operators \cite{alistarh2018convergence, stich2018sparsified, richtarik2021ef21, fatkhullin2021ef21, zheng2019communication, qian2021error, beznosikov2023biased}, both of which were originally developed in distributed optimization with a central server to compress high-dimensional gradients. 

Nowadays, communication compression has been widely used in decentralized settings. One line of research employs randomized unbiased compression operators to transmit inexact information directly, including QDGD \cite{reisizadeh2019exact}, S-NEAR-DGD \cite{iakovidou2022s}, QDSG \cite{li2017distributed, doan2020convergence} and IC-GT \cite{shah2023stochastic}. These methods typically assume that the compression operators are variance-bounded—either absolutely or proportionally to the squared norm of the input variables—and that the objective functions are (strongly) convex and smooth. Under such assumptions, QDGD \cite{reisizadeh2019exact} proves that the expected deviation from the optimal solution gradually vanishes. S-NEAR-DGD \cite{iakovidou2022s} and IC-GT \cite{shah2023stochastic} show that, with fixed step-sizes, the iterates converge linearly in expectation to a neighborhood of the optimum. QDSG \cite{li2017distributed, doan2020convergence} establishes that the iterative sequence converges almost surely to a solution under properly chosen diminishing step-sizes and consensus parameters. In addition, to mitigate the imprecision introduced by compression, DCD-SGD and ECD-SGD \cite{tang2018communication} employ difference compression and extrapolation compression techniques. A concise comparison of representative decentralized stochastic gradient-type methods with unbiased compression is provided in Table \ref{tab:intro_unb_comp}. 

Another line of work \cite{koloskova2019decentralized, singh2021squarm, zhao2022beer, liao2022compressed, yau2022docom} focuses on employing randomized contractive compression operators in conjunction with (implicit) error-compensation mechanisms to achieve efficient communication while maintaining convergence.  
CHOCO-SGD \cite{koloskova2019decentralized2} is a decentralized stochastic gradient descent method that combines difference-based contractive compression with error compensation. The authors show that CHOCO-SGD attains a convergence rate close to that of the centralized counterpart even under arbitrarily high compression ratios, for objectives that are either strongly convex \cite{koloskova2019decentralized2} or nonconvex but Lipschitz smooth \cite{koloskova2019decentralized}. Moreover, SPARQ-SGD \cite{singh2022sparq} is an event-triggered variant of CHOCO-SGD that reduces the number of communication rounds. SQuARM-SGD \cite{singh2021squarm} can be interpreted as CHOCO-SGD with local-update Nesterov momentum, which guarantees ergodic convergence in expectation for smooth objective functions. C-GT \cite{liao2022compressed} and BEER \cite{zhao2022beer} incorporate gradient-tracking into decentralized SGD with difference-based contractive compression to handle heterogeneity across multi-agents, and \cite{liu2025compresseddecentralizedmomentumstochastic} adopts a compression technique together with a momentum-based adaptive learning rate to accelerate empirical convergence. DoCoM \cite{yau2022docom} and MoTEF \cite{IslamovStichNear} integrate communication compression with momentum tracking and error feedback, showing improved theoretical and empirical performance under arbitrary data heterogeneity. DEF-ATC \cite{nassif2025differential} studies decentralized learning with bounded-distortion compression operators, a class that subsumes contractive compression operators as a special case. It also generalizes the existing difference-based compression scheme with error feedback. A concise comparison of representative decentralized stochastic gradient-type methods with contractive compression is summarized in Table \ref{tab:intro_con_comp}.
More recently, \cite{liao2024robust} develops a robust compressed push-pull method (RCPP) for smooth nonconvex objectives over general directed graphs by combining gradient tracking with compression. To incorporate privacy constraints into communication-efficient decentralized learning over directed graphs, \cite{zhu2025dp} proposes DP-CSGP, which couples compressed stochastic gradient push with node-level $(\varepsilon,\delta)$-differential privacy guarantees and provides utility bounds for general smooth nonconvex objectives.

\begin{table}[t]
    \centering
    \fontsize{7}{12}\selectfont
    \resizebox{\textwidth}{!}{
    \begin{tabular}{|C{2.5cm}|C{1.3cm}|C{4.2cm}|C{3.0cm}|C{4.2cm}|}
    \hline
    {Method} & {\shortstack[c]{Update\\scheme}} & {Step-sizes} & {Conditions on $f_i$} & {\shortstack[c]{Conditions on\\compressed operator}} \\
    \hline
    \shortstack[c]{QDGD\\\cite{reisizadeh2019exact}}
    & GD
    & \shortstack[c]{Two diminishing scales\\$\alpha=\Theta(K^{-(1-\delta)/4})$,\\$\varepsilon=\Theta(K^{-3(1-\delta)/4})$}
    & Smooth, strongly convex
    & \shortstack[c]{Unbiased quantizer;\\absolute or relative\\second-moment bound}
    \\
    \hline
    \shortstack[c]{S-NEAR-DGD\\\cite{iakovidou2022s}}
    & GD
    & Constant
    & \shortstack[c]{Strongly convex,\\Lipschitz gradients}
    & \shortstack[c]{Inexact communication model\\with random distortion}
    \\
    \hline
    \shortstack[c]{QDSG\\\cite{li2017distributed,doan2020convergence}}
    & GD
    & \shortstack[c]{Diminishing\\(subgradient-type)}
    & \shortstack[c]{Convex / strongly convex,\\possibly nonsmooth}
    & \shortstack[c]{Random/adaptive quantization;\\resolution-dependent\\quantization error}
    \\
    \hline
    \shortstack[c]{IC-GT\\\cite{shah2023stochastic}}
    & GT
    & Constant
    & Smooth, strongly convex
    & \shortstack[c]{Inexact communication noise;\\probabilistic quantization or\\additive channel noise}
    \\
    \hline
    Our work \eqref{Eq_Framework}
    & \shortstack[c]{GD/GT\\GD-M}
    & \shortstack[c]{$\eta_k=o(1/\log k)$,\\three-timescale $\{\eta_k,\theta_k,\gamma_k\}$}
    & \shortstack[c]{Path-differentiable,\\coercive}
    & \shortstack[c]{Unbiased compression operator\\(Assumption~\ref{Assumption_compress0})}
    \\
    \hline
    \end{tabular}}
    \caption{A brief comparison of decentralized stochastic gradient-type methods with unbiased compression. Here, ``GD'', ``GT'', and ``GD-M'' are abbreviations of ``gradient descent'', ``gradient tracking'', and ``gradient descent with momentum'', respectively.}
    \label{tab:intro_unb_comp}
\end{table}

\begin{table}[t]
    \centering
    \fontsize{7}{12}\selectfont
    \resizebox{\textwidth}{!}{
    \begin{tabular}{|C{2.5cm}|C{1.3cm}|C{4.2cm}|C{3.0cm}|C{4.2cm}|}
    \hline
    {Method} & {\shortstack[c]{Update\\scheme}} & {Step-sizes} & {Conditions on $f_i$} & {\shortstack[c]{Conditions on\\compressed operator}} \\
    \hline
    \shortstack[c]{CHOCO-SGD\\\cite{koloskova2019decentralized2,koloskova2019decentralized}}
    & GD
    & \shortstack[c]{Strongly convex: $\eta_k=\frac{4}{\mu(a+k)}$;\\smooth nonconvex: fixed $\eta=\Theta(\sqrt{n/K})$}
    & \shortstack[c]{Smooth strongly convex,\\or smooth nonconvex}
    & \shortstack[c]{Contractive compressor\\applied to model differences}
    \\
    \hline
    \shortstack[c]{SPARQ-SGD\\\cite{singh2022sparq}}
    & GD
    & \shortstack[c]{Strongly convex: $\eta_k=\frac{8}{\mu(a+k)}$;\\smooth nonconvex: fixed $\eta=\sqrt{n/K}$}
    & \shortstack[c]{Smooth strongly convex,\\or smooth nonconvex}
    & \shortstack[c]{Contractive compressor with\\event-triggered communication}
    \\
    \hline
    \shortstack[c]{SQuARM-SGD\\\cite{singh2021squarm}}
    & GD-M
    & \shortstack[c]{Constant / standard SGD-type\\schedule with local momentum}
    & Smooth nonconvex
    & \shortstack[c]{Contractive compressor\\with error compensation}
    \\
    \hline
    \shortstack[c]{C-GT\\\cite{liao2022compressed}}
    & GT
    & Constant
    & Smooth strongly convex
    & \shortstack[c]{Contractive compressor in\\gradient-tracking communication}
    \\
    \hline
    \shortstack[c]{BEER\\\cite{zhao2022beer}}
    & GT
    & Constant
    & \shortstack[c]{Smooth nonconvex,\\arbitrary heterogeneity}
    & \shortstack[c]{Contractive compressor on\\model and tracking differences}
    \\
    \hline
    \shortstack[c]{DoCoM\\\cite{yau2022docom}}
    & GT-M
    & Constant
    & \shortstack[c]{Smooth nonconvex;\\PL for linear convergence}
    & \shortstack[c]{Contractive compressor with\\momentum tracking and EF}
    \\
    \hline
    \shortstack[c]{MoTEF\\\cite{IslamovStichNear}}
    & GT-M
    & Constant
    & \shortstack[c]{Smooth nonconvex,\\objective bounded below}
    & \shortstack[c]{Contractive compressor with\\momentum tracking and EF}
    \\
    \hline
    Our work \eqref{Eq_Framework}
    & \shortstack[c]{GD/GT\\GD-M}
    & \shortstack[c]{$\eta_k=o(1/\log k)$,\\three-timescale $\{\eta_k,\theta_k,\gamma_k\}$}
    & \shortstack[c]{Path-differentiable,\\coercive}
    & \shortstack[c]{Contractive compression operator\\(Assumption~\ref{Assumption_compress})}
    \\
    \hline
    \end{tabular}}
    \caption{A brief comparison of decentralized stochastic gradient-type methods with contractive compression. Here, ``GD'', ``GT'', ``GD-M'', and ``GT-M'' are abbreviations of ``gradient descent'', ``gradient tracking'', ``gradient descent with momentum'', and ``gradient tracking with momentum'', respectively.}
    \label{tab:intro_con_comp}
\end{table}

Decentralized optimization methods with communication compression are intrinsically well-suited for the distributed training of neural networks, as billions of parameters need to be transmitted across the agents. In modern deep neural network architectures, widely adopted nonsmooth activation functions, such as ReLU and leaky ReLU, have become essential building blocks. The use of these activations gives rise to nonsmooth loss functions that lack Clarke regularity \cite{Clarke1998NonsmoothAA} (hereafter referred to as \textit{non‑Clarke‑regular} functions). However, the existing convergence analyses in the above-mentioned works predominantly assume that each $f_i$ is convex or smooth nonconvex (at least differentiable), thus precluding various important
 applications in the distributed training of nonsmooth neural networks. A limited number of works \cite{li2021decentralized, yan2023compressed} study the convergence of proximal-type compression-based algorithms for decentralized composite optimization problems. Nevertheless, objective functions of these problems are still weakly convex, and hence Clarke-regular, which is essentially distinct from the nonsmoothness arising in non–Clarke-regular loss functions. Consequently, whether we can establish convergence theories for decentralized compression-based methods in nonconvex nonsmooth optimization, especially in the training of nonsmooth neural networks, is a question worth exploring and contemplating.

When training nonsmooth neural networks, the subgradients of the loss function are always computed using the automatic differentiation (AD) algorithm, which is widely adopted in various popular machine learning packages, such as TensorFlow, PyTorch, and JAX. Utilizing the chain rule, the AD algorithm constructs generalized subgradients through the composition of Jacobians of each network block. However,
as the chain rule fails for non‑Clarke‑regular functions (see examples in \cite{bolte2020mathematical}), the outputs of the AD algorithm may not belong to the Clarke subdifferential \cite{Clarke1998NonsmoothAA} of such a loss function. To tackle this issue, \cite{bolte2021conservative} introduces the \textit{conservative field}, a generalization of the Clarke subdifferential applicable to a broad class of functions referred to as \textit{path-differentiable functions}. Path-differentiable functions are sufficiently general to encompass a wide range of objective functions in real-world applications, particularly the loss functions of nonsmooth neural networks. More importantly, the conservative field admits chain rules for path-differentiable functions and thus contains the output of the AD algorithm. Based on the concept of conservative field, several recent works \cite{bolte2021conservative, davis2020stochastic, xiao2023adam, le2024nonsmooth, pauwels2021incremental, xiao2023convergence,  zhang2024decentralized} leverage the ordinary differential equation (ODE) approach \cite{benaim2005stochastic, benaim2006dynamics, borkar2008stochastic, duchi2018stochastic} to study the behavior of stochastic subgradient-based methods for non-Clarke-regular functions. However, extending these results to multi-agent settings with communication compression is nontrivial, particularly in the consensus analysis and the construction of appropriate differential inclusions.





\subsection{A general framework for decentralized stochastic subgradient-type methods with communication compression}

In this paper, we consider a general framework for \underline{d}\underline{e}centralized \underline{s}tochastic subgradient-type methods with communication \underline{c}ompression (DESC) in nonsmooth optimization,
\begin{equation*}
\manualeqtag{DESC}{Eq_Framework}
 {\bm Z}_{k+1} = \underbrace{(1-\theta_k) {\bm Z}_k + \theta_k {\bm Z}_k {\bm W} + \theta_{k} {\bm E}_{k+1}({\bm W}-\Sigma)}_{\text{Local aggregation with compression}} \underbrace{ - \eta_k ({\bm H}_k +  \Xi_{k+1})}_{\text{descent step}}.
\end{equation*}
Here, ${\bm Z}_k = [{\bm z}_{1,k}, \ldots, {\bm z}_{d,k}] \in \Rmd$ denotes the collection of local variables (including the local variables $\{{\bm x}_{i, k}\}_{i\in [d]}$ and auxiliary variables corresponding to specific subgradient methods), ${\bm W}$ is a mixing matrix, ${\bm H}_k  \in \Rmd$  represents the collection of noiseless update directions of all agents, and $\Sigma$ is a diagonal matrix which conforms to the mechanism of communication compression.  $\{{\bm E}_{k+1}\}$ and $\{\Xi_{k+1}\}$ are two sequences of random variables defined on probability space $(\Omega, \mathcal{F}, \mathbb{P})$, which correspond to the compression error and  evaluation noise on ${\bm H}_{k}$, respectively.  Furthermore, $\{\eta_k\}$ and $\{\theta_{k} \}$  are two-timescale step-sizes corresponding to the descent step and the local average. Typically, they are required to satisfy $\eta_{k}/\theta_{k}\to 0$, as $k$ goes infinity. 

In \eqref{Eq_Framework}, the $i$-th term of ${\bm Z}_k {\bm W} +  {\bm E}_{k+1}({\bm W}-\Sigma)$ can be interpreted as the weighted aggregation performed by agent $i$ on the variables transmitted by its neighbors, which are perturbed by compression errors. Further, a $\theta_k$-reallocation step together with a descent step is incorporated to produce the next iterate ${\bm Z}_{k+1}$. 

The flexibility in choosing ${\bm E}_{k+1}$ and ${\bm H}_{k}$ allows this framework to encompass a wide range of decentralized methods with communication compression. Throughout this paper, we focus on two important sub-frameworks developed from \eqref{Eq_Framework} by choosing different ${\bm E}_{k+1}$ and corresponding $\Sigma$. We first introduce the following sub-framework for developing decentralized stochastic subgradient-type methods with
unbiased compression,
\begin{equation*}
 \manualeqtag{DESC-Unb}{eq:uncom}
\left\{
\begin{aligned}
{\bm Z}_{k+1} & = (1-\theta_k) {\bm Z}_k + \theta_k {\bm Z}_k {\bm W} + \theta_{k} {\bm E}_{k+1}({\bm W}-\Diag ({\bm W})) - \eta_k ({\bm H}_k + \Xi_{k+1}), \\
{\bm E}_{k+1} & = C({\bm Z}_{k}, {\bm \omega}_{k+1})-{\bm Z}_{k} := [C({\bm z}_{1,k}, {\bm \omega}_{k+1}) - {\bm z}_{1,k}, \ldots, C({\bm z}_{d,k}, {\bm \omega}_{k+1})-{\bm z}_{d,k}].\\
\end{aligned}
\right.
\end{equation*}
Here, $C(\cdot, {\bm \omega}): \mathbb{R}^m \times \Omega \to \mathbb{R}^m$ is an unbiased compression operator, i.e. $\mathbb{E}_{{\bm \omega}}[C({\bm x}, {\bm \omega})]={\bm x}$. $\{{\bm E}_{k+1}\}$ is a martingale difference sequence with respect to $\mathcal{F}_{k}= \sigma(\{{\bm Z}_{i}: i\leq k\})$. As shown later in Section \ref{sec:methods}, with specific choices of ${\bm H}_{k}$, \eqref{eq:uncom} yields variants of several decentralized stochastic subgradient-type methods with unbiased compression. For instance, as we set ${\bm Z}_{k}={\bm X}_{k}$, ${\bm H}_{k}+ \Xi_{k+1} \in [D_{F_1(\cdot,\zeta_{i,k+1})}({\bm x}_{i,k}), \ldots,$  $ D_{F_d(\cdot,\zeta_{i,k+1})}({\bm x}_{i,k})]$, and $C$ to be a randomized quantization operator $Q$, \eqref{eq:uncom} reduces exactly to the nonsmooth version of QDGD \cite{reisizadeh2019exact}, a decentralized SGD method with quantization:
\begin{equation*}
\manualeqtag{QDGD+}{unb:dsgd0}
{\bm x}_{i,k+1} \in  (1-\theta_k +\theta_k {\bm W}(i,i)){\bm x}_{i,k} + \theta_k \sum_{j\in \mathcal{N}_i}{\bm W}(i,j) Q({\bm x}_{j,k}, {\bm \omega}_{k+1}) - \eta_k D_{F_i(\cdot,\zeta_{i,k+1})}({\bm x}_{i,k}),    
\end{equation*}
where $\mathcal{N}_{i}$ denotes the neighbors of agent $i$, $D_{F_i(\cdot,\zeta_{i,k+1})}$ is a conservative field for function $F_i(\cdot,\zeta_{i,k+1})$ that describes how $F_i(\cdot,\zeta_{i,k+1})$ is differentiated.

Additionally, based on \eqref{Eq_Framework}, we introduce another sub-framework for contractive compression,

\begin{equation*}
\manualeqtag{DESC-Con}{eq:contractive}
\left\{
\begin{aligned}
{\bm Z}_{k+1} & = (1-\theta_k) {\bm Z}_k + \theta_k {\bm Z}_k {\bm W} + \theta_{k} {\bm E}_{k+1}({\bm W}-{\bm I}_{d}) - \eta_k ({\bm H}_k + \Xi_{k+1}),\\ 
{\bm E}_{k+1} & = \hat{\bm Z}_{k+1} -{\bm Z}_k,\\
\hat{\bm Z}_{k+1} & =\hat{{\bm Z}}_{k} + \gamma_k C({\bm Z}_{k}-\hat{{\bm Z}}_{k}, {\bm \omega}_{k+1}).\\
\end{aligned}
\right.
\end{equation*}
Here, $C(\cdot, {\bm \omega}): \mathbb{R}^m \times \Omega \to \mathbb{R}^m$ is a contractive compression operator, i.e. $\mathbb{E}_{{\bm \omega}}[\|C(\bm{x}, {\bm \omega})-\bm{x}\|^2] \leq(1-\alpha)\|\bm{x}\|^2, \alpha \in (0,1].$ $\hat{\bm Z}_{k+1}= [\hat{\bm z}_{1,k+1}, \ldots, \hat{\bm z}_{d,k+1}]$ is the collection of local copies, where each $\hat{\bm z}_{i,k+1}$ is an inexact copy of ${\bm z}_{i, k}$ held by the neighbors of agent $i$. $\{\gamma_k\}$ is a diminishing sequence, and $-{\bm E}_{k+1}$ is essentially an compression error term that captures the discrepancy between the desired difference ${\bm Z}_{k}- \hat{\bm Z}_k$ and the compressed term  $\gamma_k C({\bm Z}_{k}-\hat{{\bm Z}}_{k}, {\bm \omega}_{k+1})$. The update formula of ${\bm E}_{k}$ and $\hat{\bm Z}_{k}$ can be further illustrated by the following error-compensation mechanism
\begin{equation*}
\left\{
\begin{aligned}
\Delta_k ~& = {\bm Z}_{k}-{\bm Z}_{k-1}+(-{\bm E}_{k}),  \hspace{2.7cm} \triangleleft \text{ error-compensation}     \\
\hat{\bm Z}_{k+1} ~& =\hat{{\bm Z}}_{k} + \gamma_k C(\Delta_k, {\bm \omega}_{k+1}), \hspace{2.72cm} \triangleleft \text{ compression and update} \\
-{\bm E}_{k+1} ~& =  \Delta_k - \gamma_k C(\Delta_k, {\bm \omega}_{k+1}).  \hspace{2.74cm} \triangleleft \text{ compression error}  \\  \end{aligned}
\right.
\end{equation*}
Similarly, as demonstrated in Section \ref{sec:methods}, with different choices of ${\bm H}_{k}$, \eqref{eq:contractive} corresponds to various decentralized stochastic subgradient-type methods with contractive compression. For instance, when we choose ${\bm Z}_{k}={\bm X}_{k}$, ${\bm H}_{k}+ \Xi_{k+1} \in [D_{F_1(\cdot,\zeta_{i,k+1})}({\bm x}_{i,k}), \ldots, D_{F_d(\cdot,\zeta_{i,k+1})}({\bm x}_{i,k})]$,  \eqref{eq:contractive} is a nonsmooth version of CHOCO-SGD \cite{koloskova2019decentralized, koloskova2019decentralized2}:
\begin{equation*}
\manualeqtag{CHOCO-SGD+}{con:dsgd0}
\left\{
\begin{aligned}
\hat{{\bm x}}_{j, k+1} & = \hat{{\bm x}}_{j, k} + \gamma_k C({\bm x}_{j, k}-\hat{{\bm x}}_{j, k}, {\bm \omega}_{k+1}), \quad j \in \mathcal{N}_{i}, \\
{\bm x}_{i, k+1} & \in  {\bm x}_{i,k} + \theta_{k} \sum_{j\in \mathcal{N}_{i}} {\bm W}(i,j) (\hat{\bm x}_{j, k+1} - \hat{\bm x}_{i, k+1}) -\eta_{k}D_{F_i(\cdot,\zeta_{i,k+1})}({\bm x}_{i,k}).\\
\end{aligned}
\right.
\end{equation*}

\subsection{Contributions}
The main contributions of our paper are three-fold.
\begin{itemize}
\item  \textbf{Unification}

Our framework \eqref{Eq_Framework} provides a unified decentralized update scheme with compression, which covers two mainstream compression approaches: unbiased compression and contractive compression with error compensation. Moreover, the proposed framework \eqref{Eq_Framework} incorporates vanilla stochastic subgradient descent (SGD) along with various SGD-based acceleration techniques, including momentum and gradient-tracking, into the update scheme. In particular, we show that the \textit{nonsmooth extensions} of common decentralized stochastic gradient-type  methods with communication compression fit into our framework, such as QDGD \cite{reisizadeh2019exact}, CHOCO-SGD \cite{koloskova2019decentralized2}, BEER \cite{zhao2022beer} and C-GT \cite{liao2022compressed}. 
\item \textbf{Convergence}

We establish the global convergence of  \eqref{Eq_Framework} by connecting the consensus-error iterates and averaged iterates to the trajectories of the delicately constructed  noiseless differential inclusions. To the best of our knowledge, this is the first work that rigorously proves global convergence to critical points for a wide range of existing decentralized compression-based methods in nonsmooth nonconvex optimization, especially in the training of nonsmooth neural networks.

\item  \textbf{Development} 

Based on our \eqref{Eq_Framework}, we develop decentralized compression-based variants of momentum SGD and SignSGD with theoretical guarantees. Preliminary numerical experiments demonstrate the efficiency of these methods and highlight the potential of our framework for designing decentralized stochastic subgradient-type methods with communication compression.
\end{itemize}

\subsection{Organization}
The rest of this paper is organized as follows. Section \ref{sec:pre} introduces the notations and  preliminary concepts used throughout the paper. Section~\ref{sec:framework} presents our proposed framework, provides a detailed consensus analysis, and establishes its global convergence. Section \ref{sec:methods} demonstrates that our framework covers nonsmooth extensions of existing methods and enables the development of new methods with convergence guarantees. Section \ref{sec:numeric} exhibits the results of preliminary numerical experiments. In the last section, we draw conclusions and discuss possible future research directions.

\section{Preliminary}\label{sec:pre}

\subsection{Notations}
The operator $\langle\cdot, \cdot\rangle$ represents the standard inner product, while $\|\cdot\|$ represents the $\ell_2$-norm of a vector or the spectral norm of a matrix.  
 $\|\cdot\|_{1}$ stands for the $\ell_1$-norm of a vector, and  $\|\cdot\|_{\mathrm{F}}$ refers to the Frobenius norm of a matrix. Let $\mathbb{B}({\bm x}, \delta):=\left\{\tilde{{\bm x}} \in \mathbb{R}^n:\|\tilde{{\bm x}}-{\bm x}\|^2\leq \delta^2\right\}$ denote the ball centered at ${\bm x}$ with radius $\delta$. For a given set $\mathcal{Y}, \operatorname{dist}({\bm x}, \mathcal{Y}):= \min_{y \in \mathcal{Y}}\|{\bm x}-{\bm y}\|$ represents the distance between ${\bm x}$ and a set $\mathcal{Y}$. The convex hull and $d$-fold Cartesian product of $\mathcal{Y}$ is denoted by $\conv \mathcal{Y}$ and $\mathcal{Y}^d$, respectively. The notation $\otimes$ stands for the Kronecker product. The symbol $\odot$ denotes the Hadamard product, $\Delta_m:= \{(\lambda_0, \ldots, \lambda_m): \lambda_i \geq 0, \sum_{i=0}^{m}\lambda_i =1\}$ stands for the  simplex of dimension $m$.

For any positive sequence  $\left\{\eta_k\right\}$, let $\lambda_{\eta}(0):=0, \lambda_{\eta}(i):=\sum_{k=0}^{i-1} \eta_k$, and $\Lambda_{\eta}(t):=\sup \left\{k \in\mathbb{N}:\right. $ $\left. t \geq \lambda_{\eta}(k)\right\}.$
More explicitly, $\Lambda_{\eta}(t)=p$, if $\lambda_{\eta}(p) \leq t<\lambda_{\eta}(p+1)$. The set-valued mapping sign : $\mathbb{R}^n \rightrightarrows \mathbb{R}^n$ is defined by
$$(\operatorname{sign}({\bm x}))_i= \begin{cases}\{-1\}, & {\bm x}_i<0;\\ {[-1,1],} & {\bm x}_i=0;\\ \{1\}, & {\bm x}_i>0.\end{cases}$$

For any $N>0,$ let $[N]:=\{1, \ldots, N\}$. The notation $\mathbb{R}_{+}$ represents the set of all nonnegative real numbers. Notations ${\bm 1}_{d}$ and  ${\bm e}_{i}$ stand for a vector of all $1$'s and $[0, \ldots, 1, \ldots, 0]^{\top}$, where $1$ is the $i$-th component. For two integers $i$ and $j$, $i \land j:= \min\{i,j\}$. Let $(\Omega, \mathcal{F}, \mathbb{P})$ denote the probability space. We use $\sigma({\bm X})$ to denote the sigma-algebra generated by the random variable ${\bm X}$.

We say that $\left\{\mathcal{F}_k\right\}_{k \in \mathbb{N}}$ is a filtration if $\left\{\mathcal{F}_k\right\}$ is a collection of $\sigma$-algebras satisfying $\mathcal{F}_0\subseteq \mathcal{F}_1\subseteq \cdots \subseteq \mathcal{F}_{\infty} \subseteq \mathcal{F}$.  A sequence of random variables $\{\xi_{k}\}$ is a martingale with respect to a filtration $\{\mathcal{F}_{k}\}$, if $\{\xi_{k}\}$ is adapted to the filtration $\{\mathcal{F}_{k}\}$ and $\mathbb{E}[\xi_{k+1} | \mathcal{F}_{k} ] = \xi_{k}$,  for all $ k \in \mathbb{N}$; $\{\xi_{k}\}$ is a supermartingale with respect to ${\mathcal{F}_{k}}$, if ${\xi_{k}}$ is adapted to $\{\mathcal{F}_{k}\}$ and $\mathbb{E}[\xi_{k+1} | \mathcal{F}_{k}] \leq \xi_{k}$, for all $k\in\mathbb{N}$. Moreover, a sequence of random variables  $\{\xi_{k}\}$  is a martingale difference sequence with respect to $\{\mathcal{F}_{k}\}$, if $\{\xi_{k}\}$ is adapted to the filtration $\{\mathcal{F}_{k}\}$ and $\mathbb{E}[\xi_{k+1} | \mathcal{F}_{k} ] = 0$ holds for all $k\in \mathbb{N}$.

In addition, we denote the set of agent $i$'s neighbors by $\mathcal{N}_{i}$, and $\mathcal{N}_{i}^{+}:= \mathcal{N}_{i} \cup \{i\}$. We define the summation function $f$ of \eqref{Prob_DOP} as
\begin{equation}
    \label{Eq_defin_f}
    f({\bm x}) := \frac{1}{d} \sum_{i = 1}^d f_i({\bm x}).
\end{equation}

\subsection{Mixing matrix}\label{mix}
The mixing matrix ${\bm W}$ conforms to the topology of the communication network and plays an important role in aggregating local information from neighboring agents. Generally, we assume the mixing matrix ${\bm W}$ is defined to satisfy the following properties, which are standard in the literature.
\begin{defin}\cite[Section 1]{LeiJIMO}\label{Def_mixing_matrix}
Given a connected graph $\mathtt{G} = (\mathtt{V}, \mathtt{E})$, we say ${\bm W} \in \bb{R}^{d\times d}$ is a mixing matrix of $\mathtt{G}$, if it satisfies 
\begin{enumerate}
\item ${\bm W}$ is symmetric.
\item ${\bm W}$ is doubly stochastic, namely, ${\bm W}$ is nonnegative and ${\bm W} {\bm 1}_d={\bm W}^{\top} {\bm 1}_d={\bm 1}_d$.
\item ${\bm W}(i, j)=0$, if and only if $i \neq j$ and $(i, j) \notin \mathtt{E}$.
\end{enumerate}
\end{defin}
Given a graph $\mathtt{G}$, various approaches can be used to select its corresponding mixing matrix, such as the Laplacian-based constant edge-weight matrix \cite{xiao2004fast} and the Metropolis constant edge-weight matrix \cite{xiao2006distributed}. For further details on choosing the mixing matrix, we refer the reader to \cite{nedic2018network, shi2015extra}.

Proposition \ref{Prop_preliminary_mixing_matrix} is a direct corollary of \cite[Perron-Frobenius Theorem]{pillai2005perron}, which characterize the spectral property of a mixing matrix ${\bm W}$.

\begin{prop}
    \label{Prop_preliminary_mixing_matrix}
    For any mixing matrix ${\bm W} \in \bb{R}^{d\times d}$ associated with a connected graph $\mathtt{G}$, all eigenvalues of ${\bm W}$ lie in $(-1,1]$, and ${\bm W}$ has a single eigenvalue equal to $1$ with the all-ones vector ${\bm 1_{d}}$ as its right eigenvector. 
\end{prop}

\subsection{Set-valued mapping, Clarke subdifferential and conservative field}\label{sec:conserv}

A set-valued mapping $\Phi: \mathbb{R}^{n} \rightrightarrows \mathbb{R}^m $ is a mapping from $\mathbb{R}^n$ into the set of subsets of $\Rm$. The graph of $\Phi$ is defined by 
\begin{equation*}
\operatorname{graph} \Phi = \{ ({\bm x},{\bm z}) \mid {\bm x}\in \mathbb{R}^n, {\bm z}\in \Phi({\bm x})\}.
\end{equation*}
$\Phi$ is said to have a closed graph (or be graph-closed), if $\operatorname{graph} \Phi$ is a closed subset of $\Rn \times \Rm$. It is locally bounded, if for any ${\bm x}\in \mathbb{R}^n$, there exists a neighborhood $U_{{\bm x}}$ of ${\bm x}$ such that 
\begin{equation*}
\sup_{{\bm z}\in \Phi({\bm y}), {\bm y}\in U_{{\bm x}}} \| {\bm z}\|< + \infty.
\end{equation*}
In addition, $\Phi$ is convex-valued (resp. compact-valued), if $\Phi({\bm x})$ is a convex (resp. compact) subset of $\Rm$ for any ${\bm x} \in \Rn$. For $\delta>0$, we define 
\begin{equation*}
\Phi^\delta({\bm x}):=\bigcup_{{\bm y} \in \mathbb{B}({\bm x}, \delta)}(\Phi({\bm y})+\mathbb{B}({\bm 0}, \delta))
\end{equation*}
where ``+'' denotes the Minkowski sum.

\begin{defin}[Clarke subdifferential \citep{clarke1990optimization}]
For any given locally Lipschitz continuous function $f: \mathbb{R}^{n} \to \mathbb{R}$, and for any ${\bm x}\in \mathbb{R}^n$, the generalized directional derivative of  $f$ at ${\bm x}$ along the direction ${\bm d} \in \mathbb{R}^n$ is defined by  
\begin{equation*}
  f^{\circ}({\bm x};{\bm d}):=  \limsup_{{\bm y}\to {\bm x}, t\downarrow 0}\frac{f({\bm y}+t{\bm d})-f({\bm y})}{t}.
\end{equation*}
The Clarke subdifferential of $f$ at ${\bm x}\in \mathbb{R}^n$, denoted by $\partial f({\bm x})$, is given by 
\begin{equation*}
\partial f({\bm x}):= \left\{ {\bm u} \in \mathbb{R}^n:  f^{\circ}({\bm x};{\bm d}) \geq \langle {\bm u}, {\bm d} \rangle, \forall {\bm d}\in \mathbb{R}^n \right\}.
\end{equation*}
\end{defin}

Notice that $\partial f$ is a set-valued mapping that is convex-valued, graph-closed, and locally bounded. Based on the concept of generalized directional derivative, we now present the definition of Clarke regular functions.

\begin{defin}[Clarke regular \citep{clarke1990optimization}]
We say that $f$ is Clarke regular at ${\bm x} \in \Rn$, if for every direction
${\bm d}\in \Rn$, the one-sided directional derivative 
\begin{equation*}
    f^{*}({\bm x};{\bm d}):= \lim_{t\downarrow 0}\frac{f({\bm x}+t{\bm d})-f({\bm x})}{t}
\end{equation*}
exists, and equals the generalized directional derivative, i.e. $f^{*}({\bm x};{\bm d})= f^{\circ}({\bm x};{\bm d}).$ 
\end{defin}

Next, we present a brief introduction to concept of a  conservative field, which is used to describe the output of the AD algorithm applied to nonsmooth neural networks.

\begin{defin}[Conservative field \cite{bolte2021conservative}] \label{def:conservative}
    Let $D: \mathbb{R}^{n} \rightrightarrows \mathbb{R}^n$ be a nonempty set-valued mapping. We say that  $D$ is a conservative field if it is compact-valued and graph-closed, and for any absolutely continuous loop $\gamma: [0,1] \rightarrow \mathbb{R}^n$ satisfying $\gamma(0)=\gamma(1)$, it holds that 
        \begin{equation} \label{eq:cons}
     \int_0^1\max_{{\bm v} \in D(\gamma(t))}\langle\dot{\gamma}(t), {\bm v}\rangle \mathrm{d} t=0.
        \end{equation}
\end{defin}

It is worth noting that any conservative field is locally bounded \citep[Remark 3]{bolte2021conservative}. We now introduce the definition of the path-differentiable function corresponding to a conservative field.

\begin{defin}[Path-differentiable]
Let $D: \mathbb{R}^{n} \rightrightarrows \mathbb{R}^{n}$ be a conservative field. A function $f$ is said to be path-differentiable for $D$ if there exists ${\bm x}_0\in \mathbb{R}^{n}$ such that  
\begin{equation*}
f({\bm x})={f({\bm x}_{0})}+\int_0^t\langle\dot{\gamma}(s), v(s) \rangle \mathrm{d} s,
\end{equation*}
for any absolutely continuous curve $\gamma$ with $\gamma(0)={\bm x}_{0}$ and $\gamma(t)={\bm x}$ and some measurable selection $v(s) \in D(\gamma(s)) $. We also say that $D$ is a conservative field for $f$, denoted by $D_{f}$. 
\end{defin}

\begin{prop}[Corollary 1 of \cite{bolte2021conservative}] \label{lem:Clarke-Conservative}
Let $D_f:\mathbb{R}^n\rightrightarrows \mathbb{R}^n$ be a conservative field for a path-differentiable function $f: \mathbb{R}^n \to \mathbb{R}$. Then $D_{f}({\bm x})= \{\nabla f({\bm x})\}$ for almost every ${\bm x} \in \mathbb{R}^n$. Furthermore, $\partial f$ is a conservative field for $f$, and
\begin{equation*}
\partial f({\bm x}) \subseteq \conv(D_f ({\bm x}))
\end{equation*}
holds for all ${\bm x} \in \mathbb{R}^n$.
\end{prop}

Proposition \ref{lem:Clarke-Conservative} reveals that $\partial f$  is the  smallest convex-valued conservative field for $f$. Therefore, the concept of a conservative field can be regarded as an extension of the Clarke subdifferential. More importantly, when the conservative field $D_f({\bm x})$ is convex-valued, the condition ${\bm 0} \in \partial f({\bm x})$ implies that $0 \in D_f({\bm x})$.

\begin{rmk}\label{rmk:01}
The class of path-differentiable functions is general enough to encompass most objective functions encountered in real-world problems. Notably, the well-known Clarke regular functions \cite{Clarke1998NonsmoothAA} and semi-algebraic functions \cite{ojasiewicz1965EnsemblesS} are both path-differentiable. As discussed in \cite{davis2020stochastic}, an important subclass of path-differentiable functions is the class of definable functions, namely, functions whose graphs are definable in an o-minimal structure. In fact, most activation and loss functions used in deep neural networks are definable, including sigmoid, softplus, ReLU, leaky ReLU, hinge loss, and others. Owing to the invariance of definability under finite summation and composition \cite{davis2020stochastic}, any neural network built from definable blocks has a definable loss function and a definable conservative field, thereby making its loss function path-differentiable.
\end{rmk}

The following proposition shows that the definability of both $f$ and its conservative field $D_f$ implies the nonsmooth Morse–Sard property \citep{bolte2007clarke}.

\begin{prop}[Theorem 5 of \cite{bolte2021conservative}]\label{prop:finite}
 Let $f$ be a path-differentiable function that admits $\D_f$ as its conservative field. Suppose that both $f$ and $\D_f$ are definable over $\mathbb{R}^n$. Then the set $\left\{{f({\bm x})}:{\bm 0}\in \D_f({\bm x}) \right\}$ is finite.
\end{prop}

Finally,  based on the concept of a conservative field, we introduce the definition of critical points for the optimization problem \eqref{Prob_DOP}.

\begin{defin}
    \label{Defin_critical_points}
    Let $f$ in \eqref{Eq_defin_f} be a path-differentiable function that admits $\D_f$ as its convex-valued conservative field. A point \({\bm X} \in \mathbb{R}^{n \times d}\) is said to be a \(D_f\)-critical point of~\eqref{Prob_DOP} if it satisfies the consensus condition
  ${\bm X} = \frac{1}{d}{\bm X}{\bm 1_{d}} {\bm 1}_d\tp$  and 
    \begin{equation*}
       {\bm 0} \in \D_f(\frac{1}{d}{\bm X}{\bm 1_{d}}).  
    \end{equation*}
   Similarly,  ${\bm X}\in \Rnd$ is called a  Clarke-critical (or \(\partial f\)-critical) point of \eqref{Prob_DOP}, if it satisfies the same consensus condition ${\bm X} = \frac{1}{d}{\bm X}{\bm 1_{d}} {\bm 1}_d\tp$ and 
    \begin{equation*}
   {\bm 0} \in \partial f(\frac{1}{d}{\bm X}{\bm 1_{d}}).
     \end{equation*}
\end{defin}

\subsection{Stochastic approximation and differential inclusion}

ODE approaches \cite{benaim2005stochastic,borkar2009stochastic,duchi2018stochastic,davis2020stochastic} are powerful tools for analyzing convergence in stochastic approximation, particularly for the iterates of stochastic subgradient-type methods. These approaches characterize the convergence of the iterates via the asymptotic behavior of the dynamics of the associated differential inclusion. We first recall some basic definitions related to differential inclusions.

\begin{defin}
Let $\Phi: \mathbb{R}^n \rightrightarrows \mathbb{R}^n$ be a set-valued mapping. An absolutely continuous curve $\gamma: \mathbb{R}_{+} \rightarrow \mathbb{R}^n$ is called a solution (or trajectory) of the  differential inclusion 
\begin{equation}\label{eq:diff}
\dot{\bm x}(t)\in \Phi({\bm x}),
\end{equation}
with initial condition $\gamma(0) = {\bm x}_0$, if   $\gamma'(t)\in \Phi(\gamma(t))$ for almost all $t\in \mathbb{R}_{+}$. 
\end{defin}

\begin{defin}[Lyapunov function]
Let $\mathcal{B} \subset \mathbb{R}^n$ be a closed set. A continuous function $\psi: \mathbb{R}^n \rightarrow \mathbb{R}$ is said to be a Lyapunov function for the differential inclusion \eqref{eq:diff} with a stable set $\mathcal{B}$, if for any solution $\gamma$ to \eqref{eq:diff} and any $t>0$, it holds that 
\begin{equation*}
    \psi(\gamma(t))\leq \psi(\gamma(0)).
\end{equation*}
Moreover, for  $\gamma(0) \notin \mathcal{B}$, it holds for all $t> 0$ that 
\begin{equation*}
    \psi(\gamma(t)) < \psi(\gamma(0)).
\end{equation*}
\end{defin}

Consider iterates $\{{\bm x}_{k}\}$ generated by the following update scheme:
\begin{equation}
    \label{Eq_stable_random}
    {\bm x}_{k+1} \in {\bm x}_k - c\eta_k\left( \Phi^{\delta_k}({\bm x}_k) + \upsilon_{k+1}  \right), 
\end{equation}
where $\{\eta_k\}$ is a non-summable positive sequence of step-sizes, $\{\delta_{k}\}$ is a nonnegative sequence, and $\upsilon_{k+1}$ is a random noise term added when evaluating $\Phi({\bm x}_k)$. The continuous-time interpolated process $u: \mathbb{R}_{+} \rightarrow \mathbb{R}^n$  induced by \eqref{Eq_stable_random} is given by
\begin{equation*}
u(\lambda_{\eta}(k) +s) = {\bm x}_{k} +  \frac{{\bm x}_{k+1}-{\bm x}_k}{\eta_{k}} s, \quad s \in [0,\eta_{k}).
\end{equation*}
where $\lambda_{\eta}(0):=0$ and $\lambda_{\eta}(k):= \sum_{i=0}^{k-1} \eta_{i}$, for $k \geq 1$.

The following Lemma \ref{lem_stable_random} plays an important role in demonstrating the convergence properties of \eqref{Eq_stable_random}, integrating results from \cite[Lemma 2.20]{xiao2023adam}, \cite[Theorem 3.6, Proposition 3.27]{benaim2005stochastic}. In their proof, they show that the interpolated process above is a perturbed solution of differential inclusion \eqref{eq:diff}. For further details, interested readers are referred to \cite{benaim2006dynamics} and \cite{benaim2005stochastic}.

\begin{assumpt}
\label{Assumption_stable_random}
\mbox{}
\begin{enumerate}
\item The sequence $\{{\bm x}_k\}$ is uniformly bounded, and $\lim_{k \to \infty}\delta_{k} =0$. 
\item There exists a locally Lipschitz continuous Lyapunov function $\psi: \Rn \to \bb{R}$ for the differential inclusion 
        \begin{equation}
    \label{Eq_stable_DI}
    \frac{\mathrm{d}{\bm x}}{\mathrm{d}t} \in -\Phi({\bm x}),
\end{equation} 
with a stable set $\ca{B}$. Moreover, the set $\{\psi({\bm x}): {\bm x} \in \ca{B}\}$ is a finite subset of $\bb{R}$. 
\item For any $T>0$, it holds that 
\begin{equation*}
    \lim_{s \to +\infty}\sup_{s\leq i\leq \Lambda_{\eta}(\lambda_{\eta}(s) + T)} \norm{ \sum_{k = s}^{i}\eta_k \upsilon_{k+1} } =0. 
\end{equation*}
\end{enumerate}
\end{assumpt}

\begin{lem}\label{lem_stable_random}
Suppose Assumption \ref{Assumption_stable_random} holds. Let $\X_0$ be a compact subset of $\Rn$, and let the sequence $\{x_{k}\}$ be generated by the update scheme \eqref{Eq_stable_random} with $x_0 \in \X_0$.  Then it follows that
    \begin{equation*}
        \lim_{k\to \infty} \mathrm{dist} \left(\xk, \mathcal{B} \right)  = 0, 
    \end{equation*}
and the sequence $\{\psi({\bm x}_k)\}$ converges to $\psi({\bm x}^*)$, where ${\bm x}^* \in \mathcal{B}$.
\end{lem}

\section{Convergence Guarantees for General Framework}\label{sec:framework}

In this section, we establish the global asymptotic convergence properties for \eqref{Eq_Framework}. 
Section \ref{sec:basic} discuss about the relationship between ${{\bm Z}_k}$ and ${\bm H}_{k}$  in the framework \eqref{Eq_Framework}, and introduces basic assumptions for \eqref{Eq_Framework}. Section \ref{consensus} relates the consensus-error iterates to the trajectories of a noiseless continuous-time differential inclusion, and demonstrates the consensus properties of two sub-frameworks. In Section \ref{glo_con}, we connects the 
averaged iterates to another continuous-time differential inclusion, and establish the global convergence to the stable set of corresponding differential inclusion.

\subsection{Basic assumptions}\label{sec:basic}

In the nonsmooth setting, a common choice for  noiseless update directions ${\bm H}_{k}$ is  
\begin{equation*}
{\bm H}_{k} \in [\Phi_{1}({\bm z}_{1,k}), \ldots, \Phi_{d}({\bm z}_{d,k})], 
\end{equation*}
where each $\Phi_{i}$ is a set-valued mapping, such as the Clarke subdifferential $\partial f_i$ or a conservative field $D_{f_{i}}$ discussed in Section \ref{sec:conserv}. This choice naturally yields a stochastic subgradient descent method with communication compression. 

To encompass a broader class of subgradient-type methods, we introduce a family of set-valued mappings $\{\Phi_i({\bm z})\}_{i=1}^{d}$ that possess a “Lyapunov property” to constrain ${\bm H}_{k}$. Furthermore, we specify the form of the evaluation noise $\Xi_{k+1}$ and the step-sizes used in framework \eqref{Eq_Framework}.

\begin{assumpt}
\label{Assumption_framework}
\mbox{}
\begin{enumerate}
\item[(1)] There exists a sequence $\{\epsilon_k\} \subseteq \mathbb{R}_{+}$ and a family of locally bounded  and graph-closed set-valued mappings $\Phi_{i}, i \in [d]$ such that 
\begin{equation}
    \frac{1}{d}{\bm H}_k {\bm 1}_d \in \conv\left( \frac{1}{d} \sum_{i = 1}^d \Phi_i^{\epsilon_k}({\bm z}_{i,k}) \right), \quad \forall k\in \mathbb{N}. 
\end{equation}
Moreover, $\{{\bm H}_{k}\}$ is bounded and $\{\epsilon_k\}$ 
 is diminishing whenever $\{{\bm Z}_{k}\}$ is bounded. 

\item[(2)] The differential inclusion
\begin{equation}\label{eq:diff_inclu}
\frac{\mathrm{d}{\bm z}}{\mathrm{d}t} \in - \Phi({\bm z}):= -\conv \left( \frac{1}{d}\sum_{i = 1}^d \Phi_i({\bm z}) \right),
\end{equation}
admits a locally Lipschitz continuous Lyapunov function $\psi: \Rm \to \bb{R}$, whose stable set is denoted by $\ca{A}$, and  $\{\psi({\bm z}): {\bm z} \in \ca{A}\}$ is a finite subset of $\mathbb{R}$.

\item[(3)] The evaluation noise $\{\Xi_{k+1}\}$ is a martingale difference sequence, and it is uniformly bounded whenever $\{{\bm Z}_{k}\}$ is bounded. 

\item[(4)] The sequences of step-sizes  $\{\eta_{k}\}$ and $\{\theta_{k}\}$ satisfy the following conditions:
\begin{equation}\label{eq:ologk}
\sum_{i = 0}^{\infty} \eta_k = +\infty, \quad \lim_{k\to +\infty} \frac{\eta_k}{\theta_k} = 0, \quad \lim_{k\to +\infty} \theta_{k}\log(k) =0, \quad \lim_{k\to +\infty} \frac{\theta_{k}^2}{\eta_k}\log(k) =0 . 
\end{equation}
Moreover, the sequence $\{\gamma_k\}$ satisfies
\begin{equation}\label{eq:gammak}
\gamma_0 =1,  \quad \lim_{k\to \infty}\gamma_k \log(k) =0, \quad \lim_{k\to \infty}\frac{\theta_{k}}{\gamma_k} =0.  
\end{equation}

\end{enumerate}
\end{assumpt}

Assumption \ref{Assumption_framework}-(1) uses a family of set-valued mappings $\{\Phi_{i}\}$ to characterize the relationship between the averaged update direction $\frac{1}{d}{\bm H}_k {\bm 1}_d$ and local variables $\{{\bm z}_{i,k}\}_{i\in [d]}$. Note that we do not require each column of ${\bm H}_{k}$ to belong to $\Phi_i({\bm z}_{i,k})$, but instead merely restrict the averaged update direction to lie in the convex hull of the Minkowski sum of $\epsilon_k$-neighborhoods of the set-valued mapping's images. This yields a very general assumption, which enables the development of decentralized SGD-type methods within our framework.

Assumption \ref{Assumption_framework}-(2) is frequently employed in the literature, e.g., \cite{benaim2005stochastic, borkar2008stochastic, bian2009subgradient, davis2020stochastic, bolte2021conservative, josz2023lyapunov}. It captures the descent property of $\psi$ along trajectories of the differential inclusion \eqref{eq:diff_inclu}, while the property that $\{\psi({\bm z}): {\bm z} \in \ca{A}\}$ is finite corresponds to the weak Sard property in \cite[Assumption B]{davis2020stochastic}, which holds when $f$ is a definable function under the selection $\Phi:= \partial f$ and $\psi:= f$.

Assumption \ref{Assumption_framework}-(3) and (4) impose mild technical conditions on the evaluation noise $\{\Xi_{k+1}\}$, and allow for a flexible choice of three-timescale step-sizes $\{\eta_k, \theta_k, \gamma_k\}$. One simple choice is $\eta_{k} = o(1/\log k)$, $\theta_{k}=\eta_{k}(\eta_{k}\log(k))^{-s},$  $\gamma_k = \eta_{k}(\eta_{k}\log(k))^{-3s/2}$ with $s\in(0,\frac{1}{2})$.

In the following, we stipulate the standing assumptions on unbiased compression and contractive compression, and present several typical examples.

\begin{assumpt}[Unbiased compression operator]\label{Assumption_compress0}
For a compression operator $C(\cdot, \cdot): \mathbb{R}^m \times \Omega \to \mathbb{R}^m$, we assume that $C$ is unbiased and satisfies a linear growth bound, that is, there exists a constant $\beta>0$ such that
\begin{equation}
\mathbb{E}_{{\bm \omega}}[C({\bm x}, {\bm \omega})]={\bm x}, \qquad \|C({\bm x}, {\bm \omega})\|\le \beta \|{\bm x}\|,
\quad \forall {\bm x}\in\mathbb{R}^m,\ \text{a.s. in } \Omega.
\end{equation}
\end{assumpt}

\begin{examples}[Random quantization \cite{alistarh2017qsgd}]\label{exp:0}

For any ${\bm x}\in \mathbb{R}^m$ and precision level $ s\in \mathbb{N}_{+}$, the random quantization operator $Q_s$, which is an unbiased compression operator, is defined as
\begin{equation*}
Q_s({\bm x}):=  \|{\bm x}\|_2 \mathrm{sign}({\bm x}) \odot \zeta({\bm x}, s),
\end{equation*} 
where 
\begin{equation*}
\zeta({\bm x}, s) = [\zeta_1(x_1, s), \ldots, \zeta_m(x_m, s)]^{\top}, \quad
\zeta_i(x_i, s)= \begin{cases}
 \frac{l+1}{s}, & \text{ with probability }  \frac{s|x_i|}{\|{\bm x}\|} - l, \\
 \frac{l}{s},& \text{ otherwise}, \\
\end{cases}   
\end{equation*}
for some integer $l$ satisfying $\frac{l}{s}\leq  \frac{|x_i|}{\|{\bm x}\|} < \frac{l+1}{s}$.
\end{examples}

\begin{assumpt}[Contractive compression operator]\label{Assumption_compress}
\mbox{}
For a compression operator $C(\cdot, \cdot): \mathbb{R}^m \times \Omega \rightarrow \mathbb{R}^m$, we assume there exists  $\alpha\in (0,1]$ such that
\begin{equation}\label{con:original}
\mathbb{E}_{{\bm \omega}}[\|C(\bm{x}, {\bm \omega})-\bm{x}\|^2] \leq(1-\alpha)\|\bm{x}\|^2, \quad \forall \bm{x} \in \mathbb{R}^m.
\end{equation}

\end{assumpt}
Assumption \ref{Assumption_compress} is a mild condition, which requires the deviation $\|C(\bm{x}, {\bm \omega})-\bm{x}\|^2$ to be proportional to the squared norm of ${\bm x}$. Indeed, this assumption has been widely adopted in the literature \cite{huang2022lower, stich2018sparsified, richtarik2021ef21}.

\begin{examples}\label{exa:1}
\begin{enumerate}
\item Random-$k$ compression operator: 
\begin{equation*}
\left(C({\bm x}, {\bm \omega})\right)_i:= \left\{\begin{array}{ll}{\bm x}_i, & \text { if } i \in \omega, \\0, & \text { otherwise, }\end{array}\right.
\end{equation*}
where ${\bm x}\in\mathbb{R}^{m}$ and $\omega \subseteq [m]$ is a random subset with $|\omega|=k$.  
This operator satisfies Assumption~\ref{Assumption_compress} with $\alpha = \frac{k}{m}$.

\item Top-$k$ compression operator
\begin{equation*}
\left(C({\bm x}, {\bm \omega})\right)_i:=\left\{\begin{array}{ll}{\bm x}_{i}\mathbbm{1}_{\pi(i)\leq k}, & \text { if } i \leq k, \\0, & \text { otherwise, }\end{array}\right.
\end{equation*}
where ${\bm x}\in\mathbb{R}^{m}$ and $\pi$ is a permutation of $[m]$ such that $\pi(i)$ denotes the position of $|{\bm x}_i|$ in the ordering of $\{|{\bm x}_j|\}_{j=1}^m$.  
This operator satisfies Assumption~\ref{Assumption_compress} with $\alpha = \frac{k}{m}$.

\end{enumerate}
\end{examples}

\subsection{Consensus analysis}\label{consensus}

In this subsection, we investigate the consensus property of the framework \eqref{Eq_Framework}. With the notations ${\bm P}:= \frac{1}{d}{\bm 1_{d}}  {\bm 1}_{d}\tp$ and ${\bm P}_{\perp}:= {\bm I}_{d}- {\bm P}$, ${\bm Z}_{k}\in \mathbb{R}^{m \times d}$ admits the following orthogonal decomposition,
\begin{equation*}
{\bm Z}_{k} = {\bm Z}_{k}{\bm P} + {\bm Z}_{k}{\bm P}_{\perp}. 
\end{equation*} 
Here, ${\bm Z}_{\perp, k} := {\bm Z}_{k} {\bm P}_{\perp}$ measures the dissimilarity of the local variables across all agents at the $k$-th iteration, which is referred to as the consensus-error sequence.

\begin{assumpt}\label{asp:global_stability}
\mbox{}
The sequence $\{\bm Z_{k}\}$ in \eqref{Eq_Framework} is bounded.
\end{assumpt}  

Assumption~\ref{asp:global_stability} is a standard global stability condition in ODE-based analyses of nonsmooth optimization, and has been widely adopted in the literature \cite{benaim2005stochastic, benaim2006dynamics, bolte2021conservative, davis2020stochastic, castera2021inertial}). In practice, it is also considered as a mild condition.


\begin{prop}[Consensus: unbiased compression]\label{prop:consensus_unb}
Suppose Assumption \ref{Assumption_framework}, \ref{Assumption_compress0} and \ref{asp:global_stability} hold. For any sequence $\{{\bm Z}_{k}\}$ generated by \eqref{eq:uncom}, it  satisfies
\begin{equation*}
\lim_{k\to +\infty} \norm{{\bm Z}_{\perp,k}  } = 0. 
\end{equation*}
\end{prop}

\begin{proof}
By straightforward algebraic calculations, we have 
\begin{equation}\label{eq:flatten_unb}
\begin{aligned}
    {\bm Z}_{\perp, k+1} & = [{\bm Z}_{k}((1-\theta_{k}){\bm I}_{d}+\theta_{k}{\bm W})]{\bm P}_{\perp} + \theta_{k}{\bm E}_{k+1}({\bm W}-\Diag({\bm W})){\bm P}_{\perp}- \eta_{k}({\bm H}_{k}+\Xi_{k+1}){\bm P}_{\perp}\\
    & = {\bm Z}_{\perp, k} ((1-\theta_{k}){\bm I}_{d}+\theta_{k}{\bm W}) + \theta_{k}{\bm E}_{k+1}({\bm W}-\Diag({\bm W})){\bm P}_{\perp} - \eta_{k}({\bm H}_{k}+\Xi_{k+1}){\bm P}_{\perp}\\
    & = {\bm Z}_{\perp, k} - \theta_{k} [{\bm Z}_{\perp, k} ({\bm I}_{d} - {\bm W}) + \frac{\eta_{k}}{\theta_{k}}{\bm H}_{k}{\bm P}_{\perp}] + \theta_{k} [{\bm E}_{k+1}({\bm W}-\Diag({\bm W})){\bm P}_{\perp}- \frac{\eta_{k}}{\theta_{k}}\Xi_{k+1}{\bm P}_{\perp}],
\end{aligned}
\end{equation}
where the second equality follows from
$$ {\bm P}_{\perp} ((1-\theta_{k}){\bm I}_{d}+\theta_{k}{\bm W})  =((1-\theta_{k}){\bm I}_{d}+\theta_{k}{\bm W}) {\bm P}_{\perp}.
$$ 

Let $\Phi({\bm Z}):= {\bm Z}({\bm I}_{d}-{\bm W})$,
$\upsilon_{k+1}:= {\bm E}_{k+1}({\bm W}-\Diag({\bm W})){\bm P}_{\perp}- \frac{\eta_{k}}{\theta_{k}}\Xi_{k+1}{\bm P}_{\perp}$ and $\delta_{k}:= \frac{\eta_{k}}{\theta_{k}}\|{\bm H}_{k}{\bm P}_{\perp}\|$. Then \eqref{eq:flatten_unb} can be rephrased as
\begin{equation}\label{eq:unb_rephrase}
{\bm Z}_{\perp, k+1} \in {\bm Z}_{\perp, k} - \theta_{k} (\Phi^{\delta_{k}}({\bm Z}_{\perp, k}) + \upsilon_{k+1}). 
\end{equation}

By Assumption \ref{asp:global_stability} and \ref{Assumption_framework}, $\{{\bm H}_{k}\}$ is uniformly bounded and $\lim_{k\to \infty}\frac{\eta_{k}}{\theta_{k}} = 0$, which implies that $\{\delta_k\}$ diminishes to $0$. Let $\psi({\bm Z}):= \frac{1}{2}\|{\bm Z}({\bm I}_{d}-{\bm W})^{\frac{1}{2}}\|_{\mathrm{F}}^2$ and $\mathcal{B}:= \{{\bm Z} | {\bm Z}({\bm I}_{d}-{\bm W})=0 \}$. For any solution ${\bm Z}(t)$ to the differential inclusion $\frac{\mathrm{d} {\bm Z}}{\mathrm{d} t}\in -\Phi({\bm Z})$, it holds that
\begin{equation*}
\frac{\mathrm{d} \psi({\bm Z}(t))}{\mathrm{d} t} = \left\langle \frac{\partial \psi({\bm Z})}{\partial {\bm Z}}, \frac{\mathrm{d}{\bm Z}(t)}{\mathrm{d}t}\right\rangle = \left\langle {\bm Z}(t)({\bm I}_d -{\bm W}), -{\bm Z}(t)({\bm I}_d -{\bm W}) \right\rangle \leq 0.
\end{equation*}
When ${\bm Z}(0)\notin \mathcal{B}$, it holds that $\frac{\mathrm{d} \psi({\bm Z}(t))}{\mathrm{d} t}|_{t=0} <0$. Hence, $\psi({\bm Z})$ is a Lyapunov function of  $\frac{\mathrm{d} {\bm Z}}{\mathrm{d} t}\in -\Phi({\bm Z})$ and $\mathcal{B}$ is a stable set. Moreover, $\psi({\bm Z})$ is coercive and locally Lipschitz continuous, and $\{\psi({\bm x}): {\bm x}\in \mathcal{B}\} = \{0\}$, which verifies Assumption \ref{Assumption_stable_random}-(2).

In addition, one can check that
\begin{equation*}
\mathbb{E}[\upsilon_{k+1}|\mathcal{F}_k]=
\mathbb{E}\left[{\bm E}_{k+1}({\bm W}-\Diag({\bm W})){\bm P}_{\perp}-\frac{\eta_{k}}{\theta_{k}}\Xi_{k+1}{\bm P}_{\perp}\mid \mathcal{F}_{k}\right]=0.
\end{equation*}
and $\{\upsilon_{k+1}\}$ is uniformly bounded, which follows from the definition of ${\bm E}_{k+1}$ and Assumption \ref{Assumption_framework}-(4). According to \cite[Proposition 4.4]{benaim2006dynamics}, a uniformly bounded martingale difference sequence $\{\upsilon_{k+1}\}$ together with  $\{\theta_{k}\}$ of order $o(1/\log(k))$ is a special case of Assumption \ref{Assumption_stable_random}-(3).

By applying Lemma \ref{lem_stable_random}, we can conclude that the sequence $\{\|{\bm Z}_{\perp, k}({\bm I}_{d}-{\bm W})^{\frac{1}{2}}\|\}$ converges to  $\psi({\bm Z})|_{{\bm Z}\in \mathcal{B}}= 0$. This implies that $\{{\bm Z}_{\perp, k}\}$ converges to the consensus space $\{{\bm Z}| {\bm Z}={\bm z}{\bm 1}_{d}^{\top}, {\bm z}\in \mathbb{R}^{m}\}$. By the definition of $\{{\bm Z}_{\perp, k}\}$,  we further achieve that  $\lim_{k\to +\infty} \norm{{\bm Z}_{\perp,k}  } = 0$, which completes the proof.  

\end{proof}

Proposition \ref{prop:consensus_unb} describe the consensus property of the sub-framework \eqref{eq:uncom}. To reveal the analogous properties in sub-framework \eqref{eq:contractive}, we  introduce some useful Lemmas \ref{lem02}-\ref{lem:noise_to_0}, whose proof is shown in Appendix \ref{sec:appendix-01} and \ref{sec:appendix-02}.

\begin{lem}\label{lem02}
Suppose that positive sequences $\{\gamma_{k}\}$ and $\{\theta_{k}\}$  satisfy 
\begin{equation*}
\lim_{k\to +\infty} \gamma_{k}\log(k)=0, \quad \lim_{k\to +\infty} \frac{\theta_{k}}{\gamma_{k}}=0, \quad \sum_{k=1}^{\infty} \theta_{k} =+\infty.
\end{equation*}
Then, for any $a\in(0,1]$, we have
\begin{equation*}
\lim_{k\to +\infty} \sum_{i=1}^k \left(\prod_{j=i}^{k} (1-a\gamma_j)\right) C_0\theta_{i-1} =0.
\end{equation*}
\end{lem}

\begin{lem}\label{lem01}
Let $\{\upsilon_{k}\}$ be a scalar martingale difference sequence with respect to the filtration $\{\mathcal{F}_k\}$, which is uniformly bounded. Let $a \in (0,1]$, and $\{\gamma_k\}$ be a sequence satisfying
\begin{equation*}
\lim_{k\to +\infty} \gamma_{k}\log(k)=0, \qquad \sum_{k=1}^{\infty} \gamma_{k} =+\infty.
\end{equation*}
Then it follows that
\begin{equation*}
\lim_{k \to +\infty} \left(\gamma_{k}\upsilon_{k+1} + \sum_{i=1}^{k-1} \gamma_i \left( \prod_{j=i+1}^{k} (1 - a\gamma_j) \right) \upsilon_{i+1} \right)=0. 
\end{equation*}
\end{lem}

Utilizing above tools, the following lemma shows that the compression error ${\bm E}_{k}$ is bounded and converges to zero almost surely.

\begin{lem}\label{lem:noise_to_0}
Suppose Assumption \ref{Assumption_framework}, \ref{Assumption_compress} and \ref{asp:global_stability} hold. For any $\{{\bm Z}_{k}\}$ generated by \eqref{eq:contractive}, it follows that almost surely, 

\begin{equation*}
\lim\limits_{k\to +\infty} \|{\bm E}_{k}\| =0. 
\end{equation*}
\end{lem}

\begin{proof}
We first observe that ${\bm E}_{k+1}$ in \eqref{eq:contractive} can be rewritten as
\begin{equation*}
{\bm E}_{k+1} = \gamma_{k}(C({\bm Z}_{k}- \hat{\bm Z}_{k}, {\bm \omega}_{k+1})-({\bm Z}_k -\hat{\bm Z}_k))-(1-\gamma_{k})({\bm Z}_k -\hat{\bm Z}_k).
\end{equation*}
Let ${\bm S}_{k+1}:=C({\bm Z}_{k}- \hat{\bm Z}_{k}, {\bm \omega}_{k+1})-({\bm Z}_k -\hat{\bm Z}_k)$, and  $\mathcal{F}_{k}:=\sigma(\{ {\bm \omega}_{j}, {\bm Z}_j, \Xi_{j}| j\leq k\})$ denote the $\sigma$-algebra at the $k$-th iteration. Straightforward calculations yield
\begin{equation*}
\begin{aligned}
& \mathbb{E}(\|{\bm E}_{k+1}\| | \mathcal{F}_{k}) \\
 \leq & \gamma_k \mathbb{E}[\|{\bm S}_{k+1} \|| \mathcal{F}_{k}] +(1-\gamma_k)\|{\bm Z}_{k}-\hat{\bm Z}_{k}\| \\
 \leq & \gamma_k \sqrt{1-\alpha}\|{\bm Z}_{k}-\hat{\bm Z}_{k}\| + (1-\gamma_k)\|{\bm Z}_{k}-\hat{\bm Z}_{k}\|\\
 \leq & (1-(1-\sqrt{1-\alpha})\gamma_k)\|{\bm E}_k(\theta_{k-1} {\bm W} - (\theta_{k-1}+1) {\bm I}_d) + \theta_{k-1}({\bm Z}_{k-1} ({\bm W}- {\bm I}_d) - \frac{\eta_{k-1}}{\theta_{k-1}}({\bm H}_{k-1}+\Xi_{k}))\| \\
 \leq & (1-(1-\sqrt{1-\alpha})\gamma_k)(1+\theta_{k-1} -\lambda_d \theta_{k-1})\|{\bm E}_k\| +(1-(1-\sqrt{1-\alpha})\gamma_k)C_0 \theta_{k-1}
\end{aligned}
\end{equation*}
where the second inequality follows from Jensen's inequality and the definition of a contractive compression operator, $C_0:= \sup_{k\geq 0}\| {\bm Z}_{k-1}({\bm W}-{\bm I}_d)-\frac{\eta_{k-1}}{\theta_{k-1}} ({\bm H}_{k-1}+\Xi_{k})\|< +\infty$ as Assumption \ref{asp:global_stability} holds, and $\lambda_d$ is the smallest eigenvalue of ${\bm W}$. In addition, one has
\begin{equation*}
(1-(1-\sqrt{1-\alpha})\gamma_k)(1+\theta_{k-1} -\lambda_d \theta_{k-1}) = 1 -\mu_k \gamma_k   
\end{equation*}
where $\mu_k := (1-\sqrt{1-\alpha})- \frac{\theta_{k-1}}{\gamma_k}(1-\lambda_d) -(1-\sqrt{1-\alpha})(1-\lambda_d)\theta_{k-1}>0$ as $k$ is sufficiently large.

Combining Lemma \ref{lem02} and \cite[Theorem 1]{ROBBINS1971233}, we know $\lim_{k\to \infty}\|{\bm E}_{k}\|$ exists and hence $\{{\bm E}_{k}\}$  is bounded almost surely. 
This result also yields that ${\bm Z}_{k}$ and ${\bm S}_{k}$ is bounded almost surely. By carrying out the calculation further, we obtain 
\begin{equation}\label{eq:incurr}
\begin{aligned}
\|{\bm E}_{k+1}\| & \leq  \gamma_k \mathbb{E}[\|{\bm S}_{k+1} \|| \mathcal{F}_{k}] + \gamma_k(\|{\bm S}_{k+1}\|- \mathbb{E}[\|{\bm S}_{k+1}\|| \mathcal{F}_{k}]) +(1-\gamma_k)\|{\bm Z}_{k}-\hat{\bm Z}_{k}\| \\
& \leq (1-(1-\sqrt{1-\alpha})\gamma_k)\|{\bm E}_k\| + (1-(1-\sqrt{1-\alpha})\gamma_k)C_1 \theta_{k-1} + \gamma_{k}(\|{\bm S}_{k+1}\|- \mathbb{E}[\|{\bm S}_{k+1}\|| \mathcal{F}_{k}])\\
\end{aligned}
\end{equation}
where $C_1:= \sup_{k\geq 0}\| \hat{\bm Z}_{k}({\bm W}-{\bm I}_d)-\frac{\eta_{k-1}}{\theta_{k-1}} ({\bm H}_{k-1}+\Xi_{k})\|< +\infty$. With the notations $\upsilon_{k}:= \|{\bm S}_{k}\|- \mathbb{E}[\|{\bm S}_{k}\|| \mathcal{F}_{k-1}]$ and $a:=1-\sqrt{1-\alpha}\in(0,1)$, one can recursively iterate \eqref{eq:incurr} to obtain 
\begin{equation*}\label{rewrite_Ek}
\|{\bm E}_{k+1}\| \leq \prod_{i=1}^{k}(1-a \gamma_i)\|{\bm E}_1\| + \sum_{i=1}^k \left(\prod_{j=i}^{k} (1-a\gamma_j)\right) C_1 \theta_{i-1} + \sum_{i=1}^{k-1} \left(\prod_{j=i+1}^{k} (1-a \gamma_j)\right) \gamma_i \upsilon_{i+1} + \gamma_k \upsilon_{k+1}.
\end{equation*}
Since $\{\upsilon_{k}\}$ is a uniformly bounded martingale difference sequence under Assumption \ref{asp:global_stability}, and $(\{\gamma_{k}\}, \{\theta_{k}\})$ are two time-scale sequences of order $o(1/\log(k))$, we can apply Lemmas \ref{lem01} and \ref{lem02} to derive

\begin{equation*}
\lim_{k \to +\infty} \gamma_{k}\upsilon_{k+1} + \sum_{i=1}^{k-1} \gamma_i \left( \prod_{j=i+1}^{k} (1 - a\gamma_j) \right) \upsilon_{i+1}=0, 
\end{equation*}
and
\begin{equation*}
\lim_{k\to +\infty} \sum_{i=1}^k \left(\prod_{j=i}^{k} (1-a\gamma_j)\right) \theta_{i-1} =0.
\end{equation*}
Together, these facts yield that 
\begin{equation*}
\lim_{k\to +\infty}\|{\bm E}_{k}\| =0.
\end{equation*}
This completes the proof. 
\end{proof}

\begin{prop}[Consensus: contractive compression]\label{prop:consensus_con}
Suppose Assumption \ref{Assumption_framework}, \ref{Assumption_compress} and \ref{asp:global_stability} hold. Then the sequence $\{{\bm Z}_{k}\}$ generated by \eqref{eq:contractive} satisfies 
\begin{equation*}
\lim_{k\to +\infty} \norm{{\bm Z}_{\perp,k}} = 0.
\end{equation*}
\end{prop}

\begin{proof}

To begin with, it is straightforward to check that
\begin{equation}\label{eq:flatten}
\begin{aligned}
    {\bm Z}_{\perp, k+1} & = [{\bm Z}_{k}((1-\theta_{k}){\bm I}_{d}+\theta_{k}{\bm W})]{\bm P}_{\perp} + \theta_{k}{\bm E}_{k+1}({\bm W}-{\bm I}_{d}){\bm P}_{\perp}- \eta_{k}({\bm H}_{k}+\Xi_{k+1}){\bm P}_{\perp}\\
    & = {\bm Z}_{\perp, k} ((1-\theta_{k}){\bm I}_{d}+\theta_{k}{\bm W}) + \theta_{k}{\bm E}_{k+1}({\bm W}-{\bm I}_{d}){\bm P}_{\perp} - \eta_{k}({\bm H}_{k}+\Xi_{k+1}){\bm P}_{\perp}\\
    & = {\bm Z}_{\perp, k} - \theta_{k} [{\bm Z}_{\perp, k} ({\bm I}_{d} - {\bm W}) + \frac{\eta_{k}}{\theta_{k}}({\bm H}_{k}+\Xi_{k+1}){\bm P}_{\perp} + {\bm E}_{k+1}({\bm W}-{\bm I}_{d}){\bm P}_{\perp}]. 
\end{aligned}
\end{equation}
Let $\Phi({\bm Z}):= {\bm Z}({\bm I}_{d}-{\bm W})$, $\upsilon_{k}:= 0$ and $\delta_{k}:= \frac{\eta_{k}}{\theta_{k}}(\|{\bm H}_{k}\|+\|\Xi_{k+1}\|) + 2\|{\bm E}_{k+1}\|$. Lemma \ref{lem:noise_to_0} together with Assumption \ref{Assumption_framework} implies $\lim_{k\to \infty}\delta_k =0$.  Then \eqref{eq:flatten} can be rewritten as
\begin{equation}\label{con_reformulate}
{\bm Z}_{\perp, k+1} \in {\bm Z}_{\perp, k} - \theta_{k} (\Phi^{\delta_{k}}({\bm Z}_{\perp, k}) + \upsilon_{k+1}), 
\end{equation}
Let $\psi({\bm Z}):= \frac{1}{2}\|{\bm Z}({\bm I}_{d}-{\bm W})^{\frac{1}{2}}\|_{\mathrm{F}}^2$, and $\mathcal{B}:= \{{\bm Z} | {\bm Z}({\bm I}_{d}-{\bm W})=0 \}$.
Analogously to the proof of Proposition \ref{prop:consensus_unb}, we can verify each condition in Assumption \ref{Assumption_stable_random} holds and apply Lemma \ref{lem_stable_random} to obtain desired result.

\end{proof}




\subsection{Global convergence and main results}\label{glo_con}
According to consensus properties shown in  Propositions  \ref{prop:consensus_unb} and \ref{prop:consensus_con}, we can deduce that the cluster points of the sequence $\{{\bm Z}_k\}$ generated by \eqref{eq:uncom} and \eqref{eq:contractive}  coincide with those of the sequence $\{ {\bm Z}_k \frac{{\bm 1}_d {\bm 1}_d\tp}{d}\}$. As a result, we proceed to analyze the convergence properties of $\{{\bm Z}_k \frac{{\bm 1}_d}{d}\}$. Proposition \ref{prop:ave_scheme} describes the relationship between $\frac{1}{d}{\bm Z}_{k}{\bm 1}_d$ and the averaged updated direction $\frac{1}{d} {\bm H}_{k}{\bm 1}_d$ via the set-valued mapping $\Phi$.

\begin{prop}
    \label{prop:ave_scheme}
    Suppose Assumption \ref{Assumption_framework} and \ref{asp:global_stability} hold. For any sequence $\{{\bm Z}_k\}$ generated by the sub-framework \eqref{eq:uncom} or \eqref{eq:contractive}, there exists a nonnegative diminishing sequence $\{\tilde{\epsilon}_k\}$ such that 
    \begin{equation}\label{eq:ave_scheme}    
   \frac{1}{d} {\bm H}_{k}{\bm 1}_d \in  \Phi^{\tilde{\epsilon}_k}(\frac{1}{d}{\bm Z}_{k}{\bm 1}_d).
    \end{equation} 
\end{prop}

\begin{proof}
Let $\epsilon^{\star}_k = \norm{{\bm Z}_{\perp, k}}$. From the definition of ${\bm P}_{\perp}$, it follows that $\norm{{\bm z}_{i,k} - \frac{1}{d}{\bm Z}_k {\bm 1}_d} \leq \norm{{\bm Z}_k{\bm P}_{\perp}} = \epsilon^{\star}_k$. For brevity, denote  $\ca{C}_k: = \conv( \frac{1}{d}\sum_{i = 1}^d \Phi_i^{\epsilon_k + \epsilon^{\star}_k}(\frac{1}{d}{\bm Z}_k {\bm 1}_d)  )$. According to Assumption \ref{Assumption_framework}(1), one attains that 
    \begin{equation}
        \frac{1}{d}{\bm H}_k {\bm 1}_d  \in \conv\left( \frac{1}{d}\sum_{i = 1}^d \Phi_i^{\epsilon_k}({\bm z}_{i,k})  \right)
        \subseteq \ca{C}_{k}.
    \end{equation}
It remains to show that there exists a nonnegative diminishing sequence $\{\tilde{\epsilon}_k\}$ such that 
    \begin{equation}
        \label{Eq_prop_ave_scheme_1}
        \ca{C}_{k} \subseteq \Phi^{\tilde{\epsilon}_k}(\frac{1}{d}{\bm Z}_k {\bm 1}_d). 
    \end{equation}
We proceed by contradiction. Suppose there exists $\delta_{\varepsilon} > 0$ and a subsequence $\{k_j\} \subset \bb{N}_+$ such that 
    \begin{equation}
        \label{Eq_prop_ave_scheme_0}
        \sup\left\{\mathrm{dist}\left( {\bm y},  \Phi^{\delta_{\varepsilon}}(\frac{1}{d}{\bm Z}_{k_j} {\bm 1}_d) \right): {\bm y} \in \ca{C}_{k_{j}} \right\} > 0. 
    \end{equation}
    
Since $\{{\bm Z}_k\}$ is uniformly bounded, without loss of generality, assume that $\{{\bm Z}_{k_j}\}$ converges to some $\tilde{{\bm Z}} \in \Rnd$. By the closedness of the graph and the local boundedness of $\Phi_i$, we obtain  
    \begin{equation*}\lim_{j\to +\infty}\sup\left\{ \mathrm{dist}\left({\bm y}, \Phi_i(\frac{\tilde{{\bm Z}}{\bm 1}_d}{d})  \right): {\bm y} \in  \Phi_i^{\epsilon_{k_j} + \epsilon^{\star}_{k_j}}(\frac{{\bm Z}_{k_j} {\bm 1}_d}{d})\right\} = 0.
    \end{equation*}
Based on Jensen’s inequality, it follows that 
    \begin{equation*}
    \lim_{j\to +\infty}\sup\left\{ \mathrm{dist}\left({\bm y}, \Phi(\frac{1}{d}\tilde{{\bm Z}}{\bm 1}_d)  \right): {\bm y} \in  \ca{C}_{k_{j}} \right\} = 0.
    \end{equation*}
    which contradicts \eqref{Eq_prop_ave_scheme_0}. This completes the proof. 

\end{proof}

\begin{lem}\label{claim3}
Suppose that step-sizes $\{\theta_{k}\}$ and $\{\eta_{k}\}$ satisfy \begin{equation*}
\sum_{k = 0}^{\infty} \eta_k = +\infty, \quad \lim_{k\to +\infty} \frac{\eta_k}{\theta_k} = 0, \quad \lim_{k\to +\infty} \theta_{k}\log(k) =0, \quad \lim_{k\to +\infty} \frac{\theta_{k}^2}{\eta_k}\log(k) =0, 
\end{equation*}
and let $\{\upsilon_{k}\}$ be a uniformly bounded martingale difference sequence with respect to filtration $\mathcal{F}_{k}:=\sigma(\{ {\bm \omega}_{j}, {\bm Z}_j, \Xi_{j}| j\leq k\})$. Define
$\hat{\upsilon}_{k}:=\frac{\eta_{k}}{\theta_{k}}{\upsilon}_{k}$. Then, for any $T>0$, we have
\begin{equation*}
    \lim_{s \to +\infty}\sup_{s\leq i\leq \Lambda_{\eta}(\lambda_{\eta}(s) + T)} \norm{ \sum_{k = s}^{i}\theta_k \hat{\upsilon}_{k+1} } =0. 
\end{equation*}
\end{lem}

The proof of Lemma \ref{claim3} is provided in Appendix \ref{appendix:claim3}. Based on Proposition \ref{prop:ave_scheme} and Lemma \ref{claim3}, we  establish recursion relations for $\{\frac{1}{d} {\bm Z}_{k}{\bm 1}_d\}$ and derive the global convergence of framework \eqref{Eq_Framework} in Theorem \ref{thm:convergence}. 

\begin{theo}\label{thm:convergence}
Suppose Assumption \ref{Assumption_framework} and \ref{asp:global_stability} hold, and let the sequence $\{{\bm Z}_k\}$ be generated by the sub-framework \eqref{eq:uncom} or \eqref{eq:contractive}. Then
\begin{equation*}
\lim_{k\to \infty} \mathrm{dist}({\bm Z}_{k}, \{{\bm Z}\in \mathbb{R}^{m \times d}: {\bm Z}= {\bm z} {\bm 1}^{\top}, {\bm z}\in \mathcal{A}\})=0,
\end{equation*}
Moreover, the sequence $\{\psi({\bm z}_{i,k})\}$ converges for each $i \in [d]$. 
\end{theo}

\begin{proof}
Combining Proposition \ref{prop:ave_scheme} with the update schemes in \eqref{eq:uncom} and \eqref{eq:contractive}, we derive the following recurrence relations for $\{\frac{1}{d}{\bm Z}_{k}{\bm 1}_d\}$:

\textbf{Unbiased compression:} 
\begin{equation}\label{eq:unb_theorem}
\begin{aligned}
\frac{1}{d}{\bm Z}_{k+1}{\bm 1}_d & \in \frac{1}{d} {\bm Z}_k{\bm 1}_d - \eta_k \Phi^{\tilde{\epsilon}_{k}} (\frac{1}{d}{\bm Z}_{k}{\bm 1}_d)  + \theta_{k}\frac{1}{d} {\bm E}_{k+1}({\bm W}- \Diag({\bm W})){\bm 1}_d -\eta_k\frac{1}{d}  \Xi_{k+1}{\bm 1}_d.\\
\end{aligned}
\end{equation}

\textbf{Contractive compression:} 
\begin{equation}\label{eq:con_theorem}
\begin{aligned}
\frac{1}{d}{\bm Z}_{k+1}{\bm 1}_d & \in \frac{1}{d}{\bm Z}_k {\bm 1}_d - \eta_k \Phi^{\tilde{\epsilon}_{k}} (\frac{1}{d}{\bm Z}_{k}{\bm 1}_d)  -\eta_k \frac{1}{d} \Xi_{k+1}{\bm 1}_d.\\
\end{aligned}
\end{equation}

For update scheme \eqref{eq:unb_theorem}, we define $\delta_k:= \tilde{\epsilon}_{k}$ and $\upsilon_{k+1}:= \frac{\theta_{k}}{\eta_k}\frac{1}{d}{\bm E}_{k+1}({\bm W}-\Diag({\bm W})){\bm 1}_d-\frac{1}{d}\Xi_{k+1}{\bm 1}_{d}$. Then \eqref{eq:unb_theorem} can be rewritten as
\begin{equation}\label{eq:rephrase_con}
\frac{1}{d} {\bm Z}_{k+1}{\bm 1}_d  \in \frac{1}{d} {\bm Z}_k{\bm 1}_d - \eta_k \Phi^{\delta_k} (\frac{1}{d}{\bm Z}_{k}{\bm 1}_d)  + \eta_{k}\upsilon_{k+1},    
\end{equation}
where $\delta_{k}$ tends to $0$ as $k\to \infty$, $\{\upsilon_{k}\}$ is a martingale difference sequence. We aim to apply Lemma \ref{lem_stable_random} again to derive the asymptotic convergence of iterates $\{\frac{1}{d}{\bm Z}_{k}{\bm 1}_d\}$. 

It is easy to see Assumption \ref{Assumption_stable_random}-(1) holds vacuously and Assumption \ref{Assumption_stable_random}-(2) is equivalent to Assumption \ref{Assumption_framework}-(2). However, since $\{\upsilon_{k}\}$ in \eqref{eq:rephrase_con} is not uniformly bounded, we need to check whether Assumption \ref{Assumption_stable_random}-(3) holds. Denoting $\hat{\upsilon}_{k}=\frac{\eta_{k}}{\theta_{k}}{\upsilon}_{k}$, Lemma \ref{claim3} gives us
\begin{equation*}
    \lim_{s\to +\infty} \sup_{s\leq i \leq \Lambda_{\eta}( \lambda_{\eta}(s)+T)} \norm{ \sum_{k = s}^{i}\eta_k \upsilon_{k+1}}=   \lim_{s \to +\infty}\sup_{s\leq i\leq \Lambda_{\eta}(\lambda_{\eta}(s) + T)} \norm{ \sum_{k = s}^{i}\theta_k \hat{\upsilon}_{k+1} } =0. 
\end{equation*}
Applying  Lemma \ref{lem_stable_random}, we conclude that any cluster point of $\{\frac{1}{d}{\bm Z}_k {\bm 1}_d\}$ lies in $\ca{A}$, and the sequence of function values $\{\psi(\frac{1}{d}{\bm Z}_k {\bm 1}_d)\}$ converges. Combining  Proposition \ref{prop:consensus_unb}, we know   any cluster point of $\{{\bm Z}_{k}\}$ coincides with a cluster point of $\{\frac{1}{d}{\bm Z}_k {\bm 1}_d {\bm 1}_{d}^{\top}\}$, and $\lim_{k\to \infty}\psi(z_{i,k}) = \lim_{k\to \infty}\psi(\frac{1}{d}{\bm Z}_k {\bm 1}_d))$.

For update scheme \eqref{eq:con_theorem}, we define $\delta_k:= \tilde{\epsilon}_{k}$ and $\upsilon_{k+1}:= \frac{1}{d}\Xi_{k+1}{\bm 1}_d$. Then \eqref{eq:con_theorem} can be reformulated as
\begin{equation}
\frac{1}{d} {\bm Z}_{k+1}{\bm 1}_d  \in \frac{1}{d} {\bm Z}_k{\bm 1}_d - \eta_k \Phi^{\delta_k} (\frac{1}{d}{\bm Z}_{k}{\bm 1}_d)  + \eta_{k}\upsilon_{k+1}.    
\end{equation}
Similar to the proof of \eqref{eq:uncom}, one can easily verify that Assumption \ref{Assumption_stable_random} holds. Combining Lemma \ref{lem_stable_random} with Proposition \ref{prop:consensus_con},  we deduce that any cluster point of $\{{\bm Z}_{k}\}$ lies in $\{{\bm Z}\in \mathbb{R}^{m \times d}: {\bm Z}= {\bm z} {\bm 1}^{\top}, {\bm z}\in \mathcal{A}\}$ and the sequence $\{\psi({\bm z}_{k,i})\}$ converges for each $i \in [d]$. The proof is completed.

\end{proof}



\section{Developing Decentralized Stochastic Subgradient-type Methods with Communication Compression and Convergence Guarantees}\label{sec:methods}

In this section, we demonstrate that framework \eqref{Eq_Framework} encloses a wide range of decentralized stochastic subgradient-type methods with communication compression in nonsmooth optimization. Some of the methods are nonsmooth extensions of the existing approaches, while others are newly developed based on our framework \eqref{Eq_Framework}. More importantly, we establish convergence results for these decentralized methods for the minimization of nonsmooth definable functions, with applications to the training of nonsmooth neural networks.

\subsection{Decentralized SGD-type methods with communication compression}\label{sec:SGD}

\paragraph{Stochastic nonsmooth extension of QDGD.}

QDGD \cite{reisizadeh2019exact} is a decentralized SGD method with unbiased compression originally designed for smooth optimization. When the objective function is nonsmooth and the evaluation noise is present, we replace $\nabla f_{i}({\bm x}_{i,k})$ with a stochastic subgradient $ {\bm g}_{i, k} \in D_{F_i(\cdot,\zeta_{i,k+1})}({\bm x}_{i,k})$, and employ diminishing sequences $\{\theta_{k}\}$ and $\{\eta_{k}\}$ instead of constant $\theta$ and $\eta$. The stochastic nonsmooth extension of QDGD can then be compactly written as 
\begin{equation*}
\manualeqtag{QSDGD+}{unb:dsgd}
\left\{
\begin{aligned}
{\bm G}_{k} &= [{\bm g}_{1,k}, \ldots, {\bm g}_{d,k}] \in [D_{F_1(\cdot,\zeta_{i,k+1})}({\bm x}_{i,k}), \ldots, D_{F_d(\cdot,\zeta_{i,k+1})}({\bm x}_{i,k})], \\
{\bm E}_{k+1} & =C({\bm X}_{k}, {\bm \omega}_{k+1}) - {\bm X}_{k},\\
{\bm X}_{k+1} & = (1-\theta_k) {\bm X}_k + \theta_k {\bm X}_k {\bm W} + \theta_k {\bm E}_{k+1}({\bm W}-\Diag ({\bm W})) - \eta_k {\bm G}_k. \\
\end{aligned}
\right.
\end{equation*}
Here, $C$ is an unbiased compression operator. $\{\theta_{k}\}$ is a diminishing sequence of step-sizes with respect to the local average.
\paragraph{Nonsmooth extension of CHOCO-SGD.} 
CHOCO-SGD \cite{koloskova2019decentralized, koloskova2019decentralized2} is a decentralized SGD method that combines contractive compression with error compensation. We introduce a nonsmooth extension of CHOCO-SGD, given by the following update scheme:
\begin{equation*}
\manualeqtag{CHOCO-SGD+}{con:dsgd}
\left\{
\begin{aligned}
{\bm G}_{k} &= [{\bm g}_{1,k}, \ldots, {\bm g}_{d,k}] \in [D_{F_1(\cdot,\zeta_{i,k+1})}({\bm x}_{i,k}), \ldots, D_{F_d(\cdot,\zeta_{i,k+1})}({\bm x}_{i,k})], \\
  \hat{{\bm X}}_{k+1} & = \hat{{\bm X}}_{k} + \gamma_k C({\bm X}_{k}-\hat{{\bm X}}_{k}, {\bm \omega}_{k+1}),\\
  {\bm X}_{k+1} & =  {\bm X}_k  + \theta_k \hat{{\bm X}}_{k+1} ({\bm W}-{\bm I}_{d}) - \eta_k {\bm G}_k.\\
\end{aligned}
\right.
\end{equation*}
Here, $C$ is a contractive compression operator. The two-timescale sequences $\{\gamma_k\}$ and $\{\theta_{k}\}$  replace the constant $\theta$ and $\gamma =1$ used in CHOCO-SGD.  With the notation ${\bm E}_{k+1}:= \gamma_k C({\bm X}_{k}- \hat{{\bm X}_{k}}, {\bm \omega}_{k+1})-({\bm X}_{k}- \hat{{\bm X}_{k}})$, we can reformulate \eqref{con:dsgd} as 
\begin{equation*}
\left\{
\begin{aligned}
{\bm G}_{k} &= [{\bm g}_{1,k}, \ldots, {\bm g}_{d,k}] \in [D_{F_1(\cdot,\zeta_{i,k+1})}({\bm x}_{i,k}), \ldots, D_{F_d(\cdot,\zeta_{i,k+1})}({\bm x}_{i,k})], \\
  {\bm X}_{k+1}&  = (1-\theta_k) {\bm X}_k + \theta_k {\bm X}_k {\bm W} + \theta_k {\bm E}_{k+1}({\bm W}-{\bm I}_{d}) - \eta_k {\bm G}_k.
\end{aligned}
\right.
\end{equation*}

\paragraph{Nonsmooth extension of BEER and C-GT.}

BEER \cite{zhao2022beer} and C-GT \cite{liao2022compressed} utilize an auxiliary variable ${\bm Y}_{k}$ to track the stochastic gradient of the global objective function, and perform contractive compression-based communication on both 
${\bm X}_{k}$ and ${\bm Y}_{k}$ simultaneously. We consider a nonsmooth extension of BEER:


\begin{equation*}
\manualeqtag{BEER+}{con:dsgt}
\left\{
\begin{aligned}
{\bm G}_{k} &= [{\bm g}_{1,k}, \ldots, {\bm g}_{d,k}] \in [D_{F_1(\cdot,\zeta_{i,k+1})}({\bm x}_{i,k}), \ldots, D_{F_d(\cdot,\zeta_{i,k+1})}({\bm x}_{i,k})], \\
{\bm X}_{k+1} & = {\bm X}_k +  \theta_k \hat{\bm X}_{k+1}({\bm W}-{\bm I}_{d}) - \eta_k {\bm Y}_k,  \\
{\bm Y}_{k+1} & = {\bm Y}_k + \theta_k \hat{\bm Y}_{k+1}({\bm W}-{\bm I}_{d}) + {\bm G}_{k+1} -{\bm G}_{k}, \\   
 \hat{\bm X}_{k+1} & =  \hat{\bm X}_{k} + \gamma_{k}C({\bm X}_{k}-\hat{\bm X}_{k}, {\bm \omega}_{k+1}), \\
 \hat{\bm Y}_{k+1} & =  \hat{\bm Y}_{k} + \gamma_{k}C({\bm Y}_{k}-\hat{\bm Y}_{k}, {\bm \omega}_{k+1}), \\
{\bm Y}_{0} & = {\bm G}_0,\\
\end{aligned}
\right.
\end{equation*}
where $C$ is a contractive compression operator. The main differences between BEER and BEER+ are the choice of ${\bm G}_{k}$ and the use of slowly diminishing and two-timescale sequences $\{\theta_{k}\}, \{\gamma_{k}\}$ instead of fixed values. With the notation ${\bm E}_{k+1}:= \gamma_k C({\bm X}_{k}-\hat{{\bm X}}_{k}, {\bm \omega}_{k+1})-({\bm X}_{k}-\hat{{\bm X}}_{k})$, the update scheme of $\{{\bm X}_{k}\}$ becomes 
\begin{equation*}
{\bm X}_{k+1}  = (1-\theta_{k}){\bm X}_k +  \theta_k {\bm X}_{k}{\bm W} + \theta_{k} {\bm E}_{k+1}({\bm W}-{\bm I}_{d}) - \eta_k {\bm Y}_{k}.
\end{equation*}
In addition, the iterates of the nonsmooth extension of C-GT are given as follows:
\begin{equation*}
\manualeqtag{C-GT+}{con:dsgt(cgt)}
\left\{
\begin{aligned}
{\bm G}_{k} &= [{\bm g}_{1,k}, \ldots, {\bm g}_{d,k}] \in [D_{F_1(\cdot,\zeta_{i,k+1})}({\bm x}_{i,k}), \ldots, D_{F_d(\cdot,\zeta_{i,k+1})}({\bm x}_{i,k})], \\
{\bm X}_{k+1} & = {\bm X}_k +  \theta_k \tilde{\bm X}_{k+1}({\bm W}-{\bm I}_{d}) - \eta_k {\bm Y}_k,  \\
{\bm Y}_{k+1} & = {\bm Y}_k + \theta_k \tilde{\bm Y}_{k+1}({\bm W}-{\bm I}_{d}) + {\bm G}_{k+1} -{\bm G}_{k}, \\   
 \tilde{\bm X}_{k+1} & =  \hat{\bm X}_{k} + \gamma_{k}C({\bm X}_{k}-\hat{\bm X}_{k}, {\bm \omega}_{k+1}), \\
 \hat{\bm X}_{k+1} & = (1-\alpha_x)  \hat{\bm X}_{k}  +\alpha_x \tilde{\bm X}_{k+1},  \\
 \tilde{\bm Y}_{k+1} & =  \hat{\bm Y}_{k} + \gamma_{k}C({\bm Y}_{k}-\hat{\bm Y}_{k}, {\bm \omega}_{k+1}), \\
  \hat{\bm Y}_{k+1} & = (1-\alpha_y) \hat{\bm Y}_{k}  +  \alpha_y \tilde{\bm Y}_{k+1},  \\
  {\bm Y}_{0} & = {\bm G}_0,\\
\end{aligned}
\right.
\end{equation*}
where $\alpha_x \in (0,1]$. Compared to BEER+, C-GT+ includes an additional step of weighted average between the current $\tilde{\bm X}_{k+1}$ and the previous $\hat{\bm X}_{k}$. With the notation ${\bm E}_{k+1}:= \gamma_k C({\bm X}_{k}-\hat{{\bm X}}_{k}, {\bm \omega}_{k+1})-({\bm X}_{k}-\hat{{\bm X}}_{k})$, the update scheme of $\{(\hat{\bm X}_{k}, {\bm X}_{k})\}$ can be rewritten as

\begin{equation*}
\left\{
\begin{aligned}
\hat{\bm X}_{k+1} & = \hat{\bm X}_{k} +\alpha_{x} \gamma_{k} C({\bm X}_{k} -\hat{\bm X}_{k}, {\bm \omega}_{k+1}),\\ 
{\bm X}_{k+1} &  = (1-\theta_{k}){\bm X}_k +  \theta_k {\bm X}_{k}{\bm W} + \theta_{k} {\bm E}_{k+1}({\bm W}-{\bm I}_{d}) - \eta_k {\bm Y}_{k}.\\
\end{aligned}
\right.
\end{equation*}

\paragraph{Decentralized heavy-ball SGD with Nesterov momentum and communication compression.}

Heavy-ball SGD \cite{polyak1964some} accelerates the descent by introducing momentum, which helps dampen oscillations in regions with small gradients or high noise. Nesterov momentum \cite{nesterov1983method} further provides a ``look-ahead'' mechanism in the update rule, leading to faster convergence in many optimization problems. Motivated by the nonsmooth heavy-ball SGD proposed in \cite{le2024nonsmooth, xiao2023convergence}, we integrate unbiased and contractive compression with nonsmooth heavy-ball SGD equipped with Nesterov momentum. This leads to two novel methods, which are formalized as the following update schemes: \eqref{unb:dsm} and \eqref{con:dsm}.

\begin{equation*}
\manualeqtag{DSM-Unb}{unb:dsm}
\left\{
\begin{aligned}
{\bm G}_{k} &= [{\bm g}_{1,k}, \ldots, {\bm g}_{d,k}] \in [D_{F_1(\cdot,\zeta_{i,k+1})}({\bm x}_{i,k}), \ldots, D_{F_d(\cdot,\zeta_{i,k+1})}({\bm x}_{i,k})], \\
{\bm E}^{x}_{k+1} & =C({\bm X}_{k}, {\bm \omega}_{k+1})-{\bm X}_{k},\\
{\bm E}^{y}_{k+1}& =C({\bm Y}_{k}, {\bm \omega}_{k+1})-{\bm Y}_{k},\\
{\bm X}_{k+1} & = (1-\theta_k) {\bm X}_k + \theta_k {\bm X}_k {\bm W} + \theta_k {\bm E}^{x}_{k+1}({\bm W}-\Diag ({\bm W})) - \eta_k ({\bm Y}_k +\rho {\bm G}_{k}),  \\
{\bm Y}_{k+1} & = (1-\tau \eta_k)[(1-\theta_k) {\bm Y}_k + \theta_k {\bm Y}_k {\bm W} + \theta_k {\bm E}^{y}_{k+1}({\bm W}-\Diag ({\bm W}))] + \tau \eta_k {\bm G}_{k+1}.\\
\end{aligned}
\right.
\end{equation*}
Here, $C$ is an unbiased compression operator. $\{\theta_{k}\}$ stands for diminishing step-sizes corresponding to the local average. In the update step of ${\bm Y}_{k+1}$, the local aggregation based on unbiased compression is used as an inertial direction, while ${\bm G}_{k+1}$ serves as a descent direction. $\tau >0$ is the heavy-ball momentum parameter and $\rho \geq 0$ is the Nesterov momentum parameter.

\begin{equation*}
    \manualeqtag{DSM-Con}{con:dsm}
        \left\{
      \begin{aligned}
      {\bm G}_{k} &= [{\bm g}_{1,k}, \ldots, {\bm g}_{d,k}] \in [D_{F_1(\cdot,\zeta_{i,k+1})}({\bm x}_{i,k}), \ldots, D_{F_d(\cdot,\zeta_{i,k+1})}({\bm x}_{i,k})], \\
      \hat{\bm X}_{k+1}&= \hat{\bm X}_{k}+\gamma_{k}C({\bm X}_{k}-\hat{\bm X}_{k}, {\bm \omega}_{k+1}),\\
      \hat{\bm Y}_{k+1}& = \hat{\bm Y}_{k}+\gamma_{k}C({\bm Y}_{k}-\hat{\bm Y}_{k}, {\bm \omega}_{k+1}),\\
        {\bm X}_{k+1} & =  {\bm X}_k + \theta_{k} \hat{{\bm X}}_{k+1}({\bm W}-{\bm I}_{d}) - \eta_k ({\bm Y}_k + \rho {\bm G}_k),\\
        {\bm Y}_{k+1} & = (1-\tau \eta_{k})[{\bm Y}_k + \theta_{k} \hat{{\bm Y}}_{k+1}({\bm W}-{\bm I}_{d})] + \tau \eta_{k}{\bm G}_{k+1}.\\
      \end{aligned}
      \right.
  \end{equation*}
Here, $C$ is a contractive compression operator. $\{\theta_{k}\}, \{\gamma_{k}\}$ are  two-timescale diminishing 
 sequences of step-sizes. The sequence $\{{\bm Y}_{k}\}$ is updated by a local aggregation with contractive compression combined with a heavy-ball momentum step. Parameters $\tau, \rho$ are defined analogously to \eqref{unb:dsm}.


\paragraph{Decentralized sign-regularized SGD with communication compression.}

SignSGD \cite{bernstein2018signsgd, bernstein2018signsgd2} is a notable variant of SGD. It replaces the full gradient with its sign, aiming to normalize the update direction and reduce communication costs in distributed optimization. Motivated by DSignSGD proposed in \cite{zhang2024decentralized}, we present two decentralized sign-regularized SGD methods with communication compression.

\begin{equation*}
\manualeqtag{SignDSGD-Unb}{unb:signdsgd}
\left\{
\begin{aligned}
 {\bm G}_{k} &= [{\bm g}_{1,k}, \ldots, {\bm g}_{d,k}] \in [D_{F_1(\cdot,\zeta_{i,k+1})}({\bm x}_{i,k}), \ldots, D_{F_d(\cdot,\zeta_{i,k+1})}({\bm x}_{i,k})], \\
{\bm E}^{x}_{k+1} & =C({\bm X}_{k}, {\bm \omega}_{k+1})-{\bm X}_{k},\\
{\bm E}^{y}_{k+1}& =C({\bm Y}_{k}, {\bm \omega}_{k+1})-{\bm Y}_{k}, \\
{\bm X}_{k+1} & = (1-\theta_k) {\bm X}_k + \theta_k {\bm X}_k {\bm W} + \theta_k {\bm E}^{x}_{k+1}({\bm W}-\Diag ({\bm W})) - \eta_k \mathrm{sign}({\bm Y}_k+ \rho {\bm G}_{k}),  \\
{\bm Y}_{k+1} & = (1-\tau \eta_k)[(1-\theta_k) {\bm Y}_k + \theta_k {\bm Y}_k {\bm W} + \theta_k {\bm E}^{y}_{k+1}({\bm W}-\Diag ({\bm W}))] + \tau \eta_k {\bm G}_{k+1}. \\
\end{aligned}
\right.
\end{equation*}
Here, $C$ is an unbiased compression operator, and the sign map serves as a regularizer of the Nesterov momentum term ${\bm Y}_k+ \rho {\bm G}_{k}$. Parameters $\tau, \rho$, and $\{\theta_k\}$ are defined analogously to \eqref{unb:dsm}.

\begin{equation*}
    \manualeqtag{SignDSGD-Con}{con:signdsgd}
    \left\{
    \begin{aligned}
    {\bm G}_{k} &= [{\bm g}_{1,k}, \ldots, {\bm g}_{d,k}] \in [D_{F_1(\cdot,\zeta_{i,k+1})}({\bm x}_{i,k}), \ldots, D_{F_d(\cdot,\zeta_{i,k+1})}({\bm x}_{i,k})], \\
    \hat{\bm X}_{k+1}& = \hat{\bm X}_{k}+\gamma_{k}C({\bm X}_{k}-\hat{\bm X}_{k}, {\bm \omega}_{k+1}),\\
    \hat{\bm Y}_{k+1}& = \hat{\bm Y}_{k}+\gamma_{k}C({\bm Y}_{k}-\hat{\bm Y}_{k}, {\bm \omega}_{k+1}),\\
      {\bm X}_{k+1} & =  {\bm X}_k + \theta_{k} \hat{{\bm X}}_{k+1}({\bm W}-{\bm I}_{d}) - \eta_k \mathrm{sign}({\bm Y}_k + \rho {\bm G}_k),\\
      {\bm Y}_{k+1} & = (1-\tau \eta_{k})[{\bm Y}_k + \theta_{k} \hat{{\bm Y}}_{k}({\bm W}-{\bm I}_{d})] + \tau \eta_{k}{\bm G}_{k+1}.\\
    \end{aligned}
    \right.
\end{equation*}
Here, $C$ is a contractive compression operator. Step-sizes $\{\theta_k\}, \{\gamma_k\}$,  and parameters $\tau, \rho$ are defined similarly to \eqref{con:dsm}.

\paragraph{Decentralized stochastic momentum tracking with  contractive compression.}

Based on our framework \eqref{eq:contractive}, we modify compression term and the update rule for the momentum variables in DoCoM \cite{yau2022docom}, and propose a novel decentralized momentum-tracking method with contractive compression.

\begin{equation*}
\manualeqtag{DSGTM-Con}{con:dsgtm}
  \left\{
        \begin{aligned}
        {\bm G}_{k} &= [{\bm g}_{1,k}, \ldots, {\bm g}_{d,k}] \in [D_{F_1(\cdot,\zeta_{i,k+1})}({\bm x}_{i,k}), \ldots, D_{F_d(\cdot,\zeta_{i,k+1})}({\bm x}_{i,k})], \\
        \hat{{\bm X}}_{k+1} & = \hat{{\bm X}}_{k} + \gamma_k C({\bm X}_{k}- \hat{{\bm X}}_{k}, {\bm \omega}_{k+1}),\\
          {\bm X}_{k+1}  & = {\bm X}_k + \theta_k \hat{\bm X}_{k+1} ({\bm W}-{\bm I}_d)  - \eta_k {\bm Y}_k,\\
          {\bm V}_{k+1} & = (1-\tau \eta_{k}){\bm V}_{k} + \tau \eta_{k}  {\bm G}_{k+1},\\ 
          \hat{{\bm Y}}_{k+1} & = \hat{{\bm Y}}_{k} + \gamma_k C({\bm Y}_{k}- \hat{{\bm Y}}_{k}, {\bm \omega}_{k+1}),\\
          {\bm Y}_{k+1} & = {\bm Y}_k + \theta_k \hat{\bm Y}_{k+1} ({\bm W}-{\bm I}_d)+ {\bm V}_{k+1} -{\bm V}_{k}, \\ 
          {\bm V}_0 & = {\bm G}_0.\\
        \end{aligned}
  \right.
    \end{equation*}
Here, $C$ is a contractive compression operator. ${\bm V}_k$ represents the collection of momentum variables, and ${\bm Y}_{k}$ serves as the collection of momentum-tracking variables incorporating compression-based aggregation. $\{\theta_{k}\}, \{\gamma_{k}\}$ are two-timescale diminishing sequences of step-sizes.

\subsection{Convergence Guarantees}

In this part, we demonstrate that all methods discussed in Section \ref{sec:SGD} fit into our framework \eqref{Eq_Framework}, and hence inherit global convergence guarantees to $D_{f}$-critical points in the minimization of nonsmooth definable functions. Throughout this section, we impose several assumptions on the original problem \eqref{Prob_DOP} and the  methods in Section \ref{sec:SGD}. 

\begin{assumpt}\label{Assumption_obj}
\begin{enumerate}

\item[(1)] For each $i \in [d]$, $\zeta_{i}$ is drawn randomly and independently from  $\mathcal{P}_i$,  $F(\cdot, \zeta_{i})$ is definable and admits a definable conservative field $D_{F_i(\cdot, \zeta_{i})}$. 

\item[(2)] The summation function $f({\bm x})$ is proper. 

\item[(3)] The step-sizes  $\{\eta_{k}\}$ and $\{\theta_{k}\}$ satisfy 
\begin{equation*}
\sum_{i = 0}^{\infty} \eta_k = +\infty, \quad \lim_{k\to +\infty} \frac{\eta_k}{\theta_k} = 0, \quad \lim_{k\to +\infty} \theta_{k}\log(k) =0, \quad \lim_{k\to +\infty} \frac{\theta_{k}^2}{\eta_k}\log(k) =0. 
\end{equation*}
Moreover, the sequence $\{\gamma_k\}$ is diminishing and satisfies
\begin{equation*}
\gamma_0 =1, \quad \lim_{k\to \infty}\gamma_k \log(k) =0, \quad \lim_{k\to \infty}\frac{\theta_{k}}{\gamma_k} =0.  
\end{equation*}
\end{enumerate}
\end{assumpt}

\begin{rmk}
As illustrated in Remark \ref{rmk:01}, definable functions are sufficiently general to enclose the loss functions of nearly all neural networks. Furthermore, \cite{bolte2021conservative} states that the result yielded by AD algorithms is contained in a definable conservative field of definable loss function. Hence, Assumption \ref{Assumption_obj}-(1) is reasonable and mild.     
\end{rmk}

Based on \cite[Corollary 4]{bolte2021conservative},  conservativity remains invariant under both expectation and summation, and definability is also invariant under the same operations. As a result, we can define the conservative field for each $f_{i}$ as 
\begin{equation}\label{eq:defin_cons}
    D_{f_i}({\bm x}) := \bb{E}_{\zeta_{i}\sim \mathcal{P}_{i}}  D_{F_i(\cdot, \zeta_{i})}({\bm x}), 
\end{equation}
and a definable and convex conservative field for $f$ as 
\begin{equation*}
    D_f({\bm x}) :=  \conv\left(\frac{1}{d} \sum_{i = 1}^d D_{f_i}({\bm x}) \right). 
\end{equation*}

\begin{table}
\centering
\fontsize{9}{13}\selectfont
\begin{tabular}{|c|c|c|c|}
\hline
Methods & ${\bm Z}_{k}$ &  ${\bm H}_{k}$ &${\bm E}_{k+1}$ \\
\hline
QSDGD+  & ${\bm X}_{k}$ & $\bb{E}[{\bm G}_{k}|\mathcal{F}_{k}]$ & $C({\bm X}_{k}, {\bm \omega}_{k+1})- {\bm X}_{k}$ \\
\hline
DSM-Unb & $\left[\begin{smallmatrix}{\bm X}_{k}\\
{\bm Y}_{k}\end{smallmatrix}\right]$  & $\bb{E}\left[\begin{smallmatrix}
        {\bm Y}_k + \rho {\bm G}_{k}\\
        \tau ((1-\theta_{k}){\bm Y}_{k} +\theta_{k} {\bm Y}_{k}{\bm W} )- \tau {\bm G}_{k+1}
    \end{smallmatrix} |\ca{F}_k\right]$ & $\left[\begin{smallmatrix}C({\bm X}_{k}, {\bm \omega}_{k+1})- {\bm X}_{k}\\
(1-\tau\eta_{k})(C({\bm Y}_{k}, {\bm \omega}_{k+1})- {\bm Y}_{k})\end{smallmatrix}\right]$  \\
\hline
SignDSGD-Unb & $\left[\begin{smallmatrix}{\bm X}_{k}\\
{\bm Y}_{k}\end{smallmatrix}\right]$  & $\bb{E}\left[\begin{smallmatrix}
      \mathrm{sign}({\bm Y}_k+ \rho {\bm G}_{k}) \\
        \tau ((1-\theta_{k}){\bm Y}_{k} +\theta_{k} {\bm Y}_{k}{\bm W} )- \tau {\bm G}_{k+1}
    \end{smallmatrix} |\ca{F}_k\right]$ & $\left[\begin{smallmatrix}C({\bm X}_{k}, {\bm \omega}_{k+1})- {\bm X}_{k}\\
(1-\tau\eta_{k})(C({\bm Y}_{k}, {\bm \omega}_{k+1})- {\bm Y}_{k})\end{smallmatrix}\right]$\\
\hline
CHOCO-SGD+  & ${\bm X}_{k}$ & $\bb{E}[{\bm G}_{k}|\mathcal{F}_{k}]$ & $\gamma_k C({\bm X}_{k}-\hat{{\bm X}}_{k}, {\bm \omega}_{k+1})-({\bm X}_{k}-\hat{{\bm X}}_{k})$  \\
\hline
BEER+  & ${\bm X}_{k}$ & $\bb{E}[{\bm Y}_{k}|\mathcal{F}_{k}]$   & $\gamma_k C({\bm X}_{k}-\hat{{\bm X}}_{k}, {\bm \omega}_{k+1})-({\bm X}_{k}-\hat{{\bm X}}_{k})$\\
\hline
C-GT+  & ${\bm X}_{k}$ & $\bb{E}[{\bm Y}_{k}|\mathcal{F}_{k}]$   & $\gamma_k C({\bm X}_{k}-\hat{{\bm X}}_{k}, {\bm \omega}_{k+1})-({\bm X}_{k}-\hat{{\bm X}}_{k})$\\
\hline
DSM-Con & $\left[\begin{smallmatrix}{\bm X}_{k}\\
{\bm Y}_{k}\end{smallmatrix}\right]$  & $\bb{E}\left[\begin{smallmatrix}
        {\bm Y}_k + \rho {\bm G}_{k}\\
        \tau ({\bm Y}_{k} +\theta_{k} \hat{\bm Y}_{k+1}({\bm W}-{\bm I}_d))- \tau {\bm G}_{k+1}
    \end{smallmatrix} |\ca{F}_k\right]$ & $\left[\begin{smallmatrix}\gamma_k C({\bm X}_{k}-\hat{{\bm X}}_{k}, {\bm \omega}_{k+1})-({\bm X}_{k}-\hat{{\bm X}}_{k})\\
\gamma_k C({\bm Y}_{k}-\hat{{\bm Y}}_{k}, {\bm \omega}_{k+1})-({\bm Y}_{k}-\hat{{\bm X}}_{k})\end{smallmatrix}\right]$ \\
\hline
SignDSGD-Con & $\left[\begin{smallmatrix}{\bm X}_{k}\\
{\bm Y}_{k}\end{smallmatrix}\right]$  & $\bb{E}\left[\begin{smallmatrix}
      \mathrm{sign}({\bm Y}_k+ \rho {\bm G}_{k}) \\
        \tau ({\bm Y}_{k} +\theta_{k} \hat{\bm Y}_{k+1}({\bm W}-{\bm I}_d))- \tau {\bm G}_{k+1}
    \end{smallmatrix} |\ca{F}_k\right]$ & $\left[\begin{smallmatrix}\gamma_k C({\bm X}_{k}-\hat{{\bm X}}_{k}, {\bm \omega}_{k+1})-({\bm X}_{k}-\hat{{\bm X}}_{k})\\
\gamma_k C({\bm Y}_{k}-\hat{{\bm Y}}_{k}, {\bm \omega}_{k+1})-({\bm Y}_{k}-\hat{{\bm Y}}_{k})\end{smallmatrix}\right]$    \\
\hline
DSGTM-Con & $\left[\begin{smallmatrix}{\bm X}_{k}\\
{\bm Y}_{k}\end{smallmatrix}\right]$  & $\bb{E}\left[\begin{smallmatrix}
        {\bm Y}_k\\
        \tau {\bm V}_{k} - \tau {\bm G}_{k+1}
    \end{smallmatrix} |\ca{F}_k\right]$ & $\left[\begin{smallmatrix}\gamma_k C({\bm X}_{k}-\hat{{\bm X}}_{k}, {\bm \omega}_{k+1})-({\bm X}_{k}-\hat{{\bm X}}_{k})\\
\gamma_k C({\bm Y}_{k}-\hat{{\bm Y}}_{k}, {\bm \omega}_{k+1})-({\bm Y}_{k}-\hat{{\bm Y}}_{k})\end{smallmatrix}\right]$    \\
\hline
\end{tabular}
\caption{Specific choices of ${\bm Z}_{k}$, ${\bm H}_{k}$, ${\bm E}_{k+1}$ for each method in Section \ref{sec:SGD}.}
\label{tab:1}
\end{table}

\begin{lem}[Theorem 5 and 9 in \cite{bolte2021conservative}]\label{lem-app}
Let Assumption \ref{Assumption_obj} hold. Then, 
\begin{enumerate}
\item[(1)] 
$f$ is a Lyapunov function for $D_{f}({\bm x})$, whose stable set is $\mathcal{A}:=\{{\bm x}\in \mathbb{R}^n : {\bm 0} \in D_{f}({\bm x})\}$.
\item[(2)] 
  $\{f({\bm x}): {\bm 0}\in D_{f}({\bm x})\}$ is a finite set. 
\end{enumerate}

\end{lem}

\begin{lem}[Proposition 4.5 in \cite{xiao2023convergence}]\label{lem-app2}
Let Assumption \ref{Assumption_obj} hold. Then,
\begin{enumerate}
\item[(1)]  $\psi({\bm x}, {\bm y}):= f({\bm x}) +\frac{1}{2\tau}\|{\bm y}\|^2$ is a locally Lipschitz Lyapunov function for the differential inclusion 
$\frac{\mathrm{d} ({\bm x}, {\bm y})}{\mathrm{d} t} \in -\conv ( \frac{1}{d}\sum_{i=1}^{d}   [\begin{smallmatrix} {\bm y}+\rho D_{f_i}({\bm x})\\
\tau {\bm y} - \tau D_{f_i}({\bm x})   \end{smallmatrix}] )$ with stable set $\mathcal{A}:=\{({\bm x}, {\bm y}): {\bm 0} \in D_f({\bm x}), {\bm y} =0\}$. 
\item[(2)]   $\psi({\bm x}, {\bm y}):= f({\bm x}) +\frac{1}{\tau}\|{\bm y}\|_1$ is a locally Lipschitz Lyapunov function for
the differential inclusion 
$\frac{\mathrm{d} ({\bm x}, {\bm y})}{\mathrm{d} t} \in -\conv ( \frac{1}{d}\sum_{i=1}^{d} [\begin{smallmatrix} \mathrm{sign}({\bm y}+\rho D_{f_i}({\bm x}))\\
\tau {\bm y} - \tau D_{f_i}({\bm x})   \end{smallmatrix}])$ with stable set $\mathcal{A}:=\{({\bm x}, {\bm y}): {\bm 0} \in D_f({\bm x}), {\bm y} =0\}$. 
\end{enumerate}   
\end{lem}

Lemmas \ref{lem-app} and \ref{lem-app2} present the Lyapunov functions corresponding to different differential inclusions, which is a key to verifying that the methods in Section \ref{sec:SGD} satisfy Assumption \ref{Assumption_framework}.


\begin{prop}\label{prop:qdsgd}
Suppose Assumption \ref{Assumption_obj} holds,  then we have
\begin{enumerate}
\item \eqref{unb:dsgd}, \eqref{unb:dsm} and \eqref{unb:signdsgd} fit into sub-framework \eqref{eq:uncom}.
\item \eqref{con:dsgd}, \eqref{con:dsgt}, \eqref{con:dsgt(cgt)}, \eqref{con:dsm}, \eqref{con:signdsgd} and \eqref{con:dsgtm} fit into sub-framework \eqref{eq:contractive}.
\item All the aforementioned methods satisfy Assumption \ref{Assumption_framework}.
\end{enumerate}
\end{prop}

In Table \ref{tab:1}, we show the specific choices of ${\bm Z}_{k}$, ${\bm H}_{k}$, ${\bm E}_{k}$  for each method in Section \ref{sec:SGD}, thereby verifying the first two items of Proposition \ref{prop:qdsgd}. In Table \ref{tab:2}, we illustrate the specific choices of $\Phi_{i}$, $\psi$, $\mathcal{A}$ and $\epsilon_k$ for each method in Section \ref{sec:SGD}. Combining with the local boundedness of $D_{F_i(\cdot, \zeta_{i})}$,  Lemma \ref{lem-app},  and Lemma \ref{lem-app2}, we can directly verify that all the methods satisfy Assumption \ref{Assumption_framework}.

\begin{table}
\centering
\fontsize{9}{13}\selectfont
\begin{tabular}{|c|c|c|c|c|}
\hline
Methods  & $ \Phi_{i} $ & $\psi$  &  $\mathcal{A}$ & $\epsilon_{k}$ \\
\hline
QSDGD+  & $D_{f_i}$  & $f$ & $\{{\bm x} \in \Rn: {\bm 0} \in D_f({\bm x})\}$ & 0  \\
\hline
DSM-Unb & $\left[\begin{smallmatrix} {\bm y}+\rho D_{f_i}({\bm x})\\
\tau {\bm y} - \tau D_{f_i}({\bm x})   \end{smallmatrix}\right]$ & $f({\bm x}) +\frac{1}{2\tau}\|{\bm y}\|^2$ & $\{({\bm x}, {\bm y}): {\bm 0} \in D_f({\bm x}), {\bm y} =0\}$ & $\|{\bm X}_{k} - {\bm X}_{k+1}\|$ \\
\hline
SignDSGD-Unb & $\left[\begin{smallmatrix} \mathrm{sign}({\bm y}+\rho D_{f_i}({\bm x}))\\
\tau {\bm y} - \tau D_{f_i}({\bm x})   \end{smallmatrix}\right]$ & $f({\bm x}) +\frac{1}{\tau}\|{\bm y}\|_{1}$ & $\{({\bm x}, {\bm y}): {\bm 0} \in D_f({\bm x}), {\bm y} =0\}$  & $\|{\bm X}_{k} - {\bm X}_{k+1}\|$ \\
\hline
CHOCO-SGD+  & $D_{f_i}$  & $f$ & $\{{\bm x} \in \Rn: {\bm 0} \in D_f({\bm x})\}$ & 0\\
\hline
BEER+   & $D_{f_i}$  & $f$ & $\{{\bm x} \in \Rn: {\bm 0} \in D_f({\bm x})\}$ & 0\\
\hline
C-GT+   & $D_{f_i}$  & $f$ & $\{{\bm x} \in \Rn: {\bm 0} \in D_f({\bm x})\}$ & 0\\
\hline
DSM-Con & $\left[\begin{smallmatrix} {\bm y}+\rho D_{f_i}({\bm x})\\
\tau {\bm y} - \tau D_{f_i}({\bm x})   \end{smallmatrix}\right]$ & $f({\bm x}) +\frac{1}{2\tau}\|{\bm y}\|^2$ & $\{({\bm x}, {\bm y}): {\bm 0} \in D_f({\bm x}), {\bm y} =0\}$ & $\|{\bm X}_{k} - {\bm X}_{k+1}\|$ \\
\hline
SignDSGD-Con & $\left[\begin{smallmatrix} \mathrm{sign}({\bm y}+\rho D_{f_i}({\bm x}))\\
\tau {\bm y} - \tau D_{f_i}({\bm x})   \end{smallmatrix}\right]$ & $f({\bm x}) +\frac{1}{\tau}\|{\bm y}\|_{1}$ & $\{({\bm x}, {\bm y}): {\bm 0} \in D_f({\bm x}), {\bm y} =0\}$ & $\|{\bm X}_{k} - {\bm X}_{k+1}\|$  \\
\hline
DSGTM-Con & $\left[\begin{smallmatrix} {\bm y}\\
\tau {\bm y} - \tau D_{f_i}({\bm x})   \end{smallmatrix}\right]$ & $f({\bm x}) +\frac{1}{2\tau}\|{\bm y}\|^2$ & $\{({\bm x}, {\bm y}): {\bm 0} \in D_f({\bm x}), {\bm y} =0\}$ & $\|{\bm X}_{k} - {\bm X}_{k+1}\|$ \\
\hline
\end{tabular}
\caption{Specific choices of $\Phi_{i}$, $\psi$, $\mathcal{A}$ and $\epsilon_k$ 
 for each method in Section \ref{sec:SGD}.}
\label{tab:2}
\end{table}

By applying Theorem \ref{thm:convergence}, we establish the convergence guarantees for above-mentioned methods. 

\begin{theo}\label{thm:applications}
Suppose Assumption \ref{Assumption_obj} holds, and the sequence $\{{\bm X}_{k}\}$ (and $\{{\bm Y}_{k}\}$, if defined)  is  generated by one of \eqref{unb:dsgd}, \eqref{unb:dsm}, \eqref{unb:signdsgd},\eqref{con:dsgd}, \eqref{con:dsgt}, \eqref{con:dsgt(cgt)}, \eqref{con:dsm}, \eqref{con:signdsgd} and \eqref{con:dsgtm}. Assume that the sequence $\{{\bm X}_{k}\}$ (and $\{{\bm Y}_{k}\}$, if defined) is bounded. Then, any cluster point of $\{{\bm X}_{k}\}$ is a $D_f$-critical point of \eqref{Prob_DOP}, and the sequence  $\{f({\bm x}_{i,k}): k \in \bb{N}\}$ converges for each $i\in [d]$.
\end{theo}



\section{Numerical Experiments}\label{sec:numeric}

In this section, we present preliminary numerical experiments to evaluate the performance of our proposed framework, which encompasses efficient decentralized stochastic subgradient-type methods with communication compression. Our numerical comparisons are two-fold: first, we compare existing methods with their nonsmooth extensions under our framework; second, we compare these nonsmooth extensions with newly developed methods adapted to our framework. 

All numerical experiments  are conducted on a platform equipped with two Intel(R) Xeon(R) Gold 5317 CPUs ($@$ 3.00GHz and 512GB RAM) and eight NVIDIA GeForce RTX 4090 GPUs, running Ubuntu 20.04. All decentralized algorithms are implemented in Python 3.8 and PyTorch 1.13.1, using NCCL 2.14.3 (CUDA 11.7) as the communication backend.

\subsection{Testing problem and implementation details}

We train ResNet-20 \cite{wen2016learning} models in a decentralized manner on the CIFAR-10 image classification task. It is worth noting that the ResNet-20 neural network employs ReLU
as its activation function, resulting in a non-Clarke-regular but definable loss function.

\textbf{Setup.}
In the first group of numerical comparisons, the decentralized network topology is configured as a ring structure with 8 agents, consistent with the experimental setup in \cite{koloskova2019decentralized, singh2021squarm}. In the second comparison group, we use an Erd\H{o}s-R\'{e}nyi (E.R.)  random graph topology \cite{nachmias2008critical} with 8 agents to evaluate the performance in a more general decentralized network setting. The mixing matrix is chosen as the classical Metropolis constant edge weight matrix \cite{xiao2006distributed}. Moreover, we split the original dataset evenly into 8 parts,  distributing each part to one agent as its local dataset. At the beginning of each training epoch, we reshuffle the local dataset of each agent and form batches of size $128$. Training is stopped at epoch 200. All compared methods are executed five times with varying random seeds.

Furthermore, the choices for the unbiased compression operator and the contractive compression operator in our experiments are as follows:

\begin{itemize}
\item \textit{Unbiased compression}: The $8$-bit random quantization operator $Q_{2^8}$ from Example \ref{exp:0}.

\item  \textit{Contractive compression}: The $\mathrm{Random}$-$10 \%$ and $\mathrm{Top}$-$10 \%$  sparsification operators from Example \ref{exa:1} and rescaled $8$-bits random quantization operator $\frac{1}{1+ \min\{d/2^{16}, \sqrt{d}/2^8\}}Q_{2^8}$ (abbreviated as rescaled $Q_{2^8}$).
\end{itemize}

\textbf{Implementation details.}
A common strategy for updating $\{\eta_{k}\}$ is to decay it at specific milestones while keeping it constant otherwise. Meanwhile, $\{\theta_k\}$ remains constant, and $\{\gamma_k\}$ is typically fixed to $1.0$ in existing works \cite{koloskova2019decentralized, singh2021squarm, zhao2022beer} under the smooth setting. To meet the requirement in Assumption \ref{Assumption_obj}, namely that the sequences $\{(\eta_{k}, \theta_{k}, \gamma_{k})\}$ form a three-timescale scheme and are of order $o(\frac{1}{\log(k)})$, we propose another update strategy for $\{(\eta_{k}, \theta_{k}, \gamma_{k})\}$. Empirically, starting with initial $\eta_0$, we warm up $\{\eta_k\}$ in the first five epochs to obtain $\eta_{\text{warm}}$ in both strategies.

\begin{enumerate}
\item \textbf{Strategy 1:} 
\begin{equation*}
\begin{aligned}
\eta_{k}^{(1)} & = \begin{cases}
\eta_{\text{warm }}, & 5 \leq k\cdot \text{number of batches per epoch} < 100,\\
0.1\eta_{\text{warm }}, & 100 \leq k\cdot \text{number of batches per epoch} < 180,\\
0.01 \eta_{\text{warm }}, & 180 \leq k\cdot \text{number of batches per epoch} \leq 200,\\
\end{cases}\\
\theta^{(1)}_k & \equiv \theta_0, \quad \gamma^{(1)}_k  \equiv 1.0. \\
\end{aligned}
\end{equation*}
This means the step size is decayed by a factor of 0.1 at the 100th and 180th epochs.
\item \textbf{Strategy 2:}
\begin{equation*}
\begin{aligned}
\eta_{k}^{(2)}  & = \begin{cases} 
\eta_{k}^{(1)}, & 5 \leq k\cdot \text{number of batches per epoch} < 100,\\
\frac{\eta_{k}^{(1)}}{\log(k-100 +e)^{1.01}}, & 100 \leq k\cdot \text{number of batches per epoch} < 200,\\
\end{cases}\\
\theta^{(2)}_{k} & =\theta_0 \eta^{(2)}_{k}/(\eta^{(2)}_{k}\log(k))^{s}, s \in (0, \frac{1}{2}), \\
\gamma^{(2)}_k & = \theta^{(2)}_{k}/(\theta^{(2)}_{k}\log(k))^{s}, s \in (0, \frac{1}{2}).\\
\end{aligned}
\end{equation*}
\end{enumerate}
By default, we set $s = 0.25$ in our experiments. Additionally, the fine-tuned hyper-parameters $\eta_0$ and $\theta_0$ for different methods and compression operators are presented in Table \ref{tab:5} and \ref{tab:6}. The momentum parameter $\tau$ in DSM-Unb(Con), DSGTM-Unb(Con) and SignDSGD-Unb(Con) is set as $\tau= \frac{0.1}{\eta_0}$ by default.

\begin{table}
\centering
\begin{tabular}{|c|c|c|c|}
\hline
Compression operator  & Method & Learning rate $\eta_0$  & Consensus step-size $\theta_0$ \\
\hline
\multirow{5}{*}{Rescaled $Q_{2^8}$}    & CHOCO-SGD(+) & 0.06  & 0.2 \\
\cline{2-4}
& BEER(+) & 0.06 & 0.2 \\
\cline{2-4}
& DSGTM-Con & 0.15 & 0.2 \\
\cline{2-4}
& DSM-Con & 0.06 & 0.2 \\
\cline{2-4}
& SignDSGD-Con & 0.00005 & 0.2 \\
\hline
\multirow{5}{*}{$\mathrm{Random}$-$10\%$}  & CHOCO-SGD(+) & 0.075  & 0.075 \\
\cline{2-4}
& BEER(+) & 0.02 & 0.02 \\
\cline{2-4}
& DSGTM-Con & 0.02 & 0.02 \\
\cline{2-4}
& DSM-Con & 0.2 & 0.075 \\
\cline{2-4}
& SignDSGD-Con & 0.0002 & 0.075 \\
\hline
\multirow{5}{*}{$\mathrm{Top}$-$10\%$}  & CHOCO-SGD(+) & 0.1  & 0.15 \\
\cline{2-4}
& BEER(+) & 0.06 & 0.2 \\
\cline{2-4}
& DSGTM-Con & 0.1 & 0.15 \\
\cline{2-4}
& DSM-Con & 0.1 & 0.15 \\
\cline{2-4}
& SignDSGD-Con & 0.00005 & 0.15 \\
\hline
\end{tabular}
\caption{Tuned hyper-parameters of CHOCO-SGD+, BEER+, DSGTM-Con, DSM-Con and SignDSGD-Con for training ResNet-20 on CIFAR-10, corresponding to the ring/E.R. topology with 8 agents.}
\label{tab:5}
\end{table}

\begin{table}
\centering
\begin{tabular}{|c|c|c|c|}
\hline
Compression operator  & Method & Learning rate $\eta_0$  & Consensus step-size $\theta_0$ \\
\hline
\multirow{4}{*}{$Q_{2^8}$}    & QSDGD+ & 0.1  & 0.15 \\
\cline{2-4}
& DSGTM-Unb & 0.05 & 0.0075\\
\cline{2-4}
& DSM-Unb & 0.05 & 0.075 \\
\cline{2-4}
& SignDSGD-Unb & 0.0001 & 0.15 \\
\hline
\end{tabular}
\caption{Tuned hyper-parameters of QSDGD+, DSGTM-Unb, DSM-Unb and SignDSGD-Unb for training ResNet-20 on CIFAR-10, corresponding to the ring/E.R. topology with 8 agents.}
\label{tab:6}
\end{table}

\subsection{Performance comparison between existing methods and their nonsmooth extensions under our framework}\label{sec:numerical2}

Figures \ref{fig:chocosgd}, \ref{fig:beer} and \ref{fig:qsdgd} present the numerical performance of the following comparison pairs:
\begin{itemize}
\item CHOCO-SGD vs. CHOCO-SGD+,
\item BEER vs. BEER+
\item QSDGD \footnote{QDGD was originally proposed in \cite{reisizadeh2019exact}. We derive QSDGD by replacing the full gradient with stochastic gradients.} vs. QSDGD+ 
\end{itemize}

In each pair, the former represents the original method from the existing literature, using \textbf{Strategy 1} to update the step-sizes, while the latter (marked with a "+") is its direct nonsmooth extension adapted to our framework, using \textbf{Strategy 2} for step-size updates. Within each pair, the hyper-parameters (including $\theta_0$ and $\eta_0$) are set identically.

\begin{figure}[htbp]
\centering
\subfigure[Test accuracy w.r.t. Epochs]{
\includegraphics[width=5.8cm, height=3.4cm]{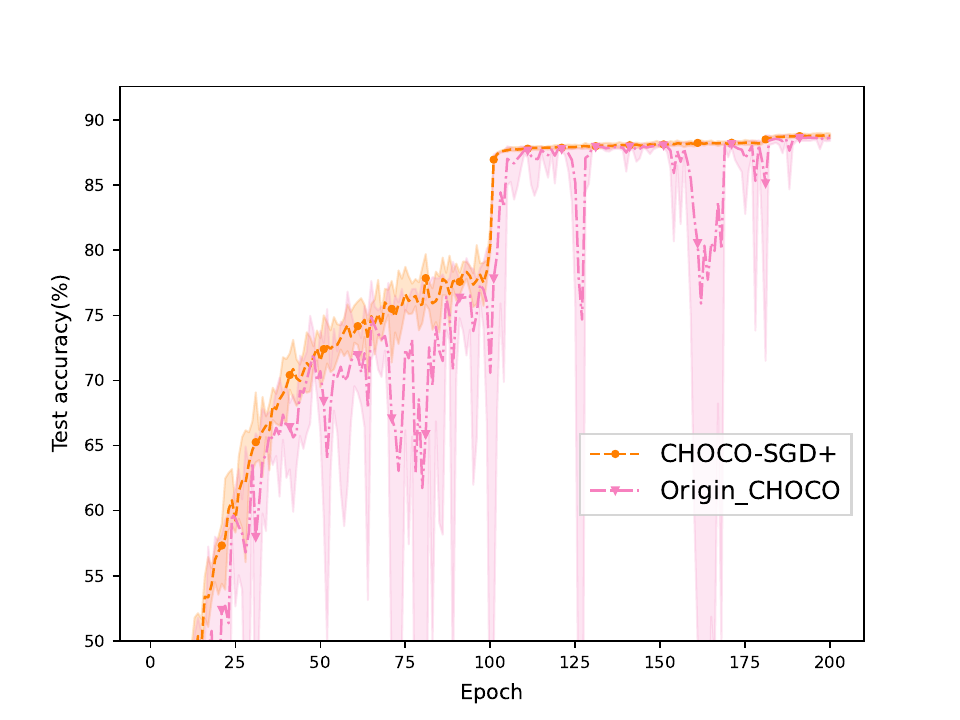}}
\hfil
\subfigure[Training loss w.r.t. Epochs]{
\includegraphics[width=5.8cm, height=3.4cm]{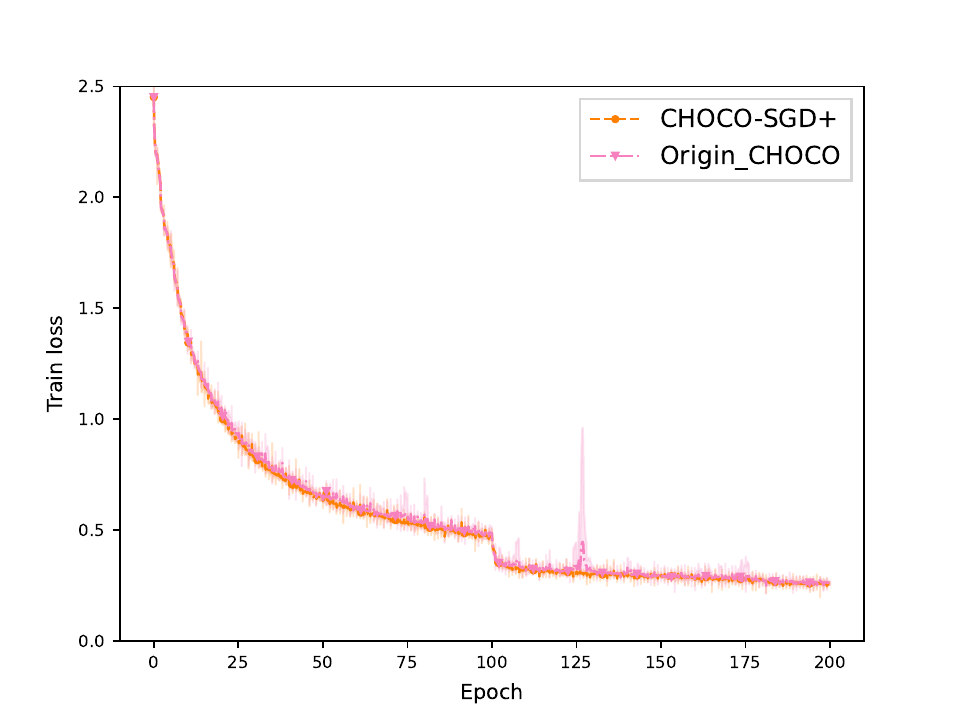}}
\caption{The comparison between CHOCO-SGD and CHOCO-SGD+ with the Random-$10\%$ operator.}
\label{fig:chocosgd}
\end{figure}

\begin{figure}[htbp]
\centering
\subfigure[Test accuracy w.r.t. Epochs]{
\includegraphics[width=5.8cm, height=3.4cm]{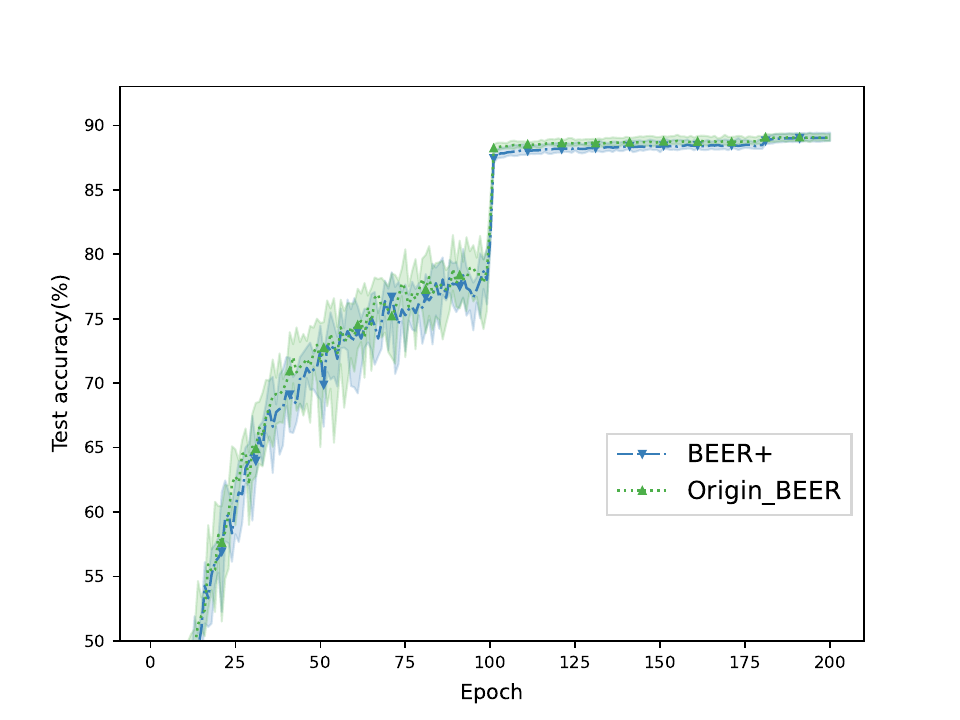}}
\hfil
\subfigure[Training loss w.r.t. Epochs]{
\includegraphics[width=5.8cm, height=3.4cm]{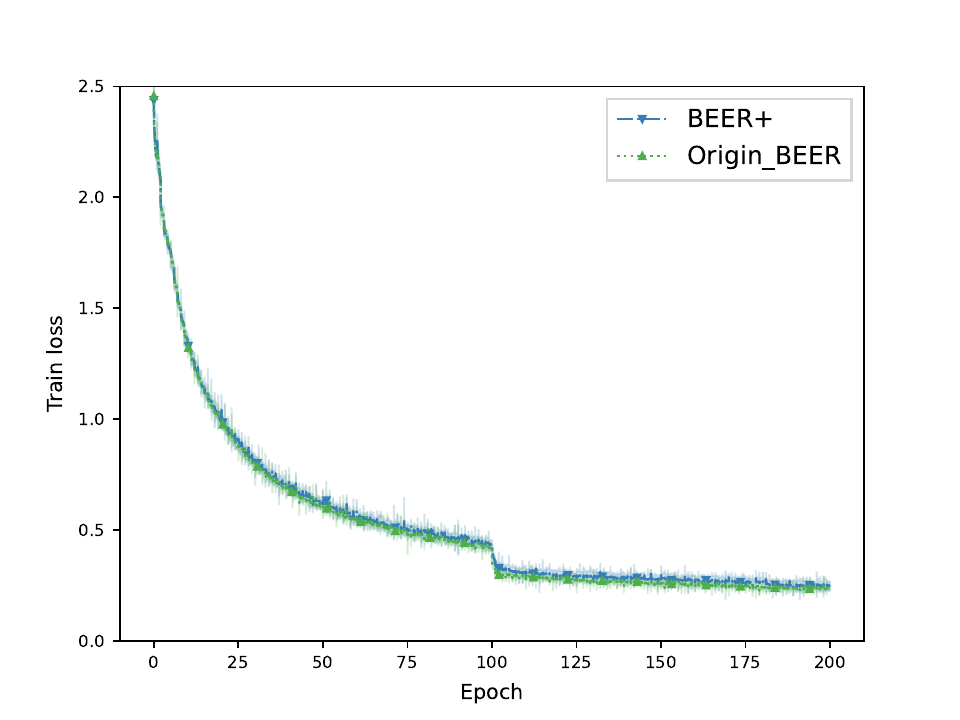}}
\caption{The comparison between BEER and BEER+ with the rescaled $Q_{2^8}$ operator.}
\label{fig:beer}
\end{figure}

\begin{figure}[htbp]
\centering
\subfigure[Test accuracy w.r.t. Epochs]{
\includegraphics[width=5.8cm, height=3.4cm]{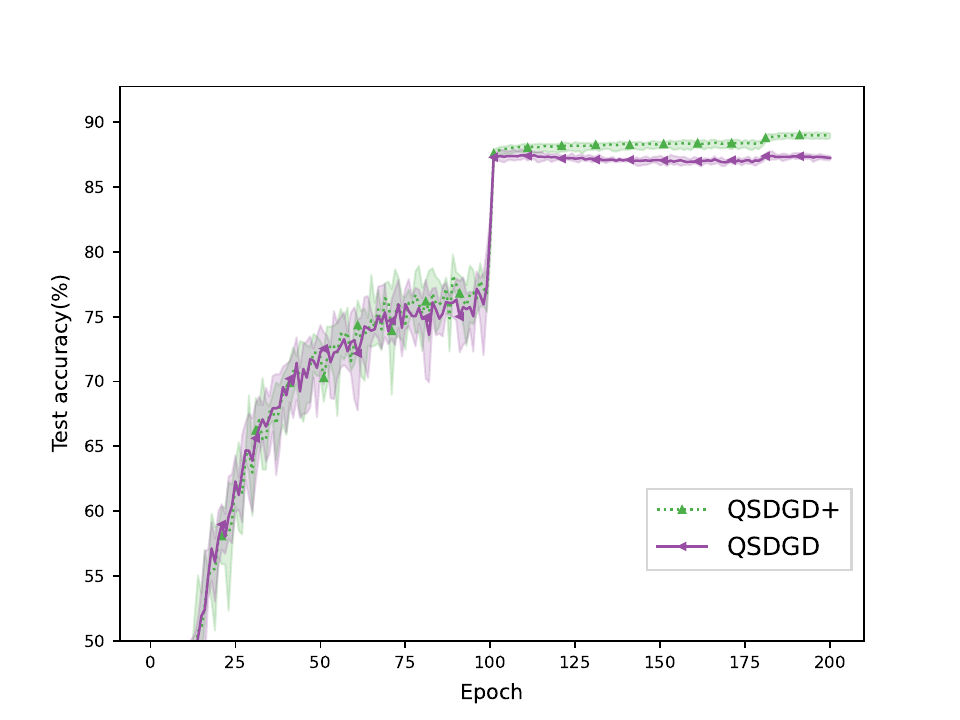}}
\hfil
\subfigure[Training loss w.r.t. Epochs]{
\includegraphics[width=5.8cm, height=3.4cm]{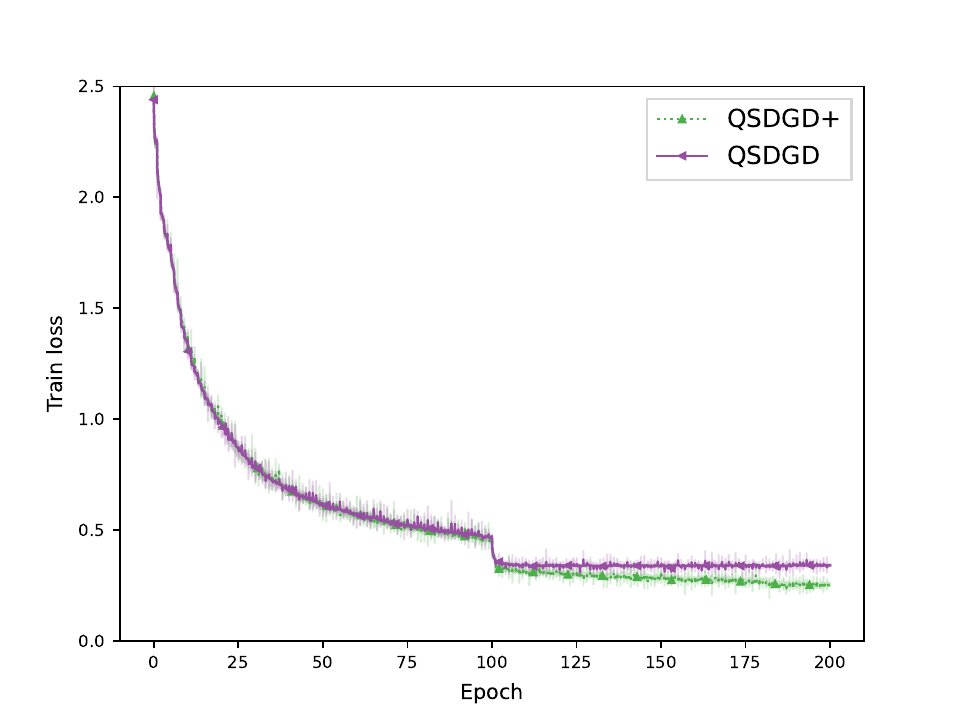}}
\caption{The comparison between QSDGD and QSDGD+ with the $Q_{2^8}$ operator.}
\label{fig:qsdgd}
\end{figure}

As shown in Figure \ref{fig:chocosgd}-\ref{fig:qsdgd}, the nonsmooth extensions of existing methods based on our framework demonstrate performance comparable to their original counterparts when training nonsmooth neural networks. Moreover, we observe that when using the Random-$10\%$ operator, the original CHOCO-SGD exhibits loss spikes and abrupt declines in test accuracy, while CHOCO-SGD+ behaves more stably and smoothly compared to the original version.

\subsection{Performance comparison between nonsmooth extensions of existing methods and the newly developed methods under our framework}\label{sec:numerical}

In this part, we conduct numerical comparisons organized into the following three groups, each of which utilizes \textbf{Strategy 2} for updating step-sizes.
\begin{itemize}
\item \textbf{GD-based methods with contractive compression}: CHOCO-SGD+, DSM-Con and SignDSGD-Con;
\item \textbf{Gradient-tracking-based methods with contractive compression}: BEER+ and DSGTM-Con;
\item \textbf{Methods with unbiased compression}: QSDGD+, DSM-Unb, SignDSGD-Unb and DSGTM-Unb.
\end{itemize}

In the first two experimental groups, we evaluate their performance over three contractive compression operators, including the $\mathrm{Top}$-$10\%$ operator, the $\mathrm{Random}$-$10\%$ operator, and the rescaled $Q_{2^8}$ operator. 
In the last experimental group, we utilize the $Q_{2^8}$ operator as our unbiased compression operator. Furthermore, we incorporate vanilla DSGD \cite{lian2017can} as the baseline algorithm. 

\begin{figure}[htbp]
\centering
\subfigure[Test accuracy w.r.t. Epochs]{
\includegraphics[width=3.6cm, height=2.6cm]{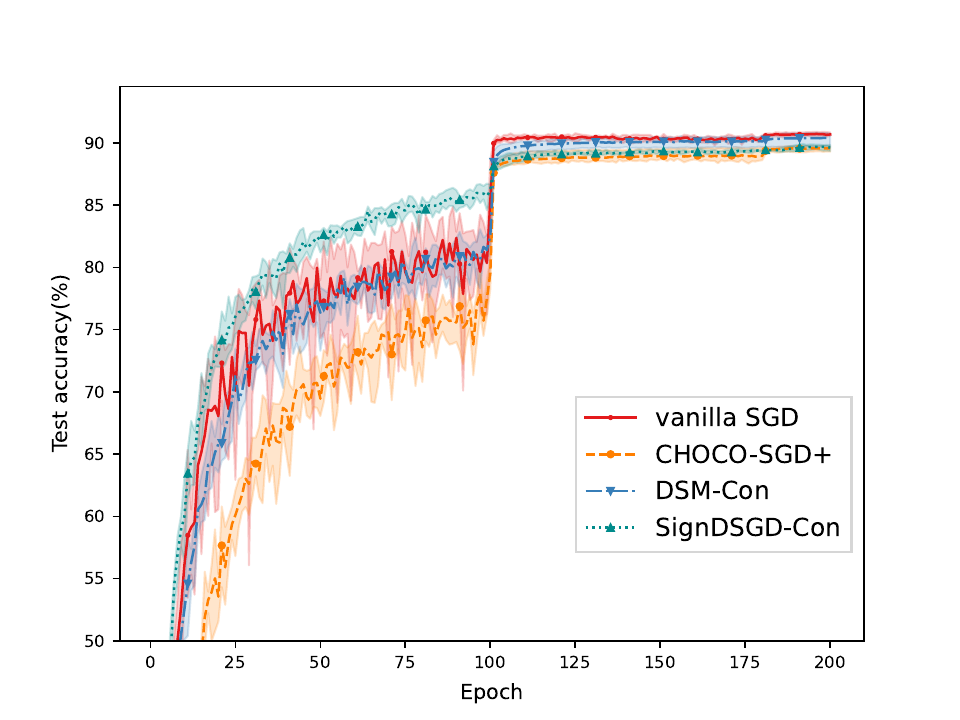}
\label{fig_1_1}}
\subfigure[Training loss w.r.t. Epochs]{
\includegraphics[width=3.6cm, height=2.6cm]{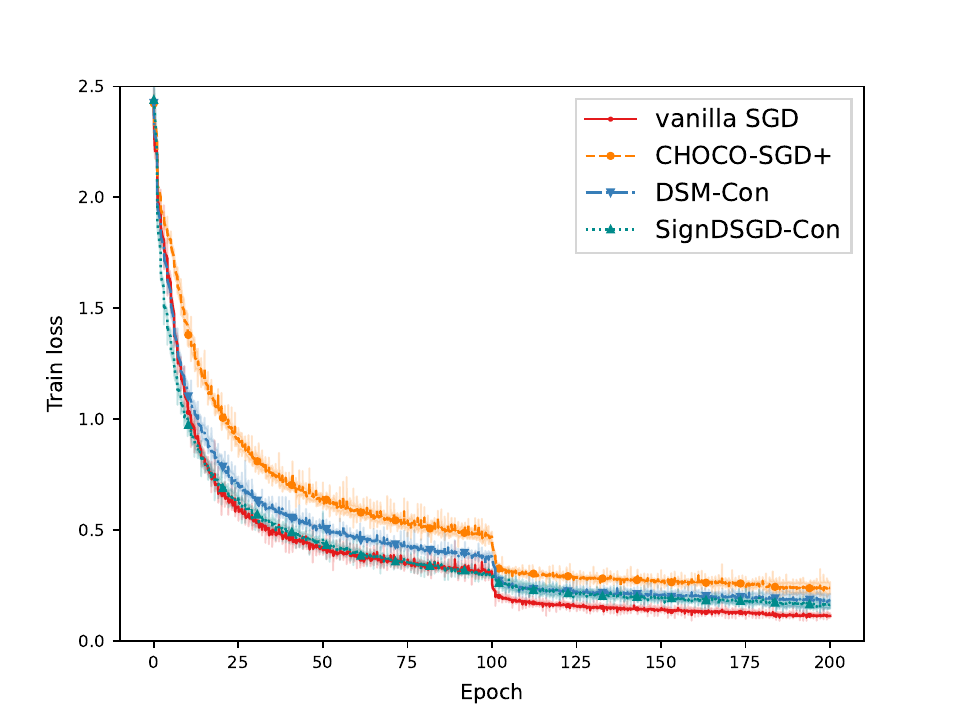}
\label{fig_1_2}}
\subfigure[Test accuracy w.r.t. Communication bits]{
\includegraphics[width=3.6cm, height=2.6cm]{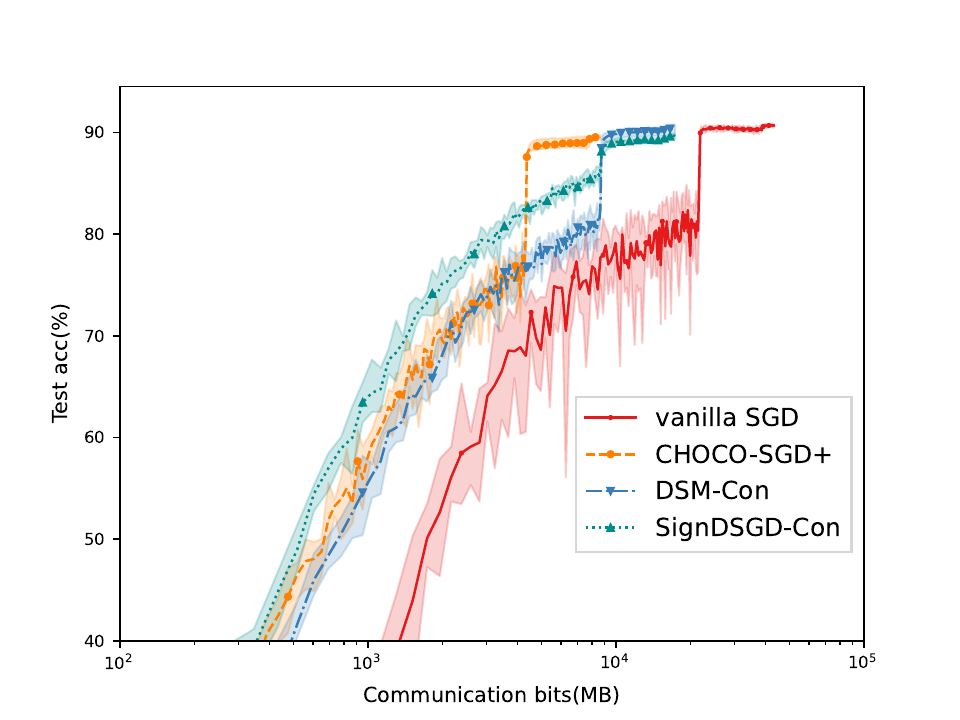}
\label{fig_1_3}}
\subfigure[Training loss w.r.t. Communication bits]{
\includegraphics[width=3.6cm, height=2.6cm]{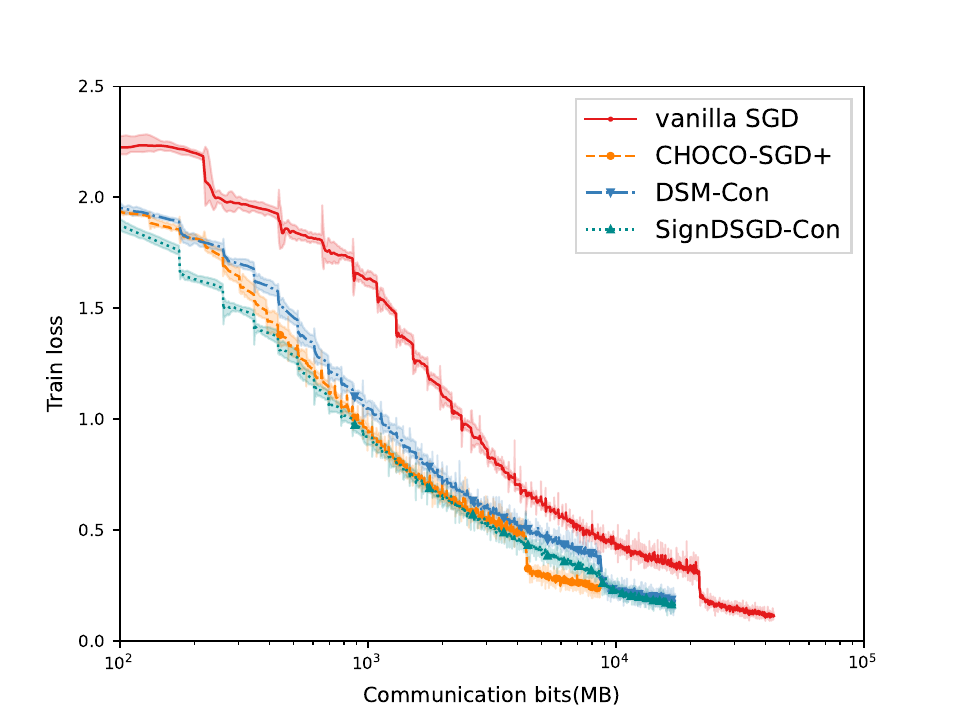}
\label{fig_1_4}}
\caption{Numerical results of GD-based methods with contractive compression ($\mathrm{Top}$-$10\%$ operator).}
\label{fig:group1}
\end{figure}

\begin{figure}[htbp]
\centering
\subfigure[Test accuracy w.r.t. Epochs]{
\includegraphics[width=3.6cm, height=2.6cm]{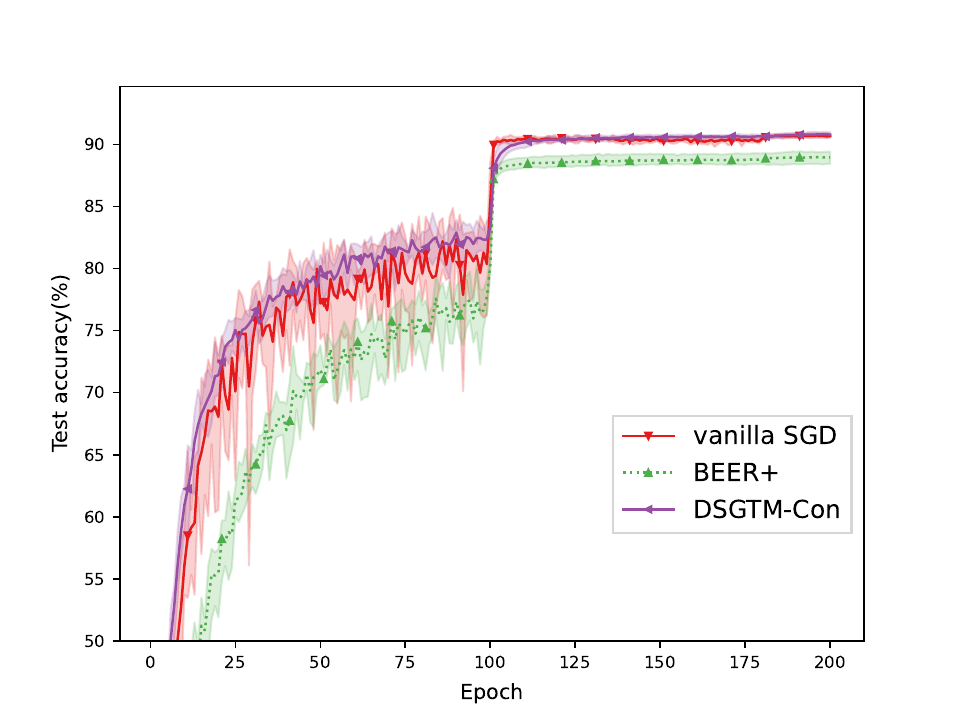}
\label{fig_2_1}}
\hfil
\subfigure[Training loss w.r.t. Epochs]{
\includegraphics[width=3.6cm, height=2.6cm]{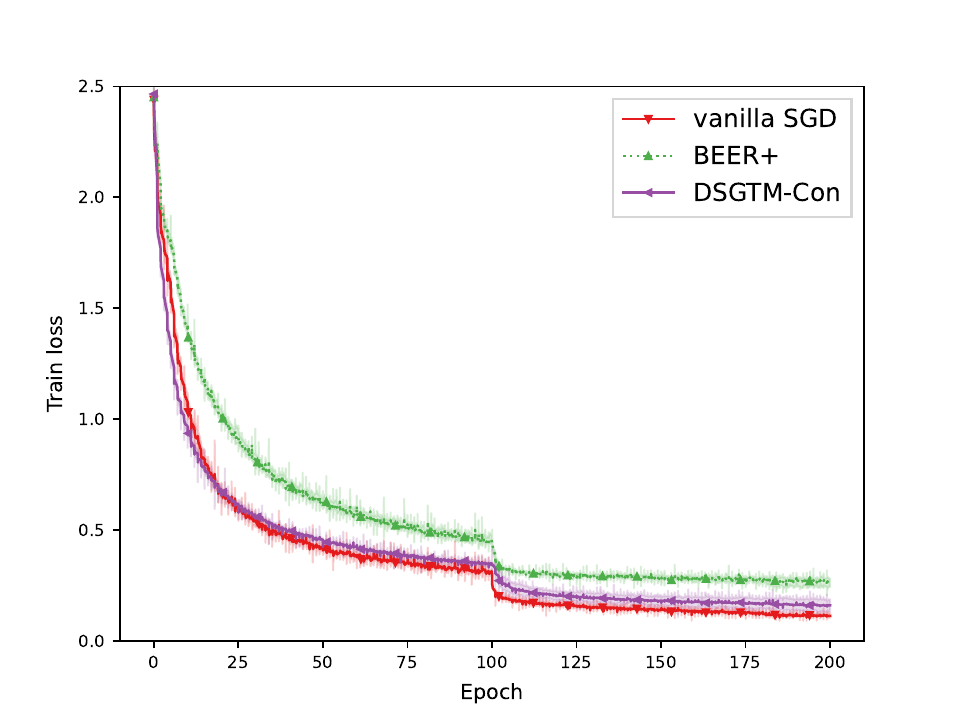}
\label{fig_2_2}}
\hfil
\subfigure[Test accuracy w.r.t. Communication bits]{
\includegraphics[width=3.6cm, height=2.6cm]{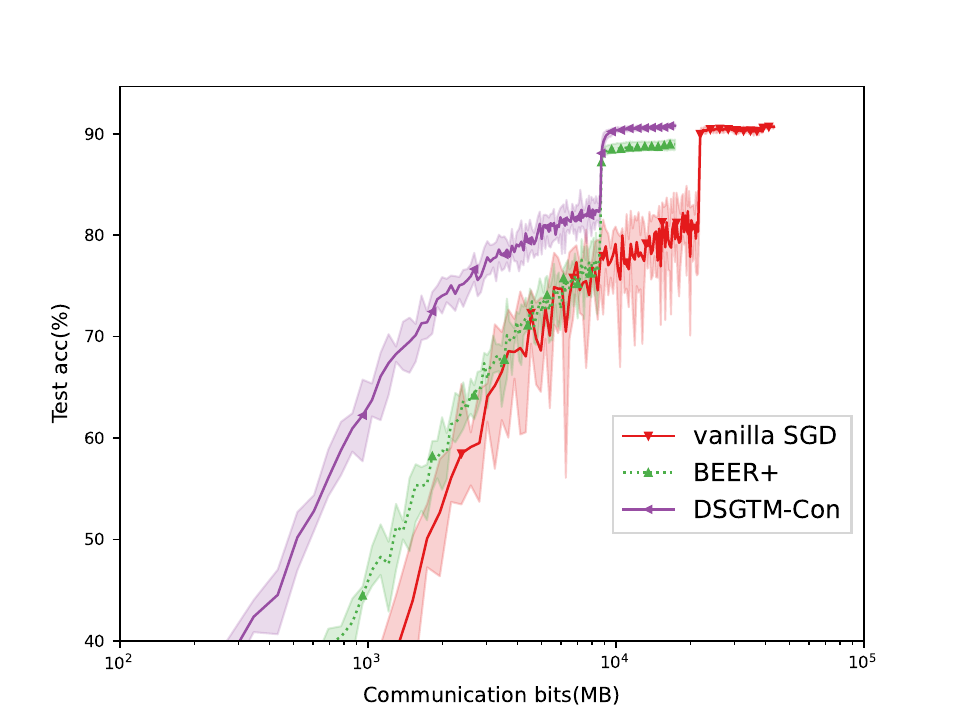}
\label{fig_2_3}}
\hfil
\subfigure[Training loss w.r.t. Communication bits]{
\includegraphics[width=3.6cm, height=2.6cm]{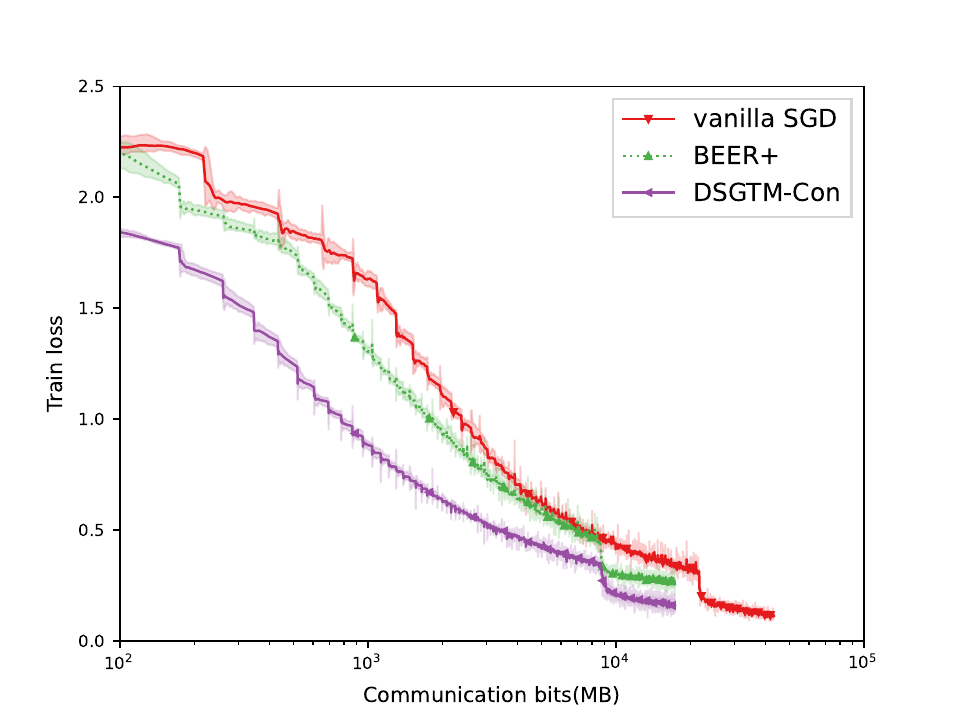}
\label{fig_2_4}}
\hfil

\caption{Numerical results of Gradient-tracking-based methods with contractive compression ($\mathrm{Top}$-$10\%$ operator).}
\label{fig:group2}
\end{figure}

\begin{figure}[htbp]
\centering
\subfigure[Test accuracy w.r.t. Epochs]{
\includegraphics[width=3.6cm, height=2.6cm]{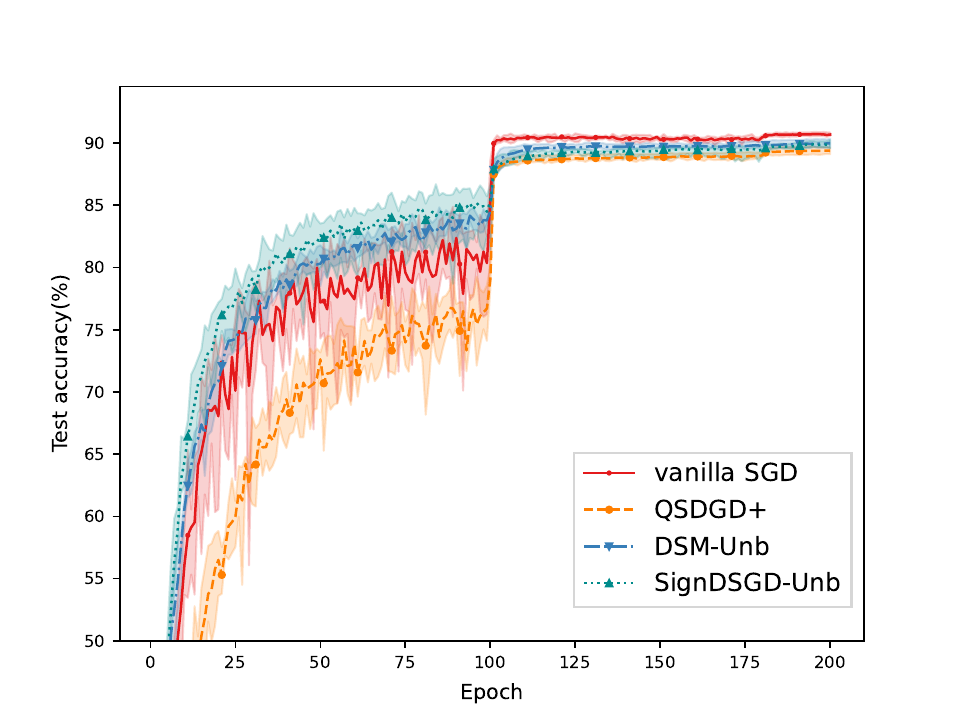}
\label{fig_4_1}}
\hfil
\subfigure[Training loss w.r.t. Epochs]{
\includegraphics[width=3.6cm, height=2.6cm]{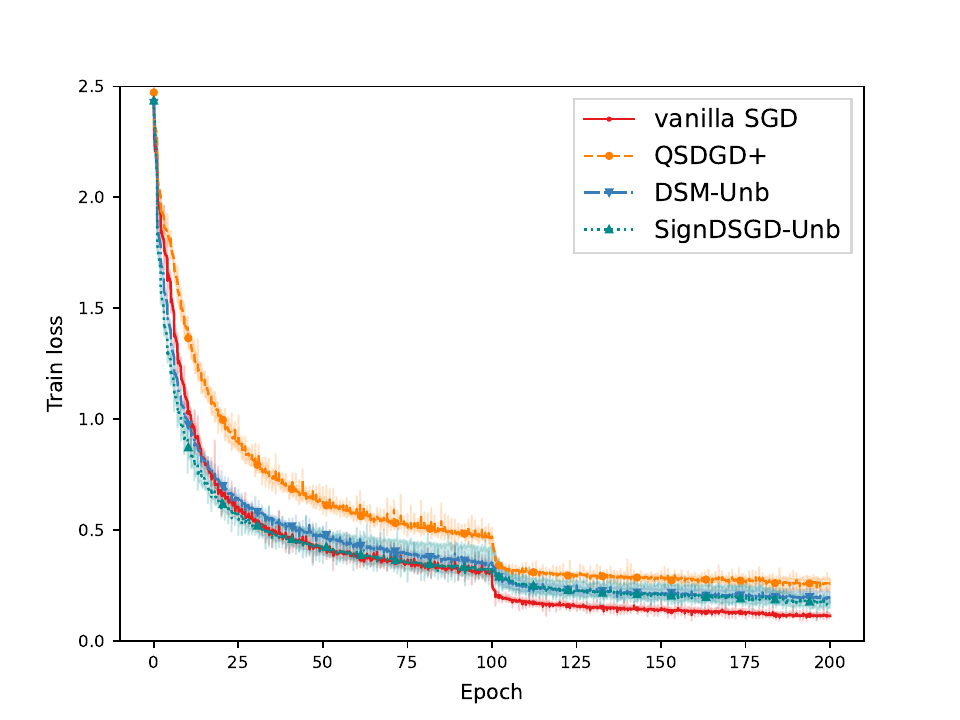}
\label{fig_4_2}}
\hfil
\subfigure[Test accuracy w.r.t. Communication bits]{
\includegraphics[width=3.6cm, height=2.6cm]{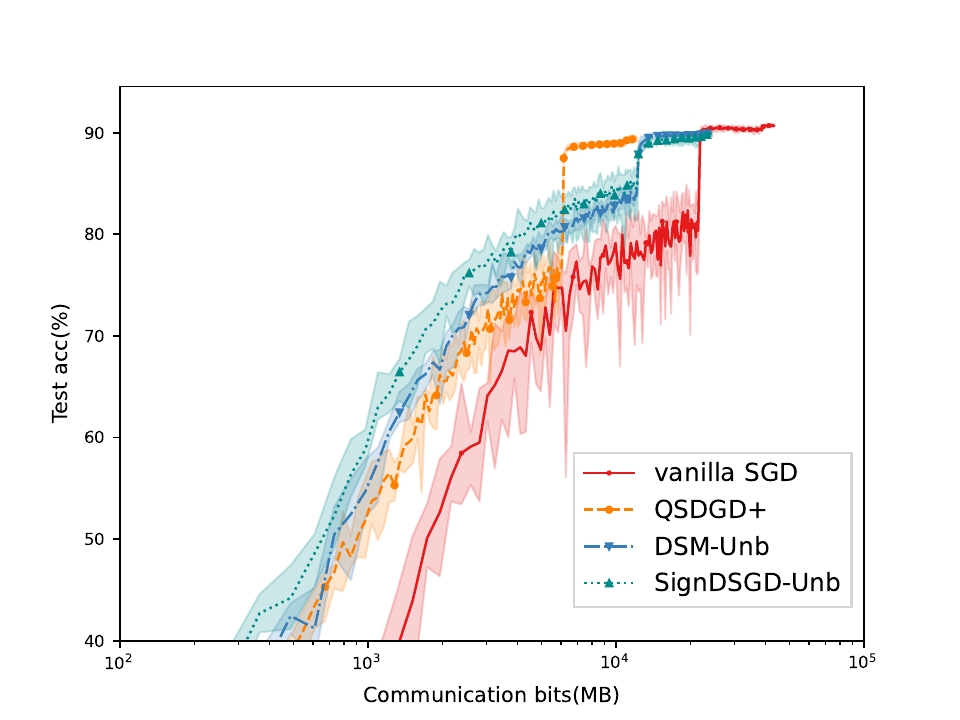}
\label{fig_4_3}}
\hfil
\subfigure[Training loss w.r.t. Communication bits]{
\includegraphics[width=3.6cm, height=2.6cm]{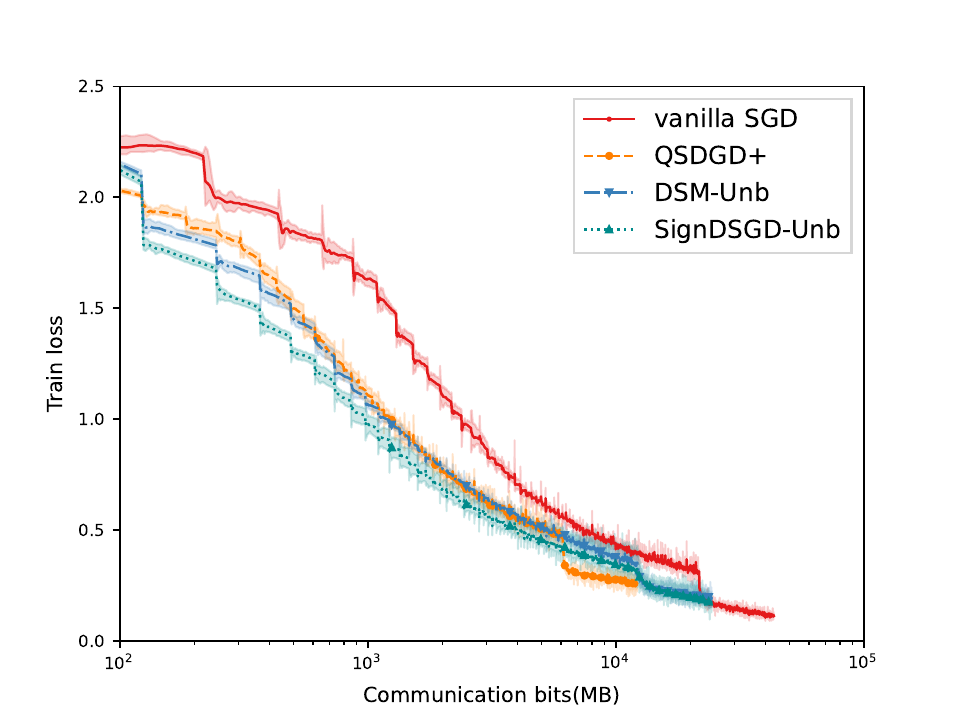}
\label{fig_4_4}}
\hfil
\caption{Numerical results of methods with unbiased compression ($Q_{2^8}$ operator).}
\label{fig:group3}
\end{figure}

\begin{figure}[htbp]
\centering
\subfigure[Test accuracy w.r.t. Epochs]{
\includegraphics[width=3.6cm, height=2.6cm]{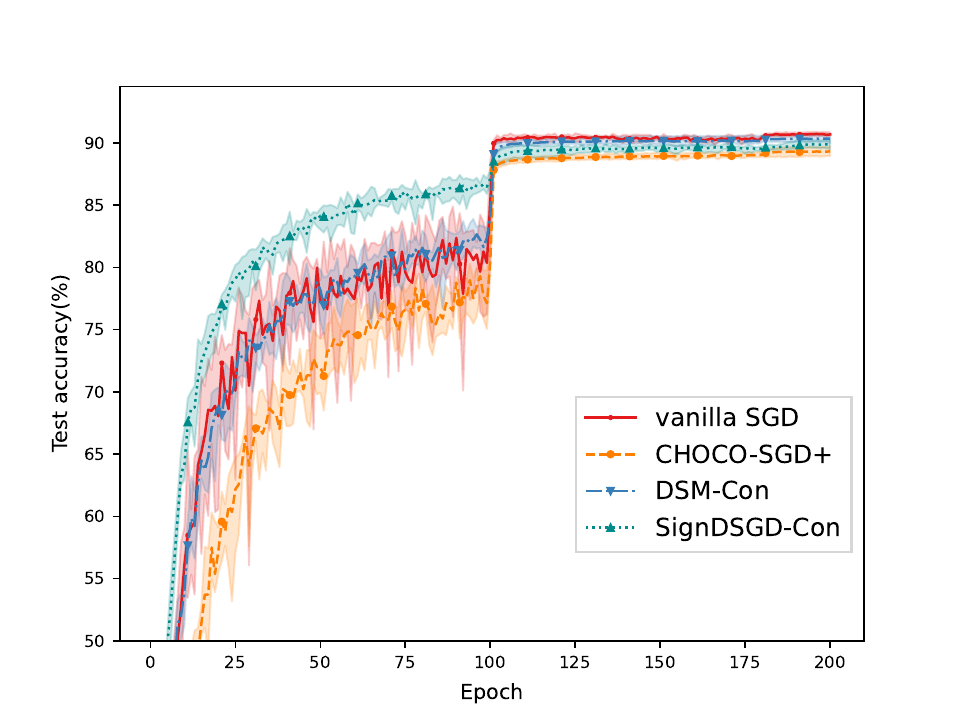}
\label{fig_5_1}}
\hfil
\subfigure[Training loss w.r.t. Epochs]{
\includegraphics[width=3.6cm, height=2.6cm]{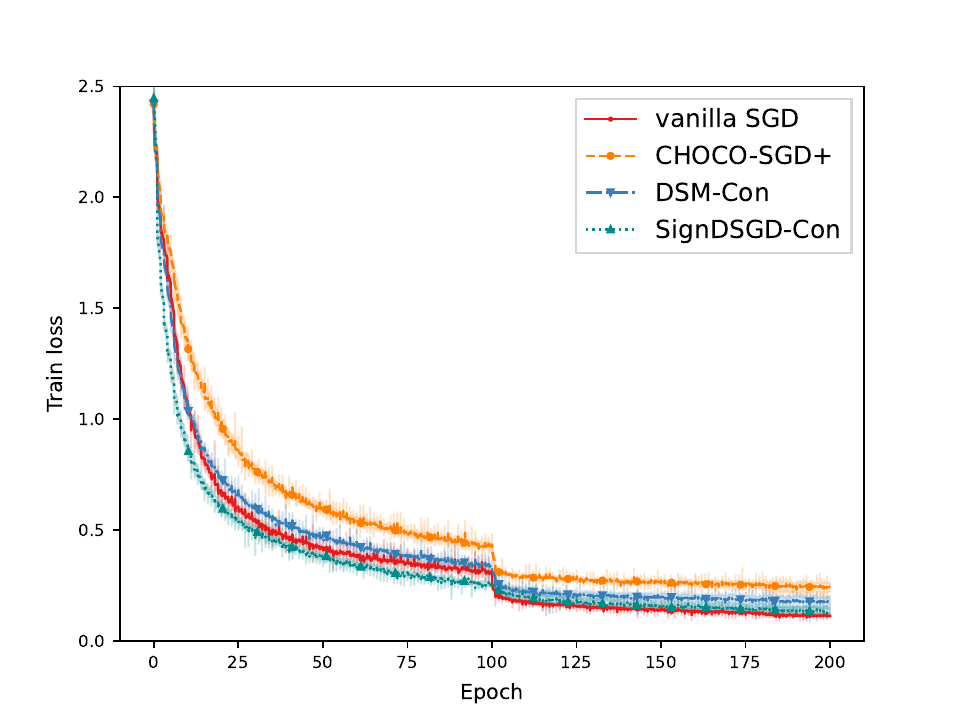}
\label{fig_5_2}}
\hfil
\subfigure[Test accuracy w.r.t. Communication bits]{
\includegraphics[width=3.6cm, height=2.6cm]{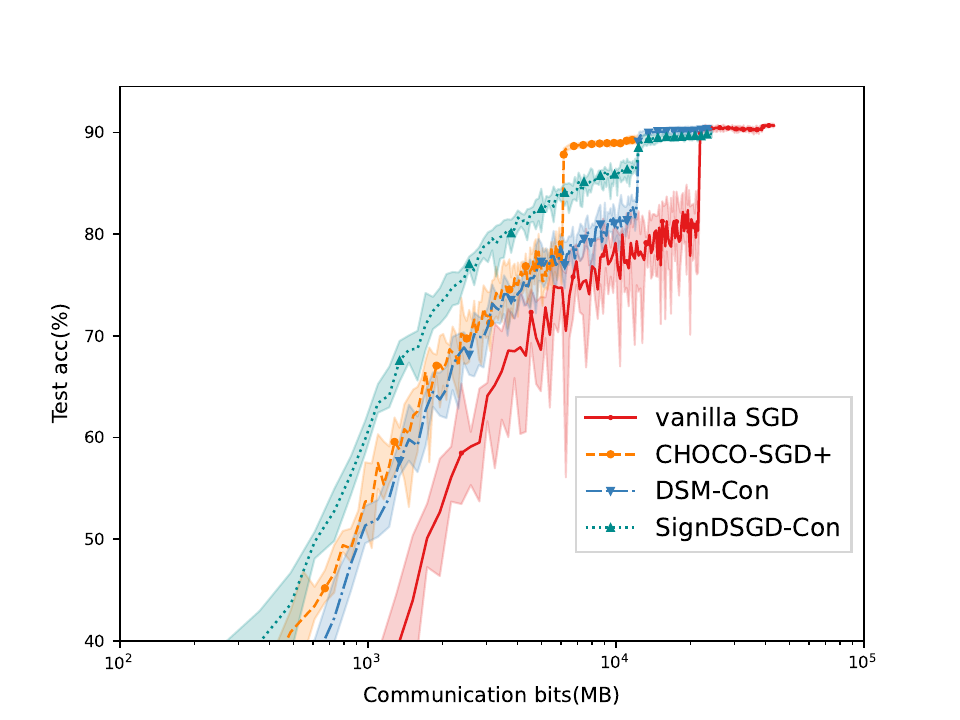}
\label{fig_5_3}}
\hfil
\subfigure[Training loss w.r.t. Communication bits]{
\includegraphics[width=3.6cm, height=2.6cm]{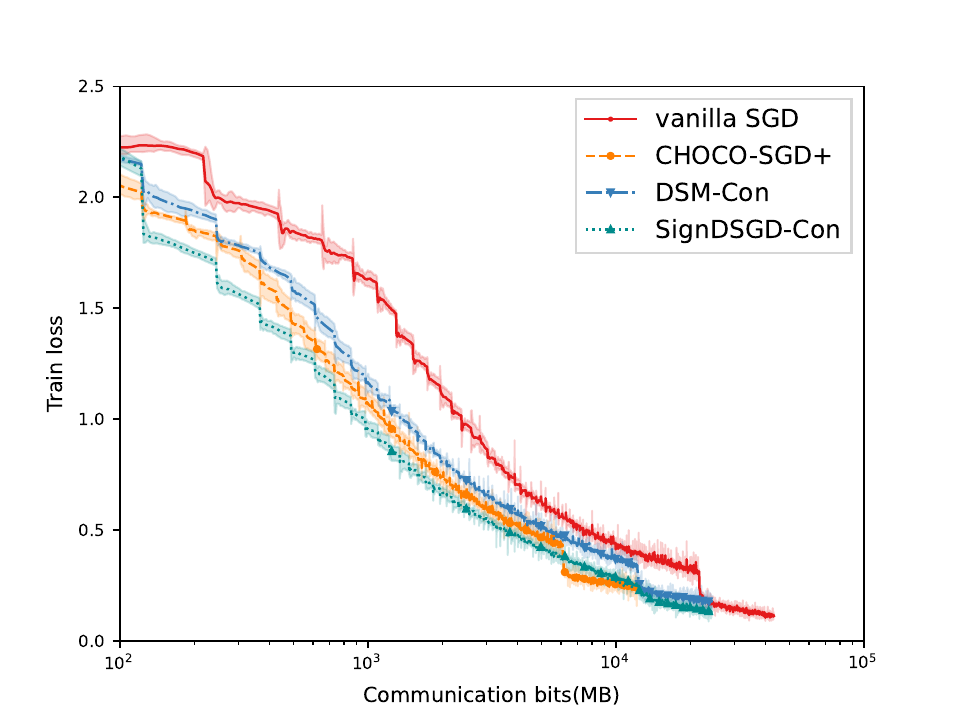}
\label{fig_5_4}}
\hfil

\caption{Numerical results of GD-based methods with contractive compression (rescaled $Q_{2^8}$ operator).}
\label{fig:quant8_group1}
\end{figure}

\begin{figure}[htbp]
\centering
\subfigure[Test accuracy w.r.t. Epochs]{
\includegraphics[width=3.6cm, height=2.6cm]{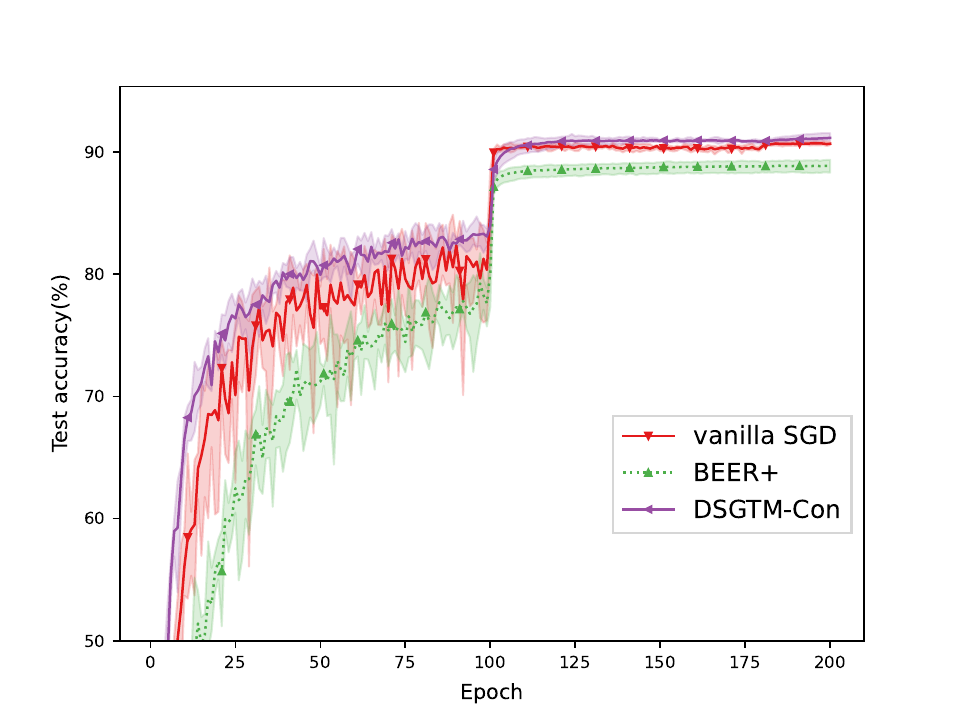}
\label{fig_6_1}}
\hfil
\subfigure[Training loss w.r.t. Epochs]{
\includegraphics[width=3.6cm, height=2.6cm]{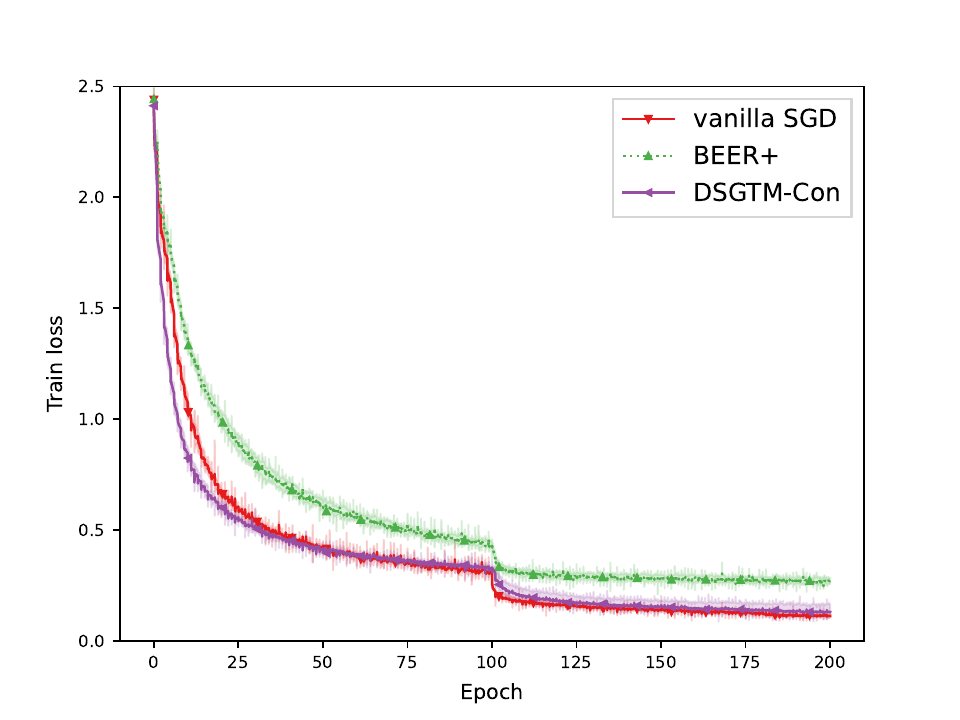}
\label{fig_6_2}}
\hfil
\subfigure[Test accuracy w.r.t. Communication bits]{
\includegraphics[width=3.6cm, height=2.6cm]{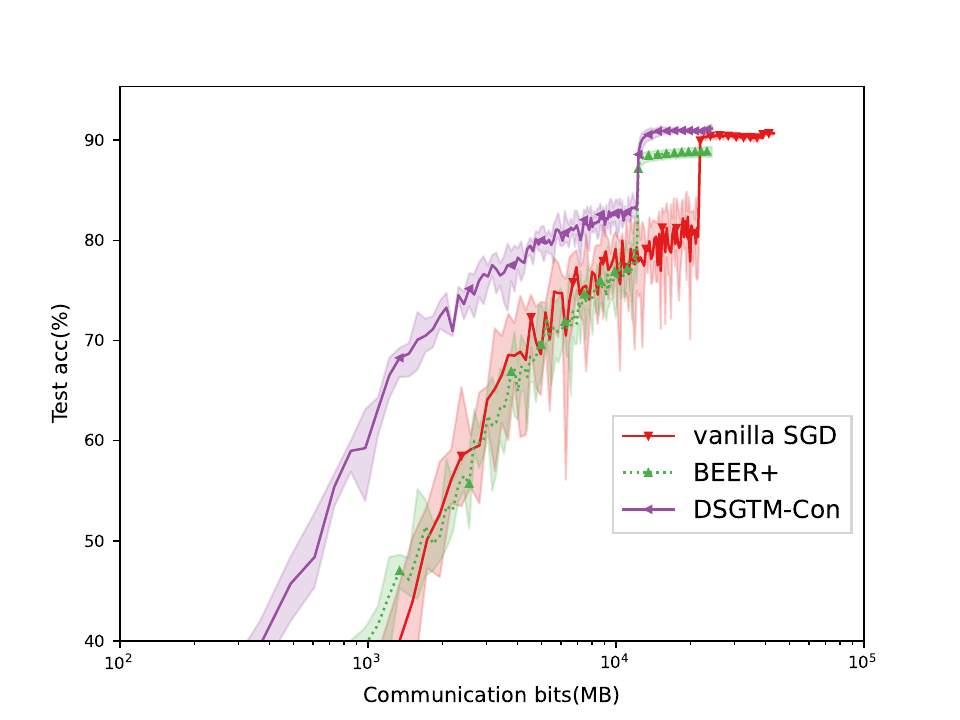}
\label{fig_6_3}}
\hfil
\subfigure[Training loss w.r.t. Communication bits]{
\includegraphics[width=3.6cm, height=2.6cm]{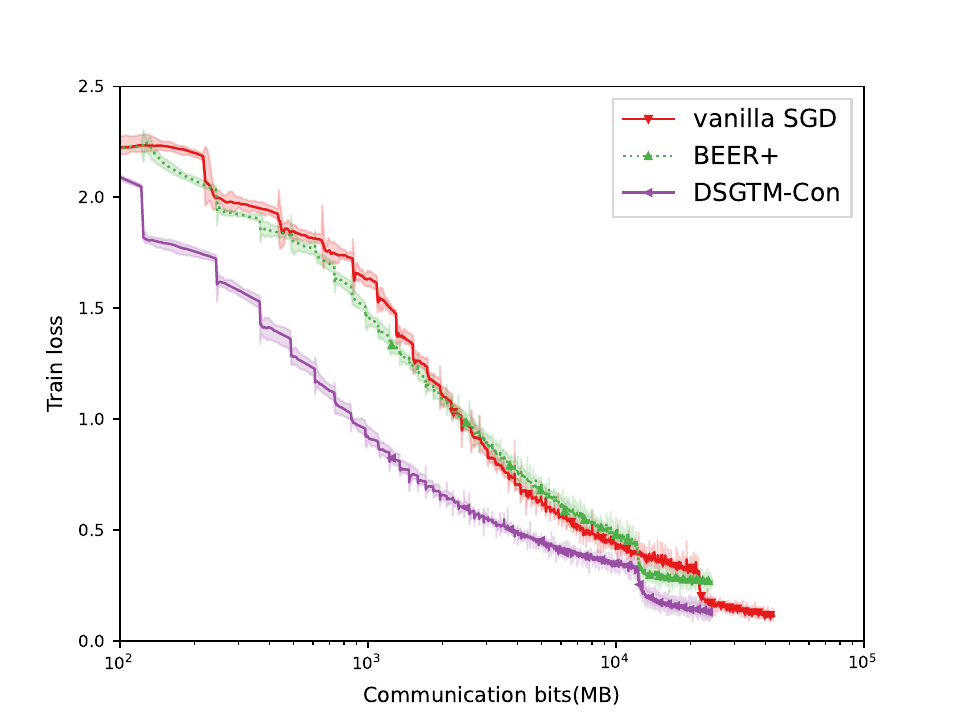}
\label{fig_6_4}}
\hfil

\caption{Numerical results of Gradient-tracking-based methods with contractive compression (rescaled $Q_{2^8}$ operator).}
\label{fig:quant8_group2}
\end{figure}

\begin{figure}[htbp]
\centering
\subfigure[Test accuracy w.r.t. Epochs]{
\includegraphics[width=3.6cm, height=2.6cm]{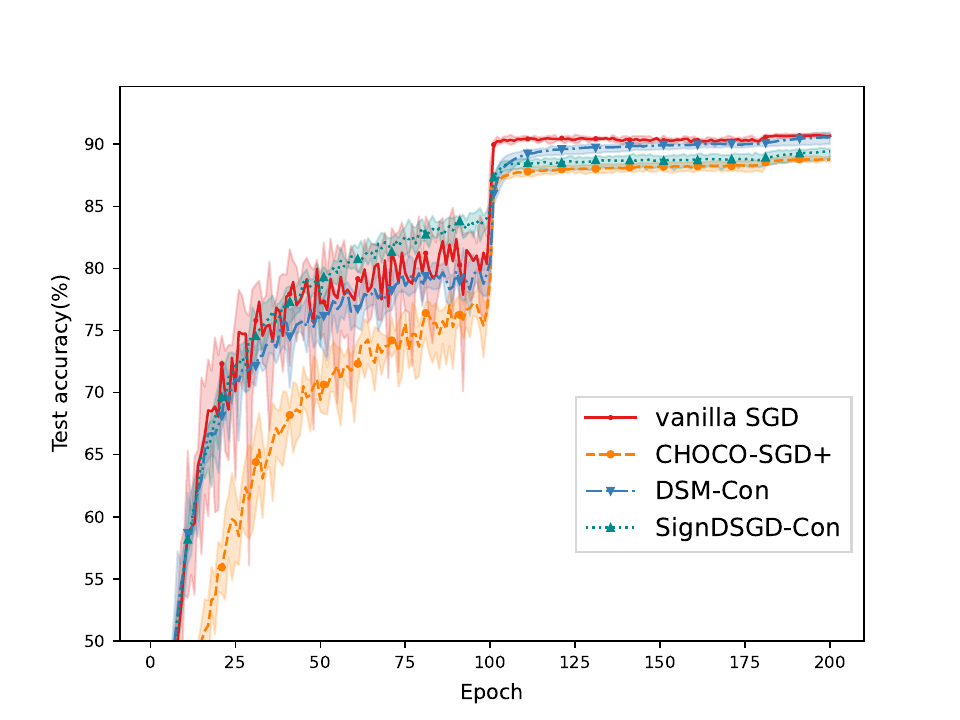}
\label{fig_7_1}}
\hfil
\subfigure[Training loss w.r.t. Epochs]{
\includegraphics[width=3.6cm, height=2.6cm]{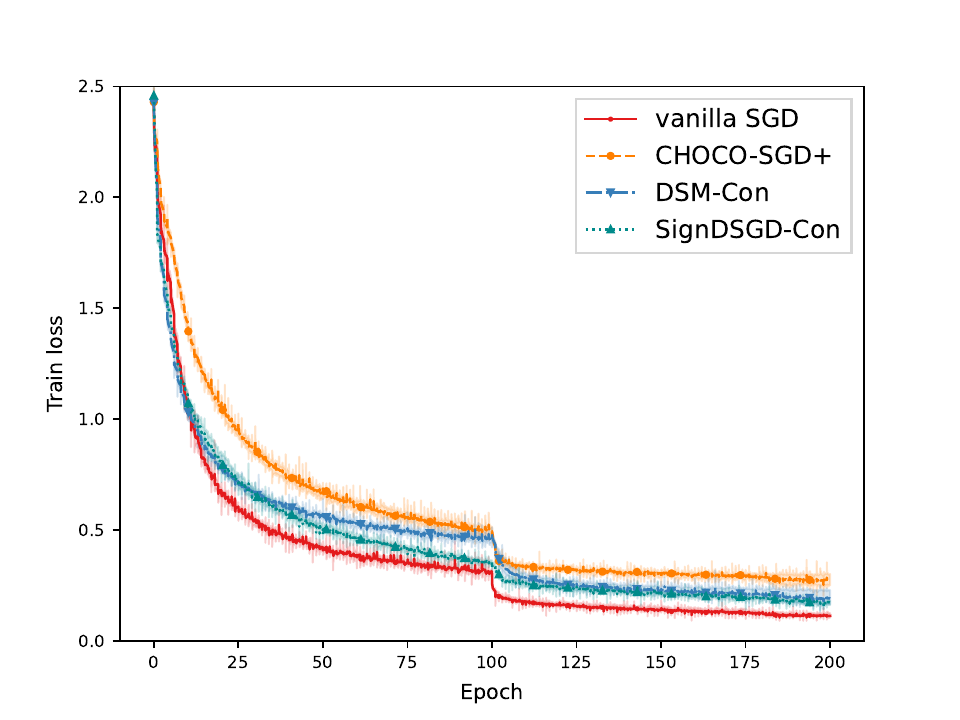}
\label{fig_7_2}}
\hfil
\subfigure[Test accuracy w.r.t. Communication bits]{
\includegraphics[width=3.6cm, height=2.6cm]{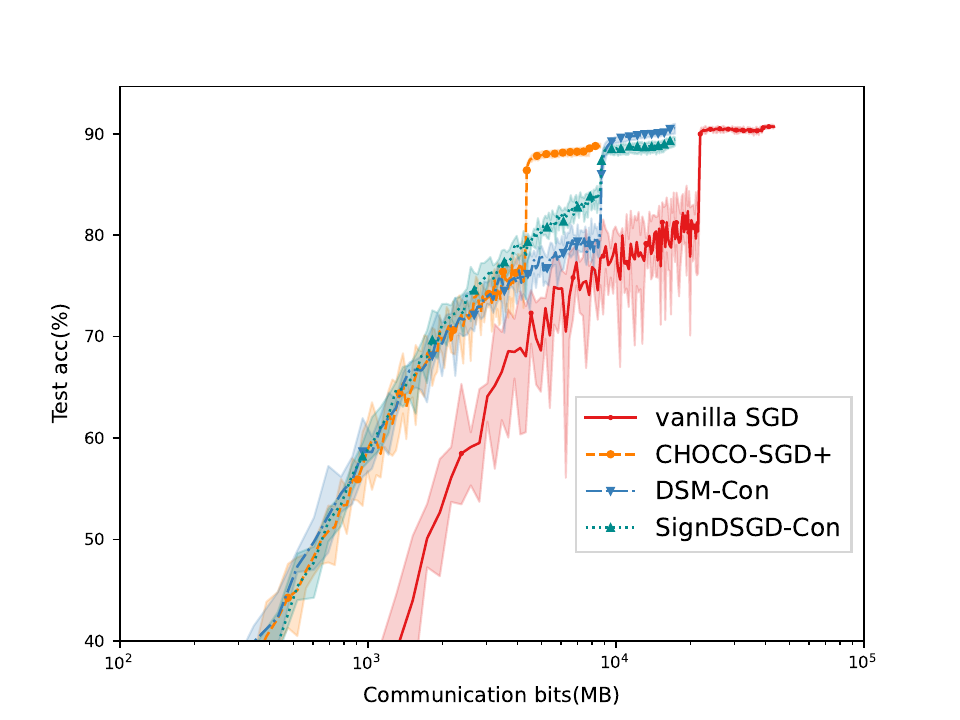}
\label{fig_7_3}}
\hfil
\subfigure[Training loss w.r.t. Communication bits]{
\includegraphics[width=3.6cm, height=2.6cm]{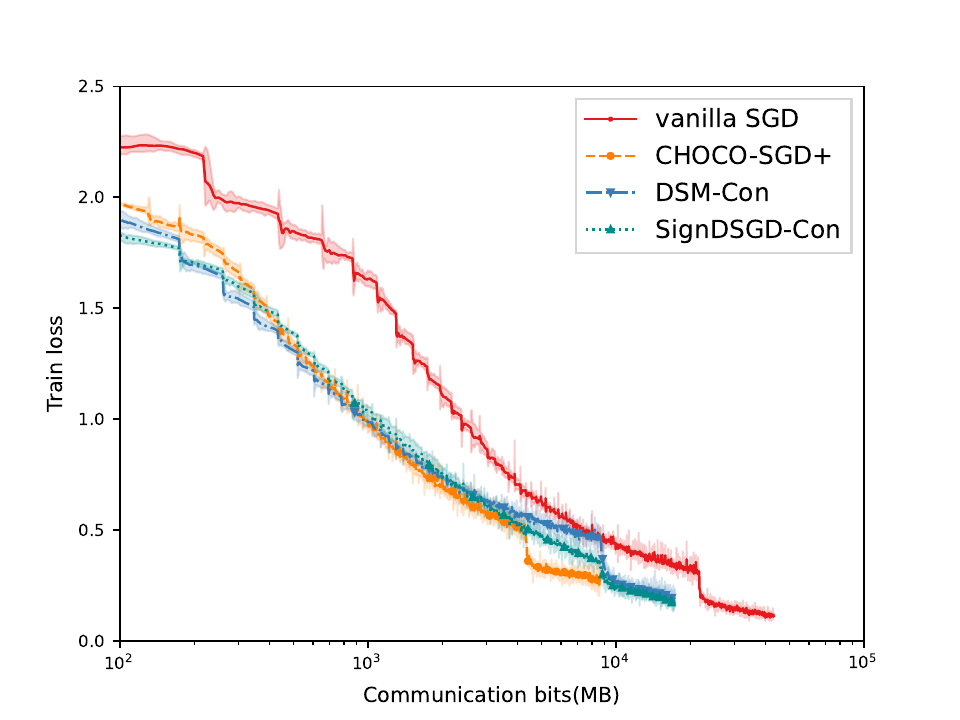}
\label{fig_7_4}}
\hfil

\caption{Numerical results of GD-based methods with contractive compression ($\mathrm{Random}$-$10\%$ operator).}
\label{fig:random10_group1}
\end{figure}

\begin{figure}[htbp]
\centering
\subfigure[Test accuracy w.r.t. Epochs]{
\includegraphics[width=3.6cm, height=2.6cm]{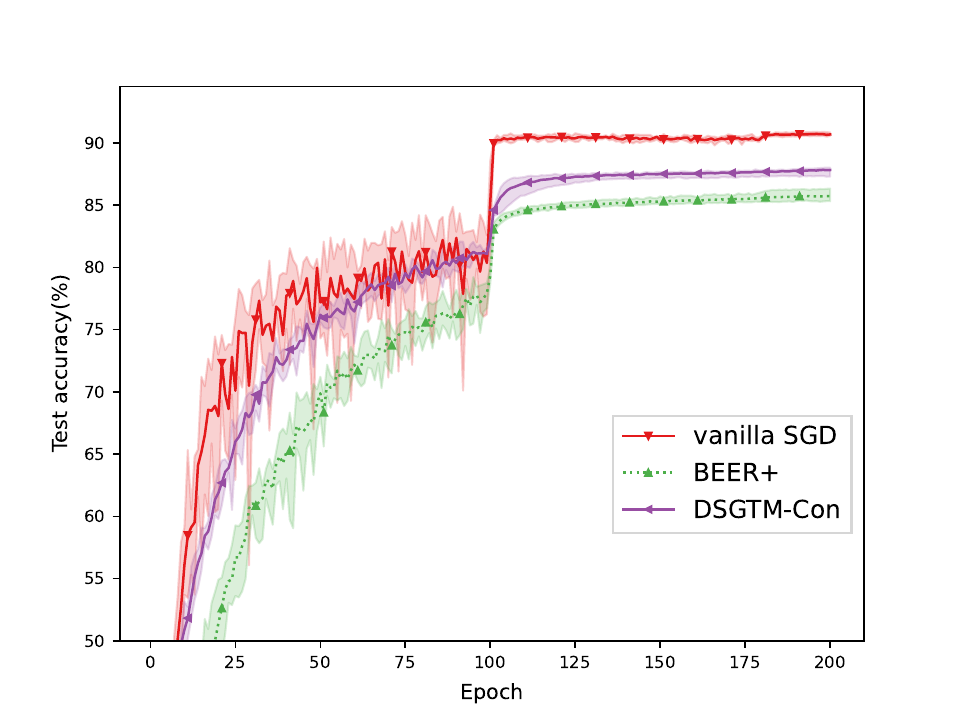}
\label{fig_8_1}}
\hfil
\subfigure[Training loss w.r.t. Epochs]{
\includegraphics[width=3.6cm, height=2.6cm]{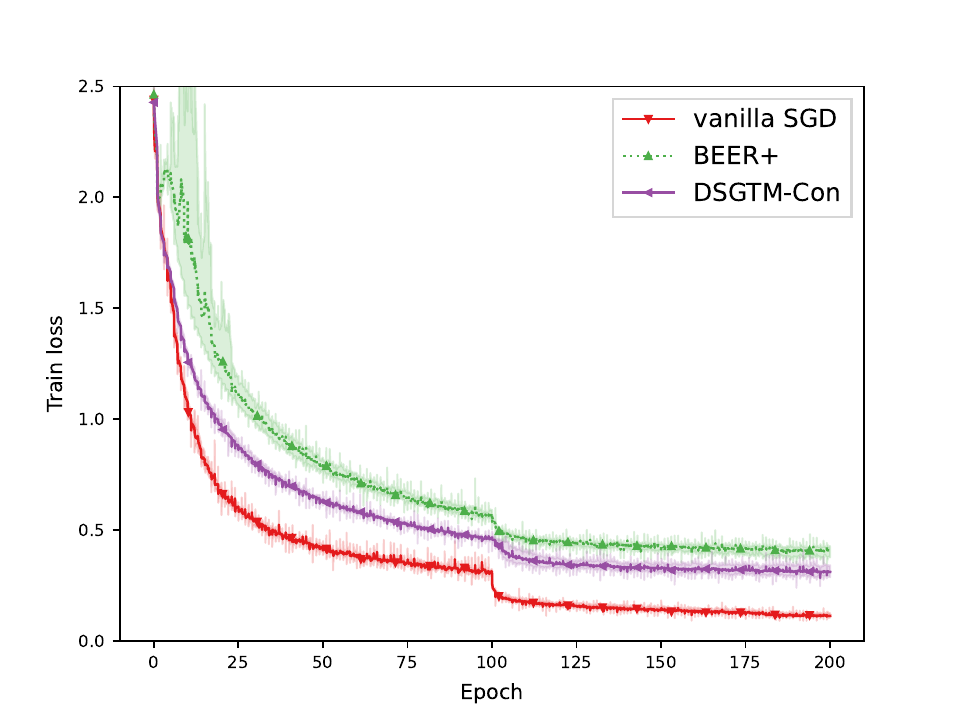}
\label{fig_8_2}}
\hfil
\subfigure[Test accuracy w.r.t. Communication bits]{
\includegraphics[width=3.6cm, height=2.6cm]{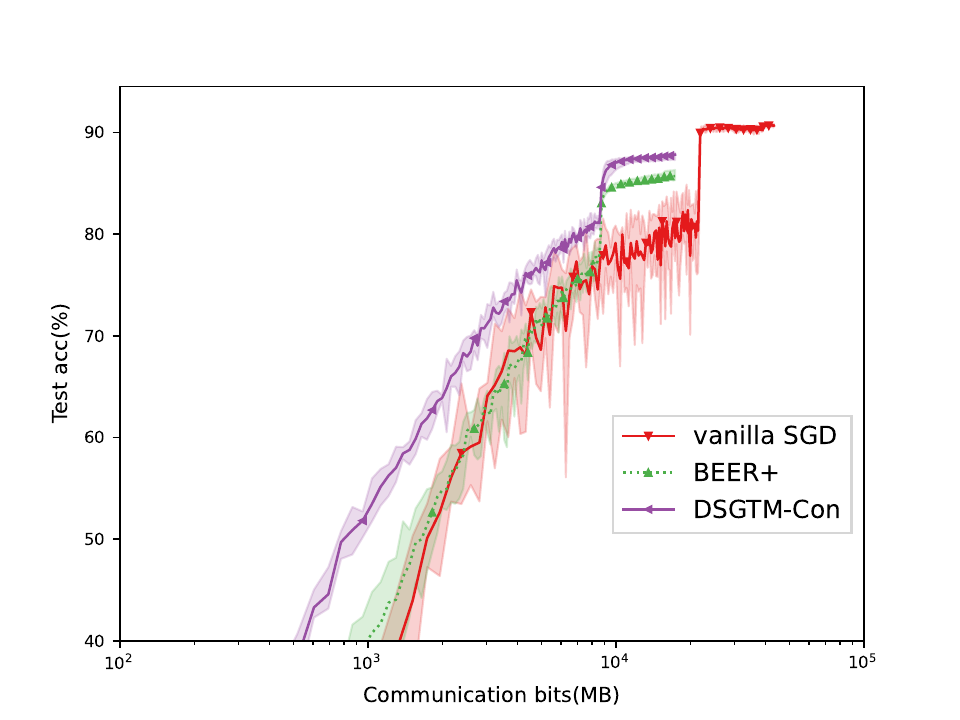}
\label{fig_8_3}}
\hfil
\subfigure[Training loss w.r.t. Communication bits]{
\includegraphics[width=3.6cm, height=2.6cm]{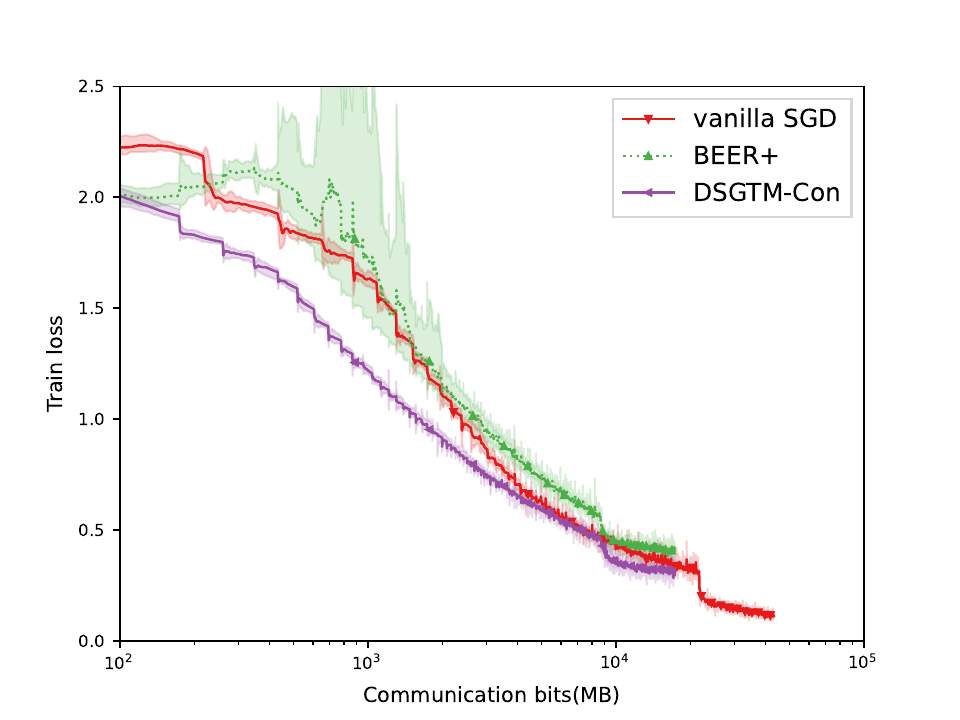}
\label{fig_8_4}}
\hfil

\caption{Numerical results of Gradient-tracking-based methods with contractive compression ($\mathrm{Random}$-$10\%$ operator).}
\label{fig:random10_group2}
\end{figure}

From Figures \ref{fig:group1}, \ref{fig:quant8_group1}, and \ref{fig:random10_group1} (a)-(b), it can be observed that SignDSGD-Con and DSM-Con outperform the extended versions of CHOCO-SGD in terms of both training loss and test accuracy, and they are closer to the baseline method, vanilla DSGD. From Figures \ref{fig:group1}, \ref{fig:quant8_group1} and \ref{fig:random10_group1} (c)-(d), we can infer that, with a smaller communication budget, CHOCO-SGD+ achieves higher test accuracy, followed by the two newly developed methods. However, with a slightly larger communication budget, SignDSGD-Con and DSM-Con perform better. Furthermore, from Figures \ref{fig:group2}, \ref{fig:quant8_group2} and \ref{fig:random10_group2}, we conclude that DSGTM-Con matches the uncompressed baseline in both test accuracy and training loss, while consistently reducing communication costs by $1/3$, demonstrating consistently better overall performance than BEER+ across all evaluated metrics.

Figure \ref{fig:group3} shows that, among unbiased compression methods, SignDSGD-Unb and DSM-Unb achieve the highest test accuracy, while QSDGD+ and DSGTM-Unb exhibit slightly inferior performance. Notably, QSDGD+ attains competitive accuracy and lower training loss with single-node communication costs below $10^4$ MB. However, as the communication budget increases, DSM-Unb emerges as the top-performing method.

Based on the comprehensive numerical comparisons, we demonstrate that our framework exhibits significant potential for developing methods that balance communication efficiency and accuracy.

\section{Concluding Remarks}
In this paper, we design a unified framework \eqref{Eq_Framework} for developing decentralized stochastic subgradient-type methods with communication compression. By configuring different compression errors and selecting diagonal matrix $\Sigma$ as either $\Diag({\bm W})$ or ${\bm I}_d$, our framework can adapt to two variants suitable for unbiased or contractive communication compression, respectively. Moreover, \eqref{Eq_Framework} employs a family of set-valued mappings $\{\Phi_{i}\}_{i\in[d]}$ to characterize the general relation between update direction and iterates, which unifies a broad class of acceleration techniques, such as Heavy-ball momentum, Nesterov momentum, and gradient tracking.

By introducing slowly diminishing sequences $\{\theta_{k}\}$ (and  $\{\gamma_{k}\}$) to regulate modified local average and the update of local copies, we establish the consensus properties and global asymptotic convergence for our framework \eqref{Eq_Framework}.  In particular, we demonstrate that a wide range of decentralized stochastic subgradient-type methods with communication compression fit into our proposed framework, including the nonsmooth extensions of QDGD, CHOCO-SGD, BEER, and others. Therefore, our theoretical results provide, for the first time, global convergence guarantees for these compression-based decentralized SGD-type methods in nonsmooth nonconvex optimization. Further, we develop several compression-based methods based on \eqref{Eq_Framework}, such as DSM-Unb(Con), SignDSGD-Unb(Con), DSGTM-Con, demonstrating the significant potential of \eqref{Eq_Framework} in developing practical methods. Preliminary numerical results validate our theoretical results and showcase the high efficiency of these methods within the framework.

Finally, several topics warrant future study to further understand the behavior of decentralized stochastic subgradient-type methods with communication compression in nonsmooth optimization. For instance, it would be valuable to explore whether, under random initialization, the iterates of our framework \eqref{Eq_Framework} almost surely converge to a Clarke-critical point, i.e., excluding potential spurious critical points within the set of $D_f$. Besides, another interesting direction is to investigate the development of nonsmooth Adam-like compression-based algorithms based on our framework, aiming to achieve a better trade-off between communication compression and accuracy.

\bibliographystyle{plain}
\bibliography{ref}

\section*{Appendix}

\subsection{Proof of Lemma \ref{lem02}}\label{sec:appendix-01}


\begin{proof}
A straightforward calculation shows that, for any $N\in \mathbb{N}^{*}$,
\begin{equation}\label{eq:prod1}
\begin{aligned}
&~ a\gamma_{k}+ \sum_{i=N}^{k} \prod_{j=i}^{k}(1-a\gamma_j)a\gamma_{i-1} + \prod_{j=N-1}^{k}(1-a\gamma_{j}) \\
= &~ a\gamma_{k} + \sum_{i=N+1}^{k}\prod_{j=i}^{k}(1-a\gamma_j)a\gamma_{i-1} +  \prod_{j=N}^{k}(1-a\gamma_{j})a\gamma_{N-1} + \prod_{j=N}^{k}(1-a\gamma_{j})(1-a\gamma_{N-1})\\
= &~ a\gamma_{k} + \sum_{i=N+1}^{k}\prod_{j=i}^{k}(1-a\gamma_j)a\gamma_{i-1} + \prod_{j=N}^{k}(1-a\gamma_{j}) \\
= &~ a\gamma_{k} + \sum_{i=k}^{k}\prod_{j=i}^{k}(1-a\gamma_j)a\gamma_{i-1} + \prod_{j=k-1}^{k}(1-a\gamma_{j}) \\
= &~ 1\\
\end{aligned}
\end{equation}
We also note that 
\begin{equation}\label{eq:prod2}
\sum_{i=1}^{N-1} \prod_{j=i}^{k}(1-a\gamma_j)a\gamma_{i-1} = \prod_{j=N-1}^{k}(1-a\gamma_{j}) - \prod_{j=0}^{k}(1-a\gamma_{j}).
\end{equation}
Since $\lim_{k\to +\infty} \frac{\theta_{k}}{\gamma_{k}}=0$, for any $\varepsilon>0$, there exists $N_0\in\mathbb{N}^*$, such that  $\frac{\theta_{k}}{\gamma_k}< \varepsilon$, for all $k \geq N_0$. Combining \eqref{eq:prod1} and \eqref{eq:prod2}, we obtain 
\begin{equation*}
\begin{aligned}
&~ \sum_{i=1}^k \left(\prod_{j=i}^{k} (1-a\gamma_j)\right) C_0 \theta_{i-1}\\
= &~ \frac{C_0}{a}\sum_{i=1}^{N_0-1} \prod_{j=i}^{k}(1-a\gamma_j)a\gamma_{i-1}\frac{\theta_{i-1}}{\gamma_{i-1}} + \frac{C_0}{a}\sum_{i=N_0}^{k} \prod_{j=i}^{k}(1-a\gamma_j)a\gamma_{i-1}\frac{\theta_{i-1}}{\gamma_{i-1}}  \\
\leq & \frac{C_0}{a}\max_{1\leq i-1 \leq N_{0}-1}\left\{\frac{\theta_{i-1}}{\gamma_{i-1}}\right\}\left[1- \prod_{j=0}^{N_0-2}(1-a\gamma_{j}) \right] \prod_{j=N_0-1}^{k}(1-a\gamma_{j})  +\frac{C_0}{a} [1-a\gamma_k - \prod_{j=N_0-1}^{k}(1-a\gamma_{j}) ]\varepsilon\\
\end{aligned}
\end{equation*}
Taking the limit as $k \to \infty$ and using the fact that \(\prod_{j=N_0-1}^{k} (1-a\gamma_j) \to 0\), we have
\begin{equation*}
\lim_{k\to \infty} \sum_{i=1}^k \left(\prod_{j=i}^{k} (1-a\gamma_j)\right) C_0 \theta_{i-1} = \frac{C_0}{a}\varepsilon
\end{equation*}
Since \(\varepsilon>0\) is arbitrary, the desired result follows.
\end{proof}

\subsection{Proof of Lemma \ref{lem01}}\label{sec:appendix-02}

\begin{proof}
Let $z_k = \gamma_k \upsilon_{k+1} + \sum_{i=1}^{k-1} \gamma_i \left( \prod_{j=i+1}^{k} (1 - a\gamma_j) \right) \upsilon_{i+1}$ and $z_0 = 0$. Define $\rho_{k,i} := \gamma_i \prod_{j=i+1}^k (1 - a\gamma_j)$ and $\rho_{k,k} := \gamma_k$. Then there exists $K > 0$ such that, for all $k \ge i \ge K$, we have $|\rho_{k,i}| \le \gamma_i$ and $a\gamma_i < 1/2$. Without loss of generality, we assume that $\rho_{k,i} \geq 0$, for all $k \geq i \geq K$. From the definition of $z_k$, it follows that
\[
z_k = \sum_{i=1}^k \rho_{k,i} \upsilon_{i+1}, \quad k \geq 1.
\]

Since the martingale difference sequence $\{\upsilon_{k}\}$ is uniformly bounded, it is sub-Gaussian. Thus, there exists a constant $M > 0$ such that, for all $k \ge 0$ and all $w \in \mathbb{R}^n$, 
\[
\mathbb{E} \left[ \exp \left( \langle w, \upsilon_{k+1} \rangle \right) \big| \mathcal{F}_k \right] \leq \exp \left( \frac{M}{2} \|w\|^2 \right).
\]

Therefore, for any $s > K, T > 0, w \in \mathbb{R}^n$ and $C > 0$, let
\[
Z_{i+1} := \exp \left\{  \langle C w, \sum_{k=s}^{i} \rho_{\Lambda_{\gamma}(\lambda_{\gamma}(s)+T),k}   \upsilon_{k+1} \rangle 
- \frac{MC^2}{2} \sum_{k=s}^{i} \rho_{\Lambda_{\gamma}(\lambda_{\gamma}(s)+T),k} ^2 \|w\|^2 \right\},
\]
where $\lambda_\gamma(0) := 0, \lambda_\gamma(i) := \sum_{k=1}^i \gamma_k$, and $\Lambda_\gamma(t) := \sup \{k \geq 0: t \geq \lambda_\gamma(k)\}$. Then, for any $i \geq s$, $\mathbb{E}[Z_{i+1}|\mathcal{F}_i] \leq Z_i$, so $\{Z_i\}_{i\geq s}$ forms a supermartingale. Hence for any $\delta > 0$, and $C > 0$, it holds that
\begin{equation}
\begin{aligned}
 & \mathbb{P} \left( \sup_{s \leq i \leq \Lambda_{\gamma}(\lambda_{\gamma}(s)+T)} \langle w, \sum_{k=s}^i \rho_{\Lambda_{\gamma}(\lambda_{\gamma}(s)+T),k} \upsilon_{k+1} \rangle > \delta \right)\\
= & \mathbb{P} \left( \sup_{s \leq i \leq \Lambda_{\gamma}(\lambda_{\gamma}(s)+T)}  \langle Cw, \sum_{k=s}^i \rho_{\Lambda_{\gamma}(\lambda_{\gamma}(s)+T),k} \upsilon_{k+1} \rangle > C\delta \right)\\
\leq & \mathbb{P} \left( \sup_{s \leq i \leq \Lambda_{\gamma}(\lambda_{\gamma}(s)+T)} Z_{i+1} > \exp \left( C\delta - \frac{MC^2}{2} \sum_{k=s}^{\Lambda_{\gamma}(\lambda_{\gamma}(s)+T)} \rho_{\Lambda_{\gamma}(\lambda_{\gamma}(s)+T),k}^2  \|w\|^2\right) \right)\\
\leq & \exp \left( \frac{MC^2}{2} \|w\|^2 \sum_{k=s}^{\Lambda_{\gamma}(\lambda_{\gamma}(s)+T)} \rho_{\Lambda_{\gamma}(\lambda_{\gamma}(s)+T),k}^2 - C\delta \right).\\
\end{aligned}
\end{equation}

Here, the last inequality is followed by Doob’s maximal inequality and the fact $\mathbb{E}[Z_{s+1}] \leq 1$. Since $C$ is arbitrary, we may set $C = \frac{\delta}{M \|w\|^2 \sum_{k=s}^{\Lambda_{\gamma}(\lambda_{\gamma}(s)+T)} \rho_{\Lambda_{\gamma}(\lambda_{\gamma}(s)+T),k}^2}$ to obtain that
\[
\mathbb{P} \left( \sup_{s \leq i \leq \Lambda_{\gamma}(\lambda_{\gamma}(s)+T)} \left\langle w, \sum_{k=s}^i \rho_{\Lambda_{\gamma}(\lambda_{\gamma}(s)+T),k} \upsilon_{k+1} \right\rangle > \delta \right) 
\leq \exp \left( -\frac{\delta^2}{2M \|w\|^2 \sum_{k=s}^{\Lambda_{\gamma}(\lambda_{\gamma}(s)+T)} \rho_{\Lambda_{\gamma}(\lambda_{\gamma}(s)+T),k}^2} \right).
\]

From the arbitrariness of $w$ and the fact that $\rho_{\Lambda_{\gamma}(\lambda_{\gamma}(s)+T),k} \leq \gamma_k$, we can further deduce that
\[
\mathbb{P} \left( \sup_{s \leq i \leq \Lambda_{\gamma}(\lambda_{\gamma}(s)+T)} \left\| \sum_{k=s}^i \rho_{\Lambda_{\gamma}(\lambda_{\gamma}(s)+T),k} \upsilon_{k+1} \right\| > \delta \right) 
\leq \exp \left( -\frac{\delta^2}{2M \sum_{k=s}^{\Lambda_{\gamma}(\lambda_{\gamma}(s)+T)} \gamma_k^2} \right) 
\leq \exp \left( -\frac{\delta^2}{2MT \gamma_{k'}} \right),
\]
which holds for some $k' \in [s, \Lambda_{\gamma}(\lambda_{\gamma}(s)+T)]$.

Let $n_0:= \inf \{j\in \mathbb{N}: \Lambda_\gamma(jT)>K\}$. For each $j \geq n_0$, there exists $k_j \in [\Lambda_\gamma(jT), \Lambda_\gamma((j+1)T)]$ such that
\[
\mathbb{P} \left( \sup_{\Lambda_\gamma(jT) \leq i \leq \Lambda_\gamma(jT+T)}  \left\| \sum_{k=\Lambda_\gamma(jT)}^i \rho_{\Lambda_{\gamma}(\lambda_{\gamma}(s)+T),k} \upsilon_{k+1} \right\| > \delta \right)  
\leq \exp \left( -\frac{\delta^2}{2MT \gamma_{k_j}} \right).
\]
Therefore, 
\begin{equation}\label{eq:borel}
\sum_{j=n_0}^{\infty} \mathbb{P} \left( \sup_{\Lambda_\gamma(jT) \leq i \leq \Lambda_\gamma(jT+T)}  \left\| \sum_{k=\Lambda_\gamma(jT)}^i \rho_{\Lambda_{\gamma}(\lambda_{\gamma}(s)+T),k} \upsilon_{k+1} \right\| > \delta \right)
\leq \sum_{j=n_0}^{\infty} \exp \left( -\frac{\delta^2}{2MT \gamma_{k_j}} \right)< +\infty.
\end{equation}
The last inequality folows from the fact that $\lim_{k \to +\infty} \gamma_k \log(k) = 0$. Let $\mathcal{E}_j$ denote the event
\[
\left\{ \sup_{\Lambda_\gamma(jT) \leq i \leq \Lambda_\gamma(jT+T)}  \left\| \sum_{k=\Lambda_\gamma(jT)}^i \rho_{\Lambda_{\gamma}(\lambda_{\gamma}(s)+T),k} \upsilon_{k+1} \right\| > \delta \right\}.
\]
By the Borel-Cantelli lemma and \eqref{eq:borel}, we can conclude that $\mathbb{P} \left(\bigcap_{j=1}^{+\infty} \bigcup_{l=j}^{+\infty} \mathcal{E}_l \right) = 0$, which indicates that
\begin{equation}\label{eq:limsupjtjt+1}
\lim_{j \to +\infty} \sup_{\Lambda_\gamma(jT) \leq i \leq \Lambda_\gamma(jT+T)} \left\| \sum_{k=\Lambda_\gamma(jT)}^i \rho_{\Lambda_\gamma(jT+T),k} \upsilon_{k+1} \right\| = 0. 
\end{equation}

For any $j \geq n_0$, 
\[
z_{\Lambda_\gamma(jT+T)} = \left( \prod_{k=\Lambda_\gamma(jT)}^{\Lambda_\gamma(jT+T)} (1-a\gamma_k) \right) z_{\Lambda_\gamma(jT)} + \sum_{k=\Lambda_\gamma(jT)}^{\Lambda_\gamma(jT+T)} \rho_{\Lambda_\gamma(jT+T),k} \upsilon_{k+1}.
\]
Since 
\begin{equation*}
\prod_{k=\Lambda_\gamma(jT)}^{\Lambda_\gamma(jT+T)} (1-a\gamma_k) \leq \prod_{k=\Lambda_\gamma(jT)}^{\Lambda_\gamma(jT+T)} \exp(-a\gamma_k) \leq \exp(-aT),
\end{equation*}
we attain that
\begin{equation}\label{eq:zLambda}
\left\| z_{\Lambda_\gamma(jT+T)} \right\| \leq \exp(-aT) \left\| z_{\Lambda_\gamma(jT)} \right\| 
+ \left\| \sum_{k=\Lambda_\gamma(jT)}^{\Lambda_\gamma(jT+T)} \rho_{\Lambda_\gamma(jT+T),k} \upsilon_{k+1} \right\|. 
\end{equation}
Together with \eqref{eq:limsupjtjt+1}, we can conclude that $\lim_{j \to +\infty} \|z_{\Lambda_\gamma(jT+T)}\| = 0$.

Finally, for any $i$ such that $\Lambda_\gamma(jT) < i \leq \Lambda_\gamma(jT + T)$, it holds that

\begin{equation*}
\begin{aligned}
\| z_{\Lambda_\gamma(jT+T)} \| & = \left\| 
\left( \prod_{k=i}^{\Lambda_\gamma(jT+T)} (1 - a\gamma_k) \right) z_i + \sum_{k=i}^{\Lambda_\gamma(jT+T)} \rho_{\Lambda_\gamma(jT+T),k} \upsilon_{k+1} 
\right\|\\
& \geq \exp(-2aT) \| z_i \| - \left\| \sum_{k=i}^{\Lambda_\gamma(jT+T)} \rho_{\Lambda_\gamma(jT+T),k} \upsilon_{k+1} \right\|\\
& \geq \exp(-2aT) \| z_i \| - \left\| \sum_{k=\Lambda_\gamma(jT)}^{\Lambda_\gamma(jT+T)} \rho_{\Lambda_\gamma(jT+T),k} \upsilon_{k+1} \right\| - \left\| \sum_{k=\Lambda_\gamma(jT)}^i \rho_{\Lambda_\gamma(jT+T),k} \upsilon_{k+1} \right\|\\
& \geq \exp(-2aT) \| z_i \| - 2 \sup_{\Lambda_\gamma(jT) \leq i \leq \Lambda_\gamma(jT + T)} \left\| \sum_{k=\Lambda_\gamma(jT)}^i \rho_{\Lambda_\gamma(s)+T,k} \upsilon_{k+1} \right\|,\\
\end{aligned}
\end{equation*}
where the first inequality follows from $1-a\gamma_k \geq \exp(-2a\gamma_k)$, when $k\geq K$, and from the fact that $\sum_{k=\Lambda_\gamma(jT)}^{\Lambda_\gamma(jT+T)}\gamma_k\leq T$. As a result, we have
\begin{equation}\label{eq:resultsupz}
\sup_{\Lambda_\gamma(jT) \leq i \leq \Lambda_\gamma(jT+T)} \| z_i \| \leq \exp(2aT) \left( \| z_{\Lambda_\gamma(jT+T)} \| + 2 \sup_{\Lambda_\gamma(jT) \leq i \leq \Lambda_\gamma(jT+T)} \left\| \sum_{k=\Lambda_\gamma(jT)}^i \rho_{\Lambda_\gamma(s)+T,k} \upsilon_{k+1} \right\| \right). 
\end{equation}
Combining \eqref{eq:limsupjtjt+1}, \eqref{eq:zLambda}, and \eqref{eq:resultsupz} together, we achieve that
\[
\limsup_{k \to +\infty} \| z_k \|\leq \lim_{j \to +\infty} \sup_{\Lambda_\gamma(jT) \leq i \leq \Lambda_\gamma(jT+T)} \| z_i \|  = 0.
\]
This completes the proof.
\end{proof}

\subsection*{Proof of Lemma \ref{claim3}}\label{appendix:claim3}

\begin{proof}

Due to the uniform boundedness of martingale difference sequence $\{\hat{\upsilon}_{k+1}\}$, it follows that $\{\hat{\upsilon}_{k+1}\}$ is sub-Gaussian. That is, there exists $M>0$ such that for any ${\bm y}\in \bb{R}^n$,
\begin{equation*}
		\bb{E}\left[ \exp\left( \inner{{\bm y}, \hat{\upsilon}_{k+1}} \right) | \ca{F}_k \right] \leq \exp\left( \frac{M}{2}\norm{{\bm y}}^2 \right).
	\end{equation*}
For any ${\bm y} \in \Rn$ and $C > 0$, define 
        \begin{equation*}
            Y_{i+1} := \exp\left[  \inner{C{\bm y}, \sum_{k = s}^i \theta_k \hat{\upsilon}_{k+1}} - \frac{MC^2}{2}\sum_{k = s}^i \theta_k^2 \norm{{\bm y}}^2  \right].
        \end{equation*} 
        Then, for any $i\geq 0$, we have $\bb{E}[Y_{i+1} | \ca{F}_i] \leq Y_{i}$. 
	Hence, for any $\delta > 0$ and any $C > 0$, it holds that 
	\begin{equation*}
		\begin{aligned}
			&\bb{P}\left( \sup_{s\leq i \leq \Lambda_{\eta}(\lambda_{\eta}(s) + T)} \inner{{\bm y}, \sum_{k = s}^i \theta_k \hat{\upsilon}_{k+1}} > \delta \right)={}\bb{P}\left( \sup_{s\leq i \leq \Lambda_{\eta}(\lambda_{\eta}(s) + T)} \inner{C{\bm y}, \sum_{k = s}^i \theta_k \hat{\upsilon}_{k+1}} > C\delta \right)\\
			\leq{}& \bb{P}\left( \sup_{s\leq i \leq \Lambda_{\eta}(\lambda_{\eta}(s) + T)} Y_i > \exp\left( C\delta - \frac{MC^2}{2} \sum_{k = s}^{\Lambda_{\eta}(\lambda_{\eta}(s) + T)} \theta_k^2 \norm{{\bm y}}^2 \right) \right)\\
			\leq{}& \exp\left( \left(\frac{M}{2}\norm{{\bm y}}^2 \sum_{k = s}^{\Lambda_{\eta}(\lambda_{\eta}(s) + T)} \theta_k^2\right)C^2   - C\delta  \right),
		\end{aligned}
	\end{equation*}
where the last inequality follows from Doob’s maximal inequality and the fact that $\bb{E}[Y_{s+1}] \leq 1$. Then,  by choosing $C = \frac{\delta}{M\norm{{\bm y}}^2 \sum_{k=s}^{\Lambda_\eta(\lambda_\eta(s)+T)} \theta_k^2}$, we obtain
	\begin{equation*}
		\bb{P}\left( \sup_{s\leq i \leq \Lambda_{\eta}(\lambda_{\eta}(s) + T)} \inner{{\bm y}, \sum_{k = s}^i \theta_k \hat{\upsilon}_{k+1}} > \delta \right) \leq \exp \left(\frac{-\delta^2}{2M\norm{{\bm y}}^2 \sum_{k = s}^{\Lambda_{\eta}(\lambda_{\eta}(s) + T)}\theta_k^2 }\right). 
	\end{equation*}
	By taking ${\bm y} = \frac{\sum_{k=s}^i \theta_k \hat{\upsilon}_{k+1}}{\norm{\sum_{k=s}^i \theta_k \hat{\upsilon}_{k+1}}}$, we deduce
	\begin{equation*}
		\bb{P}\left( \sup_{s\leq i \leq \Lambda_{\eta}(\lambda_{\eta}(s) + T)} \norm{\sum_{k = s}^i \theta_k \hat{\upsilon}_{k+1}} > \delta \right) \leq  \exp\left(\frac{-\delta^2}{2M \sum_{k = s}^{\Lambda_{\eta}(\lambda_{\eta}(s) + T)} \theta_k^2 }\right),
	\end{equation*}
    
We claim that $\sum_{k = s}^{\Lambda_{\eta}(\lambda_{\eta}(s) + T)} \theta_k^2\leq T\frac{\theta_{k'}^2}{\eta_{k'}}$
for some $k' \in [s, \Lambda_{\eta}(\lambda_{\eta}(s) + T)]$. Suppose for contradiction that it fails. Then, for all $k \in [s, \Lambda_\eta(\lambda_\eta(s)+T)]$, 
\begin{equation*}
\eta_{k}\sum_{k = s}^{\Lambda_{\eta}(\lambda_{\eta}(s) + T)} \theta_k^2 > T\theta_{k}^2.
\end{equation*}
Summing over $k$ gives
\begin{equation*}
\sum_{k = s}^{\Lambda_{\eta}(\lambda_{\eta}(s) + T)} \eta_{k}  \sum_{k = s}^{\Lambda_{\eta}(\lambda_{\eta}(s) + T)} \theta_k^2 > T \sum_{k = s}^{\Lambda_{\eta}(\lambda_{\eta}(s) + T)} \theta_{k}^2,
\end{equation*}
which contradicts the definition of $\lambda_\eta$ and $\Lambda_\eta$, since $T \geq \sum_{k=s}^{\Lambda_\eta(\lambda_\eta(s)+T)} \eta_k$. Hence the claim holds.

Consequently, there exists $k' \in [s, \Lambda_\eta(\lambda_\eta(s)+T)]$ such that
	\begin{equation*}
		\bb{P}\left( \sup_{s\leq i \leq \Lambda_{\eta}(\lambda_{\eta}(s) + T)} \norm{\sum_{k = s}^i \theta_k \hat{\upsilon}_{k+1}} > \delta \right) \leq \exp\left(\frac{-\delta^2}{2M T\frac{\theta_{k'}^2}{\eta_{k'}} }\right).
	\end{equation*}
For each $j \geq 0$, there exists $k_j\in [\Lambda(jT), \Lambda((j+1)T) ]$, such that
	\begin{equation*}
		\begin{aligned}
			&\sum_{j = 0}^{+\infty} \bb{P}\left(\sup_{\Lambda_{\eta}(jT)\leq i \leq \Lambda_{\eta}( jT+T)}\norm{ \sum_{k = s}^{i}\theta_k \hat{\upsilon}_{k+1}} \geq \delta \right) \\
			\leq{}& \sum_{j=0}^{+\infty} \exp\left( \frac{-\delta^2}{2MT \eta_{k_j}^{-1}\theta_{k_j}^2} \right) \leq \sum_{k=0}^{+\infty} 2 \exp\left( \frac{-\delta^2}{2MT\frac{\theta_k^2}{\eta_k}} \right) < +\infty. 
		\end{aligned}
	\end{equation*}
	Here the last inequality holds from the fact that $\lim_{k \to +\infty}\frac{\theta_k^2}{\eta_k} \log(k) = 0 $. According to Borel-Cantelli Theorem, we obtain that
	\begin{equation*}
		\lim_{j \to +\infty} \sup_{\Lambda_{\eta}(jT)\leq i \leq \Lambda_{\eta}( jT+T)}\norm{ \sum_{k = \Lambda_{\eta}(jT)}^{i}\theta_k \hat{\upsilon}_{k+1}} = 0. 
	\end{equation*}
	Finally, for any $jT \leq s\leq jT+T$, 
	\begin{equation*}
		\begin{aligned}
		    &\sup_{s\leq i \leq \Lambda_{\eta}( \lambda_{\eta}(s)+T)}\norm{ \sum_{k = \Lambda_{\eta}(jT)}^{i}\theta_k \hat{\upsilon}_{k+1}} \\
        \leq{}& 2\sup_{\Lambda_{\eta}(jT)\leq i \leq \Lambda_{\eta}( jT+T)}\norm{ \sum_{k = \Lambda_{\eta}(jT)}^{i}\theta_k \hat{\upsilon}_{k+1}} + \sup_{\Lambda_{\eta}((j+1)T)\leq i \leq \Lambda_{\eta}( (j+2)T)}\norm{ \sum_{k = \Lambda_{\eta}(jT + T)}^{i}\theta_k \hat{\upsilon}_{k+1}}. 
		\end{aligned}
	\end{equation*}
	Then we achieve that 
	\begin{equation*}
		\lim_{s \to +\infty} \sup_{s\leq i \leq \Lambda_{\eta}( \lambda_{\eta}(s)+T)}\norm{ \sum_{k = s}^{i}\theta_k \hat{\upsilon}_{k+1}} = 0,
	\end{equation*}
which completes the proof. 
\end{proof}

\end{document}